\documentclass[11pt]{article}
\usepackage{amsmath, amssymb, amsfonts, amsthm}

\usepackage{physics}
\usepackage{mathrsfs}
\usepackage{bbm} 
\usepackage{graphicx}
\usepackage{enumerate}         
\usepackage{dsfont}
\usepackage{enumitem}
\usepackage{color}
\usepackage{url} 

\usepackage[margin=1.25in]{geometry}

\usepackage[colorinlistoftodos]{todonotes}

\theoremstyle{plain}
\newtheorem{theorem}{Theorem}[section]
\newtheorem{assumption}{Assumption}[section]

\newtheorem{corollary}[theorem]{Corollary}

\newtheorem{lemma}[theorem]{Lemma}

\newtheorem{proposition}[theorem]{Proposition}

\newtheorem{definition}[theorem]{Definition}

\theoremstyle{remark}
\newtheorem{remark}[theorem]{Remark}
\newtheorem{example}[theorem]{Example}

\newcommand{\1}{\mathbbm{1}} 

\definecolor{bl}{rgb}{.0,.0,.999}
\definecolor{gr}{rgb}{.05,.5,.05}

\newcommand{\bbE}{\mathbb{E}}
\newcommand{\bbF}{\mathbb{F}}
\newcommand{\bbG}{\mathbb{G}}

\newcommand{\bbP}{\mathbb{P}}

\newcommand{\bbR}{\mathbb{R}}

\newcommand{\cA}{\mathcal{A}}

\newcommand{\cF}{\mathcal{F}}
\newcommand{\cG}{\mathcal{G}}

\newcommand{\cI}{\mathcal{I}}

\newcommand{\cL}{\mathcal{L}}

\newcommand{\cP}{\mathcal{P}}

\newcommand{\cT}{\mathcal{T}}

\newcommand{\cV}{\mathcal{V}}

\newcommand{\sC}{\mathscr{C}}
\newcommand{\sD}{\mathscr{D}}
\newcommand{\sU}{\mathscr{U}}

\newcommand{\bX}{\mathbf{X}}

\newcommand{\bW}{\mathbf{W}}
\newcommand{\bh}{\mathbf{h}}

\title{A Mean Field Game of Sequential Testing\thanks{This work is partially supported by the Natural Sciences and Engineering Research Council of Canada through an Alexander Graham Bell Canada Graduate Scholarship (CGSD3-535625-2019) and NSERC Discovery Grant (RGPIN-2020-06290). We would like to thank Tomoyuki Ichiba for his careful reading and helpful comments. We also thank Erik Ekstr\"om for sharing detailed notes related to his paper \cite{ekstrom2004properties}. The proof techniques in the notes were borrowed to prove continuity in the volatility for our setting.}
}

\author{Steven Campbell\thanks{Department of Statistics, Columbia University, New York, USA, \texttt{sc5314@columbia.edu}.} 
\and 
Yuchong Zhang\thanks{Department of Statistical Sciences, University of Toronto, Canada, \texttt{yuchong.zhang@utoronto.ca}. 
}
}

\date{\today}

\begin{document}

\maketitle

\begin{abstract}
We introduce a mean field game for a family of filtering problems related to the classic sequential testing of the drift of a Brownian motion. To the best of our knowledge this work presents the first treatment of mean field filtering games with stopping and an unobserved common noise in the literature. We show that the game is well-posed, characterize the solution, and establish the existence of an equilibrium under certain assumptions. We also perform numerical studies for several examples of interest.
\end{abstract}

{\small

\noindent \emph{AMS 2020 Subject Classification}
%
91A16;
91A55; %
60G35;
60G40

\tableofcontents

\section{Introduction}

A classic continuous-time optimal stopping problem is the sequential testing of the drift of a Brownian motion \cite{gapeev2004wiener,shiryaev1967two,shiryaev2007optimal}, where an agent observes a noisy signal process of an unknown binary variable $\theta$:
\begin{equation*}
	X_t=\alpha\theta t+W_t,
\end{equation*}
and tests sequentially the hypotheses $H_0: \theta=0$ and $H_1: \theta=1$ given the information gleaned from the signal process. The agent's response consists of the time $\tau$ to stop the test and make a decision, and the decision $\mathsf{d}\in\{0, 1\}$ itself, i.e.\ which hypothesis to accept.
%
%

In the Bayesian formulation, the agent has a prior probability $\mathbb{P}(\theta=1) = \pi$, and
searches for a minimizing pair $(\tau^*,\mathsf{d}^*)$ to the Bayes risk:
\begin{equation*}
	\bbE_\pi\left[c\tau+a_1\1_{\{\mathsf{d}=1,\theta=0\}}+a_2\1_{\{\mathsf{d}=0,\theta=1\}}\right],
\end{equation*}
where $c$, $a_1$ (resp.\ $a_2$) are positive constants representing the cost of observation and the cost of making a type I (resp.\ type II) error. It can be thought of as a hard-classification problem since a binary decision $\mathsf{d}$ about the state of nature must be made at $\tau$ even if the agent is not certain.

It is well-known that the solution to the classic problem depends on the agent's posterior probability $\Pi_t=\mathbb{P}(\theta=1|\mathcal{F}^X_t)$. This motivates us to consider a soft-classification approach that uses the modified Bayes risk
\begin{equation*}
	\mathbb{E}_\pi \left[c\tau+\mathscr{L}(\theta,\Pi_{\tau})\right],
\end{equation*}
where $\mathscr{L}:\{0,1\}\times[0,1]\to\mathbb{R}_+$ is a loss function satisfying $\mathscr{L}(j,\pi)=0$ if and only if $j=\pi$. 
A primary example we have in mind is the cross-entropy loss $\mathscr{L}(j, \pi)=-j\log(\pi)-(1-j)\log(1-\pi)$.

In our study, we embed this ``soft" sequential testing problem in a game setting with a continuum of agents and a finite horizon $T>0$. The infinite player limit allows us to consider mean field type interactions through population statistics. 
We take $\mu=(\mu^0, \mu^1)$ to be the distribution of stopped agents in time if $\theta =0, 1$.
In the simplest one-dimensional setting, a representative agent $i$ observes the private signal 
\begin{equation*}
	X^i_t=\int_0^t h(s,\theta, F^\theta_\mu(s))ds+W_t^i,
\end{equation*}
for some function $h$ that depends on both the state of nature $\theta$ and the fraction of agents $F^j_\mu(s)$ that have stopped by time $s$ given $\theta=j$.
In economic contexts where agents take resources with them, the signal may get weaker as more agents decide to exit the game. On the other hand, if there is an accumulation of common knowledge or reduced competition for resources, the signal may get stronger as agents leave. Both cases can be modeled by choosing a suitable $h_\mu$; see Example~\ref{ex:h}.

Each agent's risk minimization can be cast as an optimal stopping problem (parameterized by the population measure $\mu$) in terms of the private posterior probability $\Pi^i_t:=\mathbb{P}(\theta=1|\mathcal{F}^{X^i}_t)$:
\begin{equation*}
	\cV_i(\mu)=\inf_{\tau^i} \bbE_\pi\left[c \tau^i+g (\Pi^i_{\tau^i})\right],
\end{equation*}
where $g$ is a concave penalty derived from the loss function $\mathscr{L}$. 
We then search for a pair of measures $\mu=(\mu^{0}, \mu^{1})$ such that for the smallest optimal stopping time $\tau^{\ast}_\mu$ of $\cV_\cdot(\mu)$ we have a fixed point of the mapping
\begin{equation}\label{eqn:fixed.pt.map.intro}
	\mu\mapsto \left(\mathcal{L}(\tau^{\ast}_\mu |\theta=0), \mathcal{L}(\tau^{\ast}_\mu |\theta=1)\right).
\end{equation}
Any such fixed point is a mean-field equilibrium in the sense of Nash by the Exact Law of Large Numbers \cite{sun2006exact}.

Mean field games of optimal stopping (also called mean field games of timing) are generally more difficult to solve than mean field games of optimal control. An additional challenge in our setting is the incorporation of a common noise and agent learning, the latter of which requires the preservation of the information structure. 
Nonetheless, under suitable assumptions we are able to show that a strong solution exists. 

The proof follows the standard application of Schauder's fixed point theorem. However, establishing the continuity of the mapping \eqref{eqn:fixed.pt.map.intro} turns out to be a difficult task, and is a major technical contribution of the present paper. It requires a deep understanding of the representative agent problem, particularly the shape of the continuation region and the regularity of the free boundaries. We characterize the free boundaries as the ``maximal" solutions to a pair of integral equations, and establish a uniform boundary regularity in the input measures. The analysis relies on new probabilistic techniques introduced in \cite{de2019lipschitz} alongside a time and space transformation of the problem.

\subsection{Literature}

Our problem lies squarely at the intersection of sequential analysis and mean field games (MFGs). The field of sequential analysis covers a broad class of statistical problems where data is collected and analyzed sequentially over time.
It has a long and rich history, and the decades following Wald's pioneering work \cite{wald1945sequential} are marked by many important developments \cite{chernoff1959sequential,kiefer1957some,wald1948optimum,wald1950bayes}. Of particular relevance to the present study is \cite{dvoretzky1953sequential} where the theory of sequential decision making for stochastic processes in continuous time was introduced. Our motivating sequential testing problem
for the drift of a signal process \cite{shiryaev1967two} is studied in the finite horizon setting in \cite{gapeev2004wiener}, extended to a signal process whose drift has a state dependence in \cite{gapeev2011sequential}, and studied for a general (not necessarily binary) unknown state $\theta$ and arbitrary prior distribution in \cite{ekstrom2015bayesian}.

Due to our game's (infinite) population design, the interaction structure is very similar to what arises in the mean field game literature. In the initial works \cite{Lasry2006a,Lasry2006b,Lasry2007mean,huang2007large,huang2006large}, mean field games were used to study Nash equilibria with a continuum of agents. There is now a broad extant literature in this area (see e.g. \cite{bensoussan2013mean,carmona2018probabilistic,cousin2011mean}) and mean field game formulations have found many applications in economics that include population dynamics \cite{cousin2011mean}, production models \cite{cousin2011mean}, systemic risk \cite{carmona2015mean}, and models for renewable energy certificate \cite{campbell2021deep,shrivats2022mean} and electricity markets \cite{aid2021entry,bagagiolo2014mean,carmona2022mean,elie2021mean,feron2020price,gomes2021mean}.

Theoretically, there are two main branches to this theory corresponding to the use of analytical \cite{huang2007large,huang2006large,Lasry2006a,Lasry2006b,Lasry2007mean} or probabilistic \cite{carmona2013probabilistic,carmona2014master,carmona2015forward} solution techniques. The former approach involves solving a coupled system of partial differential equations (PDEs) while the latter obtains a coupled system of forward-backward stochastic differential equations (FBSDEs).
Additional elements have been incorporated into the theory to accommodate a growing set of use cases. For instance, the literature has studied games with common noise \cite{ahuja2016wellposedness,carmona2016mean,carmona2015mean,firoozi2018mean}, and constructions that involve partially observed systems and tools from filtering theory \cite{casgrain2020mean,firoozi2018mean,csen2016mean,sen2019mean}. Both of these directions are relevant to our proposed study since agents will need to make an inference about the {\it unobserved} and {\it common} state of nature $\theta$.


Spurred on by the recent developments in the original works \cite{carmona2017mean,nutz2018mean}, mean field games of optimal stopping has been an area of increasing interest in the past few years. Both of these papers take a probabilistic approach and the mean field interaction occurs through the proportion of players that have already stopped. One remarkable feature of \cite{nutz2018mean} is the tractability of the introduced model under a monotonicity assumption which allows for explicitly solvable examples to be investigated.  On the other hand, the authors of \cite{carmona2017mean}, motivated by a bank run model, study mean field games of timing in a more general setting. In their work, strong and weak notions of equilibria are defined and the existence of the latter is established under a strategic complementarity assumption using monotonicity-based proofs. By comparison, the existence of weak equilibria is established using continuity-based arguments.

An alternative analytical approach to timing games is provided in \cite{bertucci2018optimal},  where a system of forward–backward obstacle problems associated to the mean field game is analyzed. In contrast with \cite{carmona2017mean,nutz2018mean}, the proposed agent interaction is through the density of the states of the players who have yet to stop. This type of interaction is similar to the one adopted in the emerging literature on the Linear Programming approach to mean field games with stopping \cite{dumitrescu2020control,dumitrescu2021control}. An interesting and encouraging recent extension of this approach to problems with common noise and partial observations in a discrete setting is given in \cite{dumitrescu2022energy}. Another direction, presented in \cite{talbi2023dynamic}, studies an optimal stopping problem for McKean–Vlasov diffusions where the interaction arises through the law of the stopped process.

A primary goal of our current work is to combine several of the previously discussed features into a tractable family of mean field game problems with optimal stopping. By tractable we mean that the existence of equilibria can be established and that they can be characterized in detail, even if exactly solvable examples like in \cite{nutz2018mean,nutz2022mean} are not obtained. The main elements we would like to incorporate are an unobserved common noise, agent learning, and a population interaction that can influence the information structure. To the best of our knowledge these features have not been incorporated together into a single problem in the existing literature. Sequential testing style games are a promising candidate for these objectives due to the natural occurrence and interpretation of these features, and the solvability in the classic setting without interaction.

The rest of this paper is organized as follows. In Section \ref{sec:setup}, we formulate the sequential testing game and state our main result on the existence of a mean-field equilibrium. Section \ref{sec:single.agent} provides an in-depth analysis of the representative agent problem, the solution of which is then used to complete the fixed point analysis in Section~\ref{sec:main.proof.existence}. In Section~\ref{sec:classic}, we extend our results to the special case of the classic hard classification loss function for ``preemption games". Finally, Section \ref{sec:numerical} showcases several numerical examples: we explore preemption games in both the soft-classification and classic frameworks, and illustrate war of attrition games exclusively within the soft-classification setting. We draw our conclusions in Section \ref{sec:mfg.conclusion}.
Auxiliary results and lengthy technical proofs are delegated to Appendices.

\subsection{Notation}
We adopt the following notational conventions throughout this paper, where $D$ is an arbitrary Polish space. 

\begin{itemize}[noitemsep,leftmargin=*]
	\item $\cP(D)$: The space of probability measures on $D$ equipped with the topology of weak convergence. 
	\item $C_b(D)$: The space of continuous bounded functions on $D$.
	\item $C^k(D)$: The space of $k$-times continuously differentiable functions on $D$.
	\item Lip$(D)$: The space of Lipschitz functions on $D$.
	\item $\norm{\cdot}$: The standard Euclidean norm.
	\item $\norm{\cdot}_\infty$: The supremum norm given by $\|\mathbf{f}\|_\infty:=\max_{i\in d}\sup_{x\in D}|f_i(x)|$ for $d$ dimensional vector-valued functions defined on $D$.
	\item $(\Omega,\mathcal{F}, \mathbb{P})$: The (complete) probability space that corresponds to the sample space of our problem.
	\item $\mathbb{F}^X=(\cF_t^X)_t$: The filtration generated by a stochastic process $X$ augmented by the $\mathbb{P}$-null sets in $\mathcal{F}$.
	\item $\Rightarrow$: The weak convergence of measures.
\end{itemize}
In addition, for a stochastic process $X=(X_t)_{t\geq0}$, if $(t,X)$ is Markovian we denote the expectation operator conditioned on $(t,X_t)=(t,x)$ by $\mathbb{E}_{t,x}[\cdot]$. When $X$ is defined to start at $t=0$ we may drop the time subscript. Where appropriate we also write $X^{t,x}=(X^{t,x}_s)_{s\geq t}$ to denote the stochastic process $X=(X_t)_{t\geq0}$ started at $X_t=x$. When it is unambiguous, we similarly drop both subscripts on $\mathbb{E}_{t,x}[\cdot]$ in favor of writing $X^{t,x}$ under the expectation.

\section{Sequential Testing Game}\label{sec:setup}


Let $(I, \cI, \lambda)$ be an atomless probability space, to be used as the agent space, and 
let $(\Omega, \cF, \bbP)$ be another complete probability space, to be used as the sample space. We will work with a rich Fubini extension 
of the product space $(I\times \Omega, \cI\otimes\cF, \lambda \otimes \bbP)$ which supports a family of {essentially pairwise independent} $d$-dimensional Brownian motions $\{\bW^i\}_{i\in I} = \{(\bW^i_t)_{t\in [0, T]}\}_{i\in I}$ and a binary random variable $\theta\in \{0, 1\}$ that is common to all agents, i.e.\ $\theta(i, \omega)= \theta(\omega)$, and independent of $\bW^i$ for all $i \in I$. Moreover, the Conditional Exact Law of Large Numbers holds in this space; see Appendix \ref{sec:CELLN} for details. It is implicitly understood that any family of random variables or stochastic processes indexed by $i\in I$ is measurable in the extended product space.

Let $(\bbG^i)_{i\in I}$ be a family of filtrations to be specified later, representing the information available to each agent. The control of Agent $i$ is modeled as a $\bbG^i$-stopping time $\tau^i$. The following lemma is a direct consequence of the Conditional Exact Law of Large Numbers (Proposition~\ref{prop:CELLN}).
It describes the population statistics that emerges from essentially pairwise conditionally independent controls: the distribution of stopped agents at different points in time is a random measure, say $\mathfrak{m}$, from $\{0,1\}$ to $\cP([0,T])$ where the randomness is completely determined by the common noise $\theta$. 


\begin{lemma}\label{lemma:ELLN}
 If the family of stopping times $(\tau^i)_{i\in I}$ is \emph{essentially pairwise conditionally independent} given $\theta$, then
\[\lambda\{i: \tau^i\le t\} = \int_\cI \bbP(\tau^i \le t| \theta)\lambda(di).\]
That is, the fraction of agents that have stopped by time $t$ equals the conditional probability that a randomly chosen representative agent will stop by time $t$.
\end{lemma}

In the sequel, we identify the random measure $\mathfrak{m}$ with a pair of deterministic measures $\mu=(\mu^0, \mu^1) \in \cP([0, T])^2$ such that $\mathfrak{m}=\mu^\theta$.

\subsection{Signal process}

Let $\mu = (\mu^0, \mu^1)\in \cP([0, T])^2$ be the distribution of stopped agents in time corresponding to the two states of $\theta$. We introduce the regularization
\begin{equation}\label{eq:F.def}
F^j_{\mu}(t) := \int_\bbR \mu^j[0, s] \varphi(t-s) ds,
\end{equation}
where $\varphi: \bbR \rightarrow \bbR$ is a positive mollifier, i.e.\ a smooth and compactly supported function which integrates to one. It is easy to see that $F^j_\mu(\cdot)$ is increasing\footnote{Throughout we interpret ``increasing" and ``decreasing" in the non-strict sense.} and $[0,1]$-valued; it is an approximation to the true fraction of agents that have stopped. We may choose $\varphi$ to be asymmetric about zero so that it puts mass in the recent past. Then, it can be seen as introducing a delayed interaction.

Each agent (say Agent $i$) observes a private signal $\bX^i$, the strength of which depends on the collective stopping behavior of other agents:
\begin{equation}\label{eqn:private.signal0}
\bX^i_t = \int_0^t \bh^i(s, \theta, F^\theta_\mu(s)) ds + \Sigma^i \bW^i_t.
\end{equation}
Here $\Sigma^i$ is an invertible $d\times d$ matrix and $\mathbf{h}^i: \bbR_+\times\{0,1\}\times[0,1]\rightarrow \bbR^d$ 
is a bounded (jointly) continuous function such that $(t, \rho)\rightarrow \bh^i(t, j, \rho)$ is continuously differentiable for each $j$.
The observation filtration $\bbG^i$ in this case is the one generated by $\bX^i$ and augmented by the $\bbP$-null sets in $\cF$;  we write it as $\bbF^{\bX^i}=(\cF^{\bX^i}_t)_{t\in [0, T]}$. 
Using the boundedness of $\bh^i$, it can be shown that $\bbF^{\bX^i}$ is right-continuous; see \cite[Theorem 2.35]{bain2008fundamentals}. The set of $[0, T]$-valued $\bbF^{\bX^i}$-stopping times will be denoted by $\cT^{\bX^i}$.

Note that the information generated by $\bX^i$ and $(\Sigma^i)^{-1}\bX^i$ are the same. By replacing $\bX^i$ with $(\Sigma^i)^{-1}\bX^i$ and $\bh^i$ with $(\Sigma^i)^{-1}\bh^i$ if needed, we may further assume that $\Sigma^i = \mathrm{Id}$ (the identity matrix) in the rest of the paper without loss of generality. 
To simplify notation, we will frequently write $\bh_\mu^i(t,j):=\bh^i(t,j,F^j_\mu(t))$ and $\Delta \bh^i_\mu(t):=\bh^i_\mu(t,1)-\bh^i_\mu(t,0)$. 
To avoid a degenerate formulation, we also assume that there is always some information content in the signal process for all agents. That is, we enforce
\begin{equation}\label{eqn:non.degenerate.signal}
\|\Delta\bh^i_\mu(t)\| \geq \mathfrak{h}_i \quad \forall (t, \mu)\in \bbR_+\times \cP([0,T])^2,\; \forall i\in\mathcal{I},
\end{equation}
where $\mathfrak{h}_i>0$ is a constant.
%
%
An example of $\bh^i$ that fits into our framework is given below.

\begin{example}\label{ex:h} Let $d=1$, $\lambda_0^i>0$, $\lambda_1^i > -\lambda_0^i$, and 
\[\bh^i(t, j, \rho)= (j - 1/2)(\lambda_0^i + \lambda_1^i \rho).\]
In this setting $\bh^i$ is positive when $j=1$ and negative when $j=0$, and 
\[\norm{\Delta \bh^i_\mu(t)}= \lambda_0^i+\lambda^i_1 (F^1_\mu(t) + F^0_\mu(t))\ge \lambda_0^i \wedge (\lambda_0^i+\lambda_1^i).\] 
One can interpret $\lambda^i_0$ as a baseline signal and $\lambda^1_i$ as the strength of interaction. If $\lambda^i_1>0$ (resp.\ $\lambda^i_1<0$), the signal strength increases (resp.\ decreases) as more agents make a decision and exit the game. In other words, the case when $\lambda^i_1>0$ for all $i\in I$ represents a ``\emph{war of attrition}" where it is beneficial to stay in the game longer than other agents; 
the case when $\lambda^i_1<0$ for all $i\in I$ is analogous to a ``\textit{preemption game}" where agents benefit from exiting before others.
\end{example}


\subsection{Bayes risk}

The objective of each agent is to make an inference as quickly as possible while minimizing the decision error.
Specifically, Agent $i$ solves the optimal stopping problem 
\[ \inf_{\tau^i\in \cT^{\bX^i}}\mathbb{E}_\pi \left[c_i\tau^i+\mathscr{L}_i(\theta,\Pi^i_{\tau^i})\right],\]
where $c_i>0$ is the cost of observation, $\pi = \bbP(\theta=1)$ is the prior probability, 
$\Pi^i_t:=\mathbb{P}(\theta=1|\mathcal{F}^{\mathbf{X}^i}_t)$ is the posterior probability, and $\mathscr{L}_i:\{0,1\}\times[0,1]\to\mathbb{R}_+$ is a loss function satisfying $\mathscr{L}_i(j,\pi)=0$ if and only if $j=\pi$. By iterated conditioning, we can remove the unknown variable $\theta$ in the objective and write
\begin{align*}
\mathbb{E}_\pi \left[\mathscr{L}_i(\theta,\Pi^i_{\tau^i})\right] 
&= \mathbb{E}_\pi \left[ \bbE _\pi\left[\mathscr{L}_i(0,\Pi^i_{\tau^i})\1_{\theta = 0}+ \mathscr{L}_i(1,\Pi^i_{\tau^i})\1_{\theta = 1}\big|\cF^{\bX^i}_{\tau^i}\right]\right] \\
&=\mathbb{E}_\pi[g_i(\Pi^i_{\tau^i})],
\end{align*}
where
\[g_i(\pi):=\pi\mathscr{L}_i(1,\pi)+(1-\pi)\mathscr{L}_i(0,\pi).\]
It is easy to see that $g_i(0)=g_i(1)=0$, i.e.\ certainty about the state of nature incurs no penalty. We take $g_i:[0, 1]\rightarrow \bbR$ as our starting point and refer to $\mathbb{E}_\pi[c_i\tau^i+g_i(\Pi^i_{\tau^i})]$ as the Bayes risk of $\tau^i$.
If there are multiple minimizers of the Bayes risk, we break the tie by assuming that the agent favors early stopping.

Define an operator $\cA$ by
\begin{equation}\label{eq:operatorA}
(\cA f)(\pi) :=\frac{1}{2}\pi^2(1-\pi)^2 f''(\pi), \quad f\in C^2(0,1).
\end{equation}
The following technical assumption on $g_i$ will be enforced for all $i\in I$ throughout Section \ref{sec:single.agent} and \ref{sec:main.proof.existence} of the paper.


\renewcommand{\theassumption}{G} 
\begin{assumption}\label{ass:g}\
	\begin{itemize}
		\itemsep = 0em
		\item[(G1)] $g_i\in C^3(0,1)$ is symmetric about $1/2$, concave, and  $g_i(0)=g_i(1)=0$.
		\item[(G2)] $\mathcal{A}g_i$ satisfies
		$\partial_\pi (\cA g_i)<0$ on $(0,1/2)$, and $\partial_\pi (\cA g_i) >0$ on $(1/2,1)$.
		\item[(G3)] $\cA g_i(1/2)< -c_i/\mathfrak{h}_i^2$. 
	\end{itemize}
\end{assumption}
\renewcommand{\theassumption}{\arabic{section}.\arabic{assumption}} 

\begin{remark}\label{rmk:consequences.of.assumptions}
	There are a few important consequences of the above assumptions that we highlight here for future reference. First, (G1) and (G2) imply that $\cA g_i$ is a continuous, bounded and unimodal function on $(0,1)$. Since $\cA g_i$ is a univariate function, this implies that it is quasiconvex\footnote{A real-valued function defined on a convex set is said to be quasiconvex if all sublevel sets are convex.}. It is also possible to check that (G1) and (G2) imply that the composition of $g_i$ with the sigmoid function $S(x)=1/(1+e^{-x})$ is Lipschitz. 
	Indeed, differentiating $g_i\circ S$ yields
	\[\left(g_i\circ S(x)\right)'=S'(x)g_i'\circ S(x) =S(x)(1-S(x))g_i'\circ S(x),\]
	where we have used that $S'(x)=S(x)(1-S(x))$. Since $S(x)\in (0,1)$ for all $x\in\mathbb{R}$ it suffices to bound $\pi(1-\pi)g_i'(\pi)$ for $\pi\in(0,1)$. To this end, observe under (G1) $g_i$ is concave, so $\mathcal{A}g_i\leq 0$. By continuity and (G2) there is a $\beta>0$ such that: \[\inf_{\pi\in(0,1)}\mathcal{A}g_i(\pi)=\mathcal{A}g_i(1/2)=-\beta>-\infty.\] Rearranging says $|g_i''(\pi)|\leq 2\beta/\pi^2(1-\pi)^2$ for all $\pi\in(0,1)$. Using $1/2$ as a reference point we have that for any $\pi\in [1/2,1)$:
	\begin{equation*}
		|g_i'(\pi)|\leq |g_i'(1/2)|+\int_{\frac{1}{2}}^\pi |g_i''(u)|du=|g_i'(1/2)|+4\beta\log(\frac{\pi}{1-\pi})+\beta\frac{4\pi-2}{\pi(1-\pi)},
	\end{equation*}
	and a similar bound holds for $\pi\in(0,1/2]$. It is then easy to verify that $|\pi(1-\pi)g'(\pi)|\leq M$ for some $M>0$, and so $g_i\circ S$ is $M$-Lipschitz.
\end{remark}

The concavity and endpoint constraints in Assumption (G1) are natural conditions motivated by the loss function formulation, while regularity is enforced to make certain stochastic analytic tools applicable in our analysis (see e.g. Appendix \ref{sec:time-change}). It is worth emphasizing that the assumed symmetry is \textit{not} necessary for most of the analysis to go through, and is solely used to simplify the proof of the regularity of the stopping boundaries (Proposition \ref{prop:bdy.unif.local.lip}). 
The interested reader may inspect how it arises in the proof if a relaxation of the assumptions on $g$ is desired. The explicit use of the point $1/2$ in Assumption (G2)-(G3) is similarly tied to the symmetry of $g$ and can likewise be modified if $g$ is asymmetric. These last two technical assumptions are imposed primarily to ensure that the stopping time solution for the representative agent problem is well-behaved.

It is to be expected that the analysis of strong mean field equilibria generally requires stronger regularity and additional structure. These assumptions are not overly restrictive and are satisfied by many popular loss functions, as the next two examples show. The index $i$ is omitted in the examples for simplicity.

\begin{example}[Cross-entropy loss]\label{ex:g-cross-entropy}
Suppose $\mathscr{L}(j,\pi)=-j\log(\pi)-(1-j)\log(1-\pi)$ which is the most commonly used loss function in machine learning for classification problems. The induced cost function is given by
 \[g(\pi)=-\pi\log(\pi)-(1-\pi)\log(1-\pi),\]
 where we extend the domain of $g$ to $[0, 1]$ by taking limits. We also find that
  \[(\cA g)(\pi) = -\frac{1}{2}\pi(1-\pi) \quad \text{and}\quad \partial_\pi (\cA g) = \pi - \frac{1}{2}.\]
Clearly, (G1) and (G2) hold; (G3) is satisfied if $\mathfrak{h} \ge \sqrt{8c}$. 
 \end{example}

\begin{example}[$L_1$ and $L_2$ losses]\label{ex:g-l1-l2}
For the $L_1$ loss $\mathscr{L}(j,\pi)=\left|j-\pi\right|$, the induced cost function is $g(\pi)=2\pi(1-\pi)$. For the $L_2$ loss $\mathscr{L}(j,\pi)=\left(j-\pi\right)^2$, the induced cost function is $g(\pi)=\pi(1-\pi)$. In both cases, $g$ is of the form $g(\pi)=\beta \pi(1-\pi)$ for some $\beta >0$. With
  \[(\cA g)(\pi) = -\beta \pi^2(1-\pi)^2 \quad \text{and}\quad \partial_\pi (\cA g) = -2\beta \pi (2\pi^2 -3\pi + 1),\]
it is straightforward to check that all assumptions are satisfied if $\mathfrak{h}\ge \sqrt{8c/\beta}$.
\end{example}

\begin{remark}
The cost function $g(\pi)=a_1\pi\wedge a_2(1-\pi)$ for the classical sequential testing problem contains a kink, and thus, does not satisfy the $C^3$ condition in Assumption~\ref{ass:g}. The kink poses technical difficulties to the proof of boundary regularity which is essential for the existence of a mean field equilibrium. Although our main analysis does not cover this setting, we can adapt the proof to handle a subclass of mean field interaction -- the case of ``preemption games"; see Section~\ref{sec:classic}.
\end{remark}

\subsection{Mean field equilibrium}
Having introduced the interaction (via the signal process) and the optimal stopping problem faced by each agent, we next define what it means to be a (strong) mean field equilibrium of the sequential testing game.


\begin{definition} \label{def:MFE}
Let $\pi\in(0,1)$. We say $\mu = (\mu^0, \mu^1) \in  \mathcal{P}([0,T])^2$ is a mean field equilibrium
of the sequential testing game if both of the following hold.
\begin{enumerate}
 \item (Optimality) For each $i\in I$, $\tau^{i, *}_\mu$ is the smallest optimal stopping time for the single-agent problem
\begin{equation}\label{eq:single-agent-problem}
\cV_i(\mu)=\inf_{\tau^i \in \cT^{\bX^i}} \bbE_\pi\left[c_i \tau^i+g_i (\Pi^i_{\tau^i})\right],
\end{equation}
where the signal process $\bX^i$ is given by \eqref{eqn:private.signal0} with drift $\bh^i_\mu(t, \theta)$, and $\Pi^i_t=\bbP(\theta=1| \cF^{\bX^i}_t)$ is the posterior probability.    
\item (Consistency) 
Let $\iota$ be the identification function on $(I, \cI, \lambda)$, representing a randomly chosen agent according to $\lambda$. 
Then $\mu$ is a fixed point of the mapping
\begin{equation}\label{eqn:fixed.pt.mu} \mu\mapsto \left(\cL(\tau^{\iota, \ast}_\mu |\theta=0), \cL(\tau^{\iota, \ast}_\mu |\theta=1)\right),
\end{equation}
i.e.
\begin{equation}\label{eq:agg_tau_over_i}
\mu^\theta[0, t] = \int_{i\in I} \bbP(\tau^{i, \ast}_\mu \le t | \theta)\lambda(di) \quad \forall t\in [0, T].
\end{equation}

\end{enumerate}
\end{definition}

\begin{remark}
The mean field equilibrium defined above is in the sense of Nash, i.e.\ $\tau^{i, \ast}$ is the best response to $(\tau^{j, \ast})_{j\in I\backslash\{i\}}$ for all $i\in I$. To see this, observe that the family of signal processes $(\bX^i)_{i\in I}$ is essentially pairwise conditionally independent given $\theta$.
Since each $\tau^{i, \ast}_\mu$ is an $\bbF^{\bX^i}$-stopping time, the same conditional independence holds for $(\tau^{i, \ast}_\mu)_{i\in I}$. By Lemma~\ref{lemma:ELLN} and the atomless property of the agent space, we see that when all but one agents use their respective $\tau_\mu^{i, \ast}$, the fraction of agents that have stopped by time $t$ is precisely the right hand side of \eqref{eq:agg_tau_over_i}. Definition~\ref{def:MFE} then says that each $\tau^{i, \ast}_\mu$ is the best response to $\mu$ which is the population statistics that emerges from $(\tau^{j, \ast})_{j\in I\backslash\{i\}}$.

\end{remark}

We are ready to state our main result whose proof we defer to Section \ref{sec:main.proof.existence}.

\begin{theorem}\label{thm:existence}
Under Assumption \ref{ass:g}, there exists a mean field equilibrium for the sequential testing game.
\end{theorem}

As in other games with optimal stopping, the involvement of the law of the stopping times makes this problem difficult to solve. A further complication in our formulation is that the agents must learn about a common noise which couples their posterior beliefs and forces us to preserve the information structure. As a result, alternative weak solution approaches are unlikely to be directly applicable.

That said, the overarching idea for establishing Theorem \ref{thm:existence} is straightforward and standard in the mean field game literature. Namely, we equip $\mathcal{P}([0,T])^2$ with the topology of weak convergence and  apply Schauder's fixed point theorem to the map in \eqref{eqn:fixed.pt.mu}. 
Demonstrating the compactness of the underlying space is trivial, however, due to our ``strong" notion of equilibrium, proving the continuity of the mapping requires a detailed and protracted analysis. 

Indeed, in the process we show that for a weakly converging sequence of measures, the associated value functions converge, the stopping boundaries converge, and the conditional laws of the optimal stopping times converge. Many of the results and technical difficulties of the following section are undertaken for the primary purpose of ultimately proving continuity in Section \ref{sec:main.proof.existence}. Crucial challenges include deriving a representation of the stopping boundaries as the ``maximal" solutions to a pair of integral equations, and establishing a uniform regularity of the boundaries in the input measures.


\section{Analysis of the Single-Agent Problem}\label{sec:single.agent}

In this section, we fix $\mu = (\mu^0, \mu^1) \in \mathcal{P}([0,T])^2$ and study the single-agent optimal stopping problem \eqref{eq:single-agent-problem}. The index $i\in I$ is suppressed with the understanding that the analysis holds for all agents. From Section \ref{sec:value_fun} to \ref{sec:free.bdy}, we assume that Assumption \ref{ass:g} is in force.

\subsection{Posterior probability and the log-likelihood ratio}\label{sec:post.prob.and.log.likelihood}

We begin by describing the evolution of the posterior probability $\Pi_t$.

\begin{lemma}\label{lemma:Pi}
The posterior probability $\Pi=(\Pi_t)_{t\in [0, T]}$ satisfies the SDE
\begin{equation}\label{eq:Pi}
d\Pi_t = \Pi_t(1-\Pi_t) \Delta \bh^\top_\mu(t) d\overline{\mathbf{W}}_t,
\end{equation}
where 
\begin{equation}\label{eq:innovation_process}
\overline{\mathbf{W}}_t:=\mathbf{X}_t-\int_0^t [\bh_\mu(s,0) + \Delta\bh_\mu(s)\Pi_s]ds
\end{equation}
is a $d$-dimensional $\mathbb{F}^{\mathbf{X}}$-Brownian Motion known as the innovation process.
\end{lemma}

\begin{proof}

For any bounded measurable function $\phi$, 
standard nonlinear filtering theory implies that $\Xi_t(\phi):=\bbE[\phi(t,\theta)|\cF^{\bX}_t ]$ satisfies the Kushner-Stratonovich equation
\begin{equation}\label{eq:Kushner}
d\Xi_t(\phi) = \left(\Xi_t(\phi \bh^\top_\mu)-\Xi_t(\phi)\Xi_t(\bh^\top_\mu)\right) d\overline{\bW}_t,
\end{equation}
where 
\begin{equation}\label{eq:innovation}
\overline{\mathbf{W}}_t=\mathbf{X}_t-\int_0^t\Xi_s(\mathbf{h}_\mu)ds
\end{equation}
is an $\mathbb{F}^{\mathbf{X}}$-Brownian Motion; see e.g.\ \cite[Theorem 3.35 and Proposition 2.30]{bain2008fundamentals}. Setting $\phi(t,\theta):= \theta =\1_{\{\theta = 1\}}$ leads to $\Xi_t(\phi)=\bbE[\1_{\{\theta=1\}}|\cF^\bX_t] =\Pi_t$ and
\[\Xi_t(\phi \bh_\mu)=\bbE[\1_{\{\theta = 1\}} \bh_\mu(t, 1) |\cF^\bX_t] = \bh_\mu(t,1)\Pi_t.\]
We also find that $\Xi_t(\bh_\mu)
=\bh_\mu(t,0) + \Delta\bh_\mu(t)\Pi_t$.
%
Substituting these expressions into \eqref{eq:Kushner} and \eqref{eq:innovation} yields the desired results.
\end{proof}

\begin{remark}\label{rmk:dist.pi.mu}
The distribution of the $\Pi$ depends on $\mu$ via the real-valued deterministic function $\norm{\Delta\bh_\mu}$, by L\'{e}vy's characterization of Brownian motion. We will refer to $\norm{\Delta\bh_\mu}$ as the ``volatility" of $\Pi$. By the non-degeneracy condition \eqref{eqn:non.degenerate.signal} and the boundedness of $\bh$, we have that 
\[\mathfrak{h}\le \norm{\Delta\bh_\mu} \le 2\sqrt{d}\norm{\bh}_\infty=:\mathfrak{H},\]
and these bounds are uniform in $\mu$. Hence $\Pi$ is the unique strong solution to \eqref{eq:Pi}. Since $\bh(\cdot,j, \cdot)$ is continuously differentiable and $F^j_\mu$ is the convolution with a mollifier, it is also straightforward to check that $\norm{\Delta\bh_\mu}\in C^1[0,T]\cap \text{Lip}([0,T])$, and the Lipschitz constant can be made independent of $\mu$.
\end{remark}

\begin{remark}\label{rmk:filtration}
The $\bbP$-augmentation of the natural filtration generated by $\overline{\mathbf{W}}$, denoted by $\bbF^{\overline{\mathbf{W}}}=(\cF_t^{\overline{\bW}})_{t\in [0,T]}$, coincides with $\bbF^\bX$. Indeed, it is clear from the definition \eqref{eq:innovation_process} that $\cF_t^{\overline{\mathbf{W}}}\subseteq \cF_t^\bX$. The reverse set inclusion follows from the fact that $(\bX, \Pi)$ is the unique strong solution of the system of SDEs
\[
\begin{cases}
d\bX_t = [\bh_\mu(t,0)+\Delta \bh_\mu(t)\Pi_t]dt + d\overline{\bW}_t, & \bX_0 = 0, \\
d\Pi_t = \Pi_t(1-\Pi_t) \Delta \bh_\mu^\top(t)d\overline{\bW}_t, & \Pi_0 = \pi. 
\end{cases}
\]

\end{remark}

As in Wald's sequential probability ratio test, it is often convenient to work with the log-likelihood ratio process
\begin{equation}\label{eq:defn-L}
L_t:= \log\left(\frac{\Pi_t}{1-\Pi_t}\right)=S^{-1}(\Pi_t),
\end{equation}
where $S^{-1}$ is the logit function, or the inverse of the sigmoid function $S$. It\^{o}'s formula implies
\begin{equation}\label{eq:L.dynamics}
d L_t=\frac{1}{2}\norm{\Delta \bh_\mu(t)}^2\frac{e^{L_t}-1}{e^{L_t}+1}dt+\Delta \bh^\top_\mu(t)d\overline{\mathbf{W}}_t.
\end{equation}
We see that the state-dependence in the diffusion coefficient of $\Pi_t$ is removed by such a spatial transformation. 
The next lemma shows that the log-likelihood ratio process is conditionally Gaussian given $\theta$; hence, it enjoys many nice properties.


\begin{lemma}\label{lem:properties.of.L}
\mbox{}
\begin{itemize}
\item[(i)] Conditioned on $\theta=1$,
\begin{equation}\label{eq:cond.dyn.L.1}
dL_t=\frac{1}{2}\norm{\Delta \bh_\mu(t)}^2dt+\Delta \bh^\top_\mu(t) d\bW_t, \ L_0=\log\left(\pi/(1-\pi)\right),
\end{equation}
and conditioned on $\theta=0$,
\begin{equation}\label{eq:cond.dyn.L.0}
dL_t=-\frac{1}{2}\norm{\Delta \bh_\mu(t)}^2dt+\Delta \bh^\top_\mu(t) d\bW_t, \ L_0=\log\left(\pi/(1-\pi)\right).
\end{equation}
\item[(ii)] For any $L_0 \in \bbR$ and $t>0$, $L_t$ admits a density on $\bbR$. Moreover, $L_t$ almost surely diverges to $\pm\infty$ as $t\to\infty$ but does not explode to $\pm\infty$ in finite time.
\end{itemize}
\end{lemma}

%
	
\begin{proof}
Expanding $d\overline{\bW}_t$ in \eqref{eq:L.dynamics} using \eqref{eqn:private.signal0}, \eqref{eq:innovation_process} and \eqref{eq:defn-L} yields
\[d L_t=-\frac{1}{2}\norm{\Delta \bh_\mu(t)}^2 dt + \Delta \bh_\mu^\top(t)[\bh_\mu(t,\theta)-\bh_\mu(t,0)]dt+\Delta \bh^\top_\mu(t)d\bW_t.
\]
Since $\bW$ is independent of $\theta$, we get \eqref{eq:cond.dyn.L.1} and \eqref{eq:cond.dyn.L.0} by conditioning.

Because $\Delta \bh_\mu(t)$ is deterministic, $L_t$ is conditionally Gaussian with nonzero variance (see \eqref{eqn:non.degenerate.signal}). Hence, it has a density of Gaussian mixture type. Conditioned on $\theta=1$, we can write by \eqref{eq:cond.dyn.L.1} and a time change that
\[L_t=L_0+\frac{1}{2}\alpha(t)+\widetilde{W}_{\alpha(t)}\]
for some Brownian motion $\widetilde{W}$ and ``clock" $\alpha(t):=\int_0^t\norm{\Delta \bh_\mu(s)}^2ds$ satisfying $\alpha(t)\in (0,\infty)$ for all $t>0$ and $\lim_{t\rightarrow \infty}\alpha(t)=\infty$ (recall $\mathfrak{h}\leq \norm{\Delta \bh_\mu}\leq \mathfrak{H}$). Clearly, $L$ does not explode in finite time. By the Law of Large Numbers for Brownian motion,
\[L_t=\alpha(t)\left(\frac{L_0}{\alpha(t)}+\frac{1}{2}+\frac{\widetilde{W}_{\alpha(t)}}{\alpha(t)}\right)\xrightarrow{\mbox{a.s.}}\infty \ \ \ \mbox{as}\ \ t\rightarrow \infty.\]
The case when we condition on $\theta=0$ is analogous.
\end{proof}


The following corollary is an immediate consequence of the bijection $\Pi_t=S(L_t)$.
\begin{corollary}\label{cor:properties.of.Pi}
For any $\Pi_0 \in (0,1)$ and $t>0$, $\Pi_t$ admits a density on $(0,1)$. Moreover, $\Pi_t$ almost surely converges to 0 or 1 as $t\to\infty$ but does not exit $(0,1)$ in finite time.
\end{corollary}

\subsection{Value function and optimal stopping time}\label{sec:value_fun}

We see in Lemma~\ref{lemma:Pi} that $\Pi_t$ is a diffusion in the observation filtration $\bbF^\bX$ which coincides with $\bbF^{\overline{\mathbf{W}}}$ by Remark~\ref{rmk:filtration}. This allows us to use a Markovian approach to the optimal stopping problem \eqref{eq:single-agent-problem}. For $(t,\pi)\in [0,T]\times [0,1]$, define the value function
\begin{equation}\label{eqn:value.function}
	V(t,\pi)=V(t,\pi;\mu):=\inf_{\tau\in \cT_t^{\bX}}\mathbb{E}\left[c(\tau-t)+g(\Pi_\tau^{t,\pi})\right],
\end{equation}
where $\Pi^{t,\pi}=(\Pi_s^{t,\pi})_{s\geq t}$ is given by \eqref{eq:Pi} with the initial condition $\Pi^{t,\pi}_t=\pi$, and $\cT^\bX_t$ is the set of $[t,T]$-valued $\bbF^\bX$-stopping times. 
One can also reparametrize the value function in terms of $L$:
\begin{equation}\label{eqn:value.function.L}
	\tilde V(t,l):=V(t,S(l))=\inf_{\tau\in \cT_t^{\bX}}\mathbb{E}\left[c(\tau-t)+g\circ S(L_\tau^{t,l})\right].
\end{equation}
Since $g\circ S$ is Lipschitz continuous by Remark \ref{rmk:consequences.of.assumptions} 
and the coefficients of $L$ satisfy the usual Lipchitz and linear growth conditions, we know from standard optimal stopping theory (see e.g.\ \cite[Proposition 4.7 and Remark 4.1]{touzi2012optimal}) that $\tilde V(t, l)$, and hence $V(t, \pi)=\tilde V(t, S^{-1}(\pi))$, is (jointly) continuous. 

Define the stopping region
\begin{equation}\label{eqn:stop.region}
	\mathscr{D}=\{(t, \pi)\in [0, T] \times [0,1]: V(t, \pi) = g(\pi)\}
\end{equation}
and the continuation region $\mathscr{C} = \mathscr{D}^c$.
Then
\begin{equation}\label{eqn:smallest.stop.time}\tau_{\mathscr{D}} = \inf\{s\ge t: (s, \Pi^{t,\pi}_s)\in \mathscr{D}\}
\end{equation}
is the smallest optimal stopping time for the problem \eqref{eqn:value.function}; see \cite[Theorem 2.4 and Corollary 2.9]{peskir2006optimal}. The optimality of $\tau_\mathscr{D}$ implies that \eqref{eqn:value.function} can equivalently be posed over $\bbF^{\Pi}$-stopping times, and that $V$ depends on $\mu$ only via $\norm{\Delta \bh_\mu}$. With a slight abuse of notation, in the next lemma we write $V(t, \pi; \eta)$ for the value function when the volatility $\norm{\Delta \bh_\mu}$ of $\Pi$ is replaced by a general measurable function $\eta: [0,T]\rightarrow \bbR_+$.

\begin{proposition}\label{prop:value.func}
	\mbox{}
	\begin{enumerate}
		\item[(i)] (Concavity) $V(t,\pi)$ is concave in $\pi$ for each fixed $t$.
		\item[(ii)] (Monotonicity in volatility) 
		If $\eta_1(t) \ge \eta_2(t) \ge 0 \, \forall\, t$, then $V(t,\pi; \eta_1)\leq V(t,\pi; \eta_2)$ $\forall \,(t,\pi) \in [0,T]\times [0,1]$. Consequently, for $\underline{V}:=V(\,\cdot\,; \mathfrak{H})$ and $\overline{V}:=V(\,\cdot\,;\mathfrak{h})$, we have that $\underline{V}\leq V\leq \overline{V}$.
	\end{enumerate}
\end{proposition}
\begin{proof}
	See \cite[Proposition 4.1]{ekstrom2015bayesian} for (i) and an analogue of  \cite[Proposition 4.2]{ekstrom2015bayesian} for (ii). Both proofs are based on an approximation argument where the stopping time is restricted to take values in a discrete set, alongside with the dynamic programming principle and the fact that the price of a European option with a concave contract function is concave in the underlying martingale diffusion and non-increasing in the volatility.
\end{proof}

Let $\mathcal{L}:=\frac{1}{2}\|\Delta\bh_\mu(t)\|^2\pi^2(1-\pi)^2\partial_{\pi\pi}$
be the infinitesimal generator of $\Pi$. The following regularity result can also be obtained for $V$ by appealing to standard arguments (c.f. \cite[Theorem 2.7.7]{karatzas1998methods}). While \cite[Theorem 2.7.7]{karatzas1998methods} pertains to American options, the arguments can be consulted for our purposes and carry over naturally to this setting. We omit the details.

\begin{proposition}\label{prop:C1.interior.cont.reg.}
	The value function $V$ satisfies:
	\[\partial_tV(t,\pi)+\mathcal{L}V(t,\pi)=-c \ \ \ \ \forall (t,\pi)\in \mathscr{C}.\]
	In particular, $V$ is $C^{1,2}$ in $\mathscr{C}$.
\end{proposition}

\subsection{Structure of the continuation region}

In order to study our game, we will need a detailed characterization of the continuation region $\sC$. Unlike in the classic problem where $g(\pi)=a_1\pi\wedge a_2(1-\pi)$, we will see that $\sC$ may be empty for some specifications of $g$ and $c$. We show that under suitable conditions, $\sC$ is characterized by two boundaries which are bounded away from $0$ and $1$, and enclose a strip around $[0,T)\times\{1/2\}$.

The continuation region $\sC$ is closely related to the following set which can be computed a priori:
\[\sU:=\left\{ (t,\pi)\in[0,T)\times(0,1):\mathcal{L}g(t,\pi)<-c\right\}.\]
Let $\sC_t:=\sC\cap\left(\{t\}\times(0,1)\right)$ be the $t$-slice of $\sC$ and let $\sC_{[t,T)}:=\bigcup_{t\leq s<T}\sC_t$. Similarly, we define $\sU_t:=\sU\cap\left(\{t\}\times(0,1)\right)$ and $\sU_{[t,T)}:=\bigcup_{t\leq s<T}\sU_t$.

\begin{proposition}\label{prop:cont.region.nonempty}\mbox{}
	\begin{enumerate}
		\item[(i)] $\sU_t\subset \sC_t$.
		\item[(ii)] $\sC_{[t,T)}=\emptyset$ if and only if $\sU_{[t,T)}=\emptyset$.
		\item[(iii)] 
		There exist boundaries $b(\cdot)<1/2<B(\cdot)$ such that 
		\begin{equation}\label{eqn:cont.reg.bdy.repr}
			\sC=\{(t,\pi)\in [0,T)\times(0,1):b(t)<\pi<B(t)\}.
		\end{equation}  
		\end{enumerate}
	\end{proposition}
	
	\begin{proof}
		(i) Let $(t,\pi)\in\sU_t$. By continuity there exists an $\epsilon>0$ such that $\mathcal{L}g<-c$ on $U:=[t,t+\epsilon)\times (\pi-\epsilon,\pi+\epsilon)\subseteq[0,T)\times(0,1)$. Let $\tau_U$ be the first exit time of $\Pi^{t,\pi}$ from $U$. As $\tau_U\in \cT^{\mathbf{X}}_{t}$, by Dynkin's formula,
		\begin{align*}
			V(t,\pi)&\leq \mathbb{E}\left[c(\tau_U-t)+g(\Pi_{\tau_U}^{t,\pi})\right]\\
			&=g(\pi)+\mathbb{E}\left[c(\tau_U-t)+\int_t^{\tau_U}\mathcal{L}g(s,\Pi_{s}^{t,\pi})ds\right]<g(\pi).
		\end{align*}
		It follows that $(t,\pi)\in \sC_t$.
		
		(ii) Necessity follows from (i). For sufficiency suppose $\mathcal{L}g(s,\pi)\geq -c$ for all $(s,\pi)\in [t,T)\times (0,1)$. For any stopping time $\tau\in\cT^{\mathbf{X}}_{s}$ with $s\ge t$, we have by Dynkin's formula that
		\[\mathbb{E}\left[g(\Pi^{s,\pi}_\tau)\right]=g(\pi)+\mathbb{E}\left[\int_{s}^\tau\mathcal{L}g(r,\Pi_{r}^{s,\pi})dr\right]\geq g(\pi)-\mathbb{E}\left[c(\tau-s)\right],\]
		and thus, $\mathbb{E}[c(\tau-s)+g(\Pi^{s,\pi}_\tau)]\geq g(\pi)$.
		Since $\tau$ and $(s,\pi)$ are arbitrary, the result $\sC_{[t,T)}=\emptyset$ is immediate.
		
		(iii) Assumption (G1) implies that $g\ge 0$, which further yields $V\geq0$. On the other hand, we always have $V\le g$. Therefore, it must hold that $V(t,0)=g(0)=0=g(1)=V(t,1)$. Define for $t\in [0,T)$, 
		\[\underline{\sD}_t:=\{\pi\in[0,1/2]:V(t,\pi)=g(\pi)\}, \quad b(t):=\sup \underline{\sD}_t,\] 
		\[\overline{\sD}_t:=\{\pi\in [1/2,1]:V(t,\pi)=g(\pi)\}, \quad B(t):=\inf \overline{\sD}_t.\]
		It is easy to see that $\underline{\sD}_t$ and $\overline{\sD}_t$ are non-empty and closed (by continuity). It follows that $b(t)\in \underline{\sD}_t$ and $B(t)\in \overline{\sD}_t$.
		Claim that $b(t) < 1/2 < B(t)$. Indeed, $\cL g(t, 1/2) =\|\Delta\bh_\mu(t)\|^2 \cA g(1/2) < -c$ by Assumption (G3), which implies $(t,1/2)\in \sU_t$. By (i), we then have $(t,1/2)\in \sC_t$. It remains to show that the curves $b$ and $B$ define $\sC$.
		
		From the definition of $\sC_t$, it is clear that $\{t\}\times(b(t),B(t))\subseteq \sC_t$. We want to show the reverse set inclusion, 
		i.e.\ there are no isolated ``pockets" of the continuation region below $b$ or above $B$.
		We will focus on the case $0< \pi < b(t)$ as the other case is analogous. Note that the endpoints $0$ and $b(t)$ have already been eliminated since they lie in $\underline{\sD}_t$.
		
		First, observe that if $0< \pi < b(t)$, then $\pi \notin \sU_t$. This is a consequence of the quasiconvexity of $\cA g$ (see Remark \ref{rmk:consequences.of.assumptions}) from which one can deduce that $\sU_t$ is convex (hence connected).
		Being a convex set which contains $(t,1/2)$ and which does not contain $(t, b(t))$, $(t, B(t))$ (by (i)), $\sU_t$ must be a subset of $\{t\}\times(b(t),B(t))$.
		
		Next, we show that $O:=\sC\cap \{(t,\pi)
		\in[0,T)\times(0,1):0< \pi < b(t)\}$ is empty. Suppose to the contrary $\exists (t_0,\pi_0)\in O$. Let $\tau_O$ be the first time $\Pi^{t_0,\pi_0}$ exits from $O$. It must coincide with the first exit time from $\sC$ when starting from $(t_0, \pi_0)$ because the boundaries of $\{(t,\pi)\in[0,T)\times(0,1):0< \pi < b(t)\}$ are not in $\sC$. 
		By optimality we can write:
		\[V(t_0,\pi_0)=\mathbb{E}\left[c(\tau_O-t)+g(\Pi^{t_0,\pi_0}_{\tau_O})\right].\]
		Since $\mathcal{L}g\geq-c$ on $O$ (as $O\cap\sU=\emptyset$), we find as in (ii) that $V(t_0,\pi_0)\geq g(\pi_0)$. This is a clear contradiction to $(t_0,\pi_0)\in \sC$.
	\end{proof}
	
	The next proposition says the boundaries $b$, $B$ are uniformly bounded away from $0$, $1/2$, and $1$. The proof makes use of the solution to the infinite horizon version of the problem with constant volatility which we address in Appendix  \ref{sec:inf.horizon} and may be of independent interest.
	
	
	\begin{proposition}\label{prop:bounds.on.bdys}
		There exist constants $b_*,b^*,B_*,B^*$, independent of $\mu$, such that the boundaries $b$ and $B$ defined in Proposition~\ref{prop:cont.region.nonempty}(iii) satisfy
		\[0<b_*\leq b(t)\leq b^*<1/2<B_*\leq B(t)\leq B^*<1, \quad \forall\, t\in [0, T).\]
	\end{proposition}
	
	\begin{proof}
		Let $\underline{V}$ be defined in Proposition~\ref{prop:value.func}(ii), and let $\underline{V}_\infty$ be the value function for the infinite horizon problem associated with the volatility $\mathfrak{H}$. By Proposition \ref{prop:value.func}(ii) and the nature of the infimum, we have that $V\geq \underline{V}\geq \underline{V}_\infty$. Consequently, the continuation region for $V$ must be contained in that for $\underline{V}_\infty$. The latter is defined by two constant boundaries $0<b_*<B^*<1$ (see Proposition~\ref{prop:inf.horizon.bdy}). Thus, we have shown $0< b_\ast \le b(t) < 1/2 < B(t) \le B^\ast < 1$ for $t\in [0, T)$.
		
		To get the inner bounds, we use Assumption (G3). By continuity, there exist constants $b^\ast < 1/2 < B_\ast$ such that $\cA g < -c/\mathfrak{h}^2$ on $[b^\ast, B_\ast]$. Again, using the relations $\cL g(t, \pi) = \norm{\Delta_\mu \bh(t)}^2 \cA g(\pi)$ and $\norm{\Delta_\mu \bh(t)} \ge \mathfrak{h}$, as well as Proposition \ref{prop:cont.region.nonempty}(i), we obtain $\{t\}\times[b^\ast, B_\ast]\subseteq \sU_t \subseteq \sC_t$ for all $t$.
	\end{proof}

\subsection{Regularity of the stopping boundaries}

With our characterization of the continuation region complete, we turn to establishing the regularity of the defining boundaries. 

\begin{proposition}\label{prop:bdy.unif.local.lip}
	The boundaries $b(\cdot)$ and $B(\cdot)$ are Lipschitz on any compact subset of [0,T) and the local Lipschitz constant can be made independent of $\mu$.
	
\end{proposition}

\begin{proof}
	
	We will treat the upper boundary as the lower boundary is analogous. Define $J=V-g$ with the additional restriction that its domain is the upper region $(t,\pi)\in[0,T)\times[1/2,1)$ for all $t$.  The upper boundary $B(\cdot)$ is the minimum element of the zero level set of $J$, i.e. $J(t,B(t))=0$ for $t\in[0,T)$. The motivation for this proof is to use the derivatives of $J$ to obtain a local estimate on the regularity of the boundary. Since by Proposition \ref{prop:C1.interior.cont.reg.} we can only make use of the derivatives in the interior of the continuation region, we first step away from the boundary and attempt to characterize the regularity of a family of ``$\delta$-level" boundaries $B_\delta(t)$. By obtaining a uniform estimate in $\delta$ and showing that $B(t)=\lim_{\delta\to 0}B_\delta(t)$ we are able to extend the regularity to $B$. To obtain the regularity, the idea is to use a time and space transformation (see Appendix \ref{sec:time-change}) alongside the differentiable 
	flow and pathwise derivative techniques of \cite{de2019lipschitz}. We model much of our analysis after the approach of \cite[Theorem 4.3]{de2019lipschitz}.
	
	We begin by collecting some properties of $J$ and the continuation region. From Lemma \ref{lem:sign.der.hat.V} and the discussion preceding Proposition \ref{prop:der.bd.for.bdy.reg} we have $\partial_\pi J(t,\cdot)>0$ on $(1/2,B(t))$ so that $J(t,\cdot)$ is strictly increasing on $(1/2,B(t))$. Moreover, by definition $J\leq 0$, and  by continuity $\lim_{\pi\uparrow B(t)}J(t,\pi)=0$. By Proposition \ref{prop:bounds.on.bdys} we can pick a $\overline{\pi}\in(1/2,B_*)$ so that for all $t$, $(t,\overline{\pi})\in\mathscr{C}$. Evidently, $J(t,\overline{\pi})<0$ for all $t$.
	
	We now fix $t\in[0,T)$. To analyze the boundary regularity locally we choose an (arbitrary) interval of analysis that is independent of $\mu$. Pick $t_0,t_1\in [0,T)$ with $t_0<t<t_1$ so that $[t_0,t_1]\subset [0,T)$ (if $t=0$ then we can extend $J$ to $(-\epsilon,T)\times[1/2,1)$ so that this argument can be repeated without loss of generality). Next, we want to ensure that a given $\delta$-level boundary is well-defined for all $s\in[t_0,t_1]$. Define $\delta_0:=\sup_{s\in [t_0,t_1]}J(s,\overline{\pi})$. By continuity and compactness, the supremum is attained and consequently $\delta_0<0$. As $J(s,\cdot)$ is strictly increasing and continuous, we have that for fixed $s\in [t_0,t_1]$ that $J(s,\pi)$ attains every $\delta\in(\delta_0,0)$ on the spatial interval $(1/2,B(s))$. So, for $s\in[t_0,t_1]$ if $\delta\in(\delta_0,0)$ there is a single point defining the $\delta$-level set. 
	
	Since $J$ is continuously differentiable in the interior of the continuation region we have by the implicit function theorem that there is a continuously differentiable function $B_\delta(s)$ defined on $(t_0,t_1)$ such that
	\begin{equation}\label{eqn:delta.level}
		J(s,B_\delta(s))=\delta.
	\end{equation}
	Notice that we must have $B_\delta(s)<B(s)$ and $B_\delta(s)$ is increasing in $\delta$. By monotonicity and the boundedness of $B(s)$ there is a limit function $B_0(s)$ such that
	\[\lim_{\delta\uparrow0}B_\delta(s)=B_0(s)\leq B(s), \ \ \ s\in(t_0,t_1).\]
	On the other hand, by taking limits across the previous equality \eqref{eqn:delta.level} we have
	\[0=\lim_{\delta\uparrow0}J(s,B_\delta(s))=J(u,B_0(s))\]
	so that $B_0(s)\geq B(s)$ by definition of $B(s)$. We conclude $B_0(s)=B(s)$ and so $B_\delta(s)\uparrow B(s)$ on $(t_0,t_1)$ as $\delta\uparrow 0$.

	We want to use this convergence to show that $B(s)$ is Lipschitz on any compact subset of $(t_0,t_1)$ containing $t$. To obtain our desired result we will adapt the arguments of \cite[Theorem 4.3]{de2019lipschitz}. Let $\epsilon>0$ be arbitrarily small enough so that $t\in I_{\epsilon}:=[t_0+\frac{\epsilon}{2},t_1-\frac{\epsilon}{2}]\subset (t_0,t_1)$. Again by the implicit function theorem we have
	\[B_\delta'(s)=-\frac{\partial_t J(s,B_\delta(s))}{\partial_\pi J(s,B_\delta(s))}, \ \ \ s\in I_{\epsilon}.\]
	Then by Proposition \ref{prop:der.bd.for.bdy.reg} we find that there exists a $K_\epsilon$ \textit{independent} of $\delta$ and the measure $\mu$ such that:
	\[\left|\frac{\partial_t J(s,B_\delta(s))}{\partial_\pi J(s,B_\delta(s))}\right|\leq K_\epsilon, \ \ \ s\in I_\epsilon.\]
	The proof of this result may be of independent interest as it makes use of the time and space change discussed in Appendix \ref{sec:time-change} alongside pathwise derivative techniques inspired by \cite{de2019lipschitz}. As a result of this bound, by $B_\delta(s) \to B(s)$ and the compactness of uniformly bounded and uniformly Lipschitz functions (by the Arzel\`a-Ascoli Theorem) we have that $B(s)$ must also be $K_{\epsilon}$-Lipschitz on $I_{\epsilon}$. 
	
	As $ t_0<t<t_1$ and $\epsilon>0$ were arbitrary we have that for any $t\in[0,T)$ there exists an $\epsilon'>0$ and a constant $K_{\epsilon'}$, both independent of $\mu$, such that $B(s)$ is $K_{\epsilon'}$-Lipschitz on $[t-\epsilon',t+\epsilon']\cap[0,T)$. As $\mathbb{R}$ is locally compact, we get that there is a uniform (in $\mu$) Lipschitz constant for $B$ on any compact subset of $[0,T)$. This completes the proof.
\end{proof}

\subsection{Free-boundary problem and boundary integral equations}\label{sec:free.bdy}

Our final step in the analysis of the representative agent problem is to derive the free-boundary problem associated with the value function and obtain a characterization of the free boundaries as the solution to a pair of integral equations.

\begin{proposition}\label{prop:PDE}
	The value function $V$ and stopping boundaries $b(\cdot),B(\cdot)$ satisfy the free boundary problem
	\[\begin{cases}
		\partial_tV(t,\pi)+\mathcal{L}V(t,\pi)=-c & \pi\in (b(t),B(t)), \  t\in[0,T)\\
		V(t,\pi)=g(\pi) & \pi\in [0,b(t)]\cup[B(t),1], \ t\in[0,T)
	\end{cases}.\]
	Moreover, the smooth fit condition holds in the sense that the function $\pi\mapsto V(t,\pi)$ is $C^1[0,1]$ for all $t\in[0,T)$.
\end{proposition}

\begin{proof}
	The form of the PDE when $t$ and $\pi$ satisfy $b(t)<\pi<B(t)$ follows from Proposition \ref{prop:C1.interior.cont.reg.}. The smooth fit condition follows from Proposition \ref{prop:smooth.fit} which exploits the time change techniques introduced in Appendix \ref{sec:time-change}.
\end{proof}

%

\begin{proposition}\label{prop:boundary.integral.equation}
	The pair $(b,B)$ solves:
	\begin{equation}\label{eq:integral.eq}
		\begin{cases}
		g(r(t))=\mathbb{E}\left[g( \Pi_{T}^{t,r(t)})\right]+c\int_t^{T}\mathbb{P}\left(\Pi_u^{t,r(t)}\in(r(u),R(u))\right)du\\
		\quad \quad \quad  -\int_t^{T}\mathbb{E}\left[\mathcal{L}g(u,\Pi_u^{t,r(t)})\mathds{1}_{\left\{\Pi_u^{t,r(t)}\in(0,r(u))\cup(R(u),1)\right\}}\right]du.\\
		g(R(t))=\mathbb{E}\left[g( \Pi_{T}^{t,R(t)})\right]+c\int_t^{T}\mathbb{P}\left(\Pi_u^{t,R(t)}\in(r(u),R(u))\right)du\\
		\quad \quad \quad  -\int_t^{T}\mathbb{E}\left[\mathcal{L}g(u,\Pi_u^{t,R(t)})\mathds{1}_{\left\{\Pi_u^{t,R(t)}\in(0,r(u))\cup(R(u),1)\right\}}\right]du.
		\end{cases}
	\end{equation}
	Moreover, it is the maximal continuous solution in the sense that if the pair $(r,R)$ is a continuous solution to the above integral equation with $r\leq b$ and $B\leq R$ (i.e. $b$ and $B$ are contained between $r$ and $R$), then $r=b$ and $B=R$.
\end{proposition}

\begin{proof}
	To establish this result we will adapt the arguments of \cite[Theorem 5.1]{ekstrom2015bayesian}. Let $(t,\pi)\in [0,T)\times (0,1)$. For sufficiently small $\epsilon>0$ define
	\[\tau_\epsilon:=\inf\{s\geq t: \Pi_s^{t,\pi}\not\in(\epsilon,1-\epsilon)\} \wedge (T-\epsilon).\]
	Applying \cite[Theorem 3.1 and Remark 3.2]{peskir2005change}
	to $V(s\wedge \tau_\epsilon, \Pi_{s\wedge\tau_\epsilon}^{t,\pi})$ we obtain:
	\begin{align*}\mathbb{E}\left[V(\tau_\epsilon, \Pi_{\tau_\epsilon}^{t,\pi})\right]&=V(t,\pi)+\mathbb{E}\left[\int_t^{\tau_\epsilon}(\partial_t +\mathcal{L})V(u,\Pi_u^{t,\pi})\mathds{1}_{\left\{\Pi_u^{t,\pi}\not\in\{b(u),B(u)\}\right\}}du\right]\\
		&=V(t,\pi)-c\mathbb{E}\left[\int_t^{\tau_\epsilon}\mathds{1}_{\left\{\Pi_u^{t,\pi}\in(b(u),B(u))\right\}}du\right]\\
		&\quad +\mathbb{E}\left[\int_t^{\tau_\epsilon}\mathcal{L}g(u,\Pi_u^{t,\pi})\mathds{1}_{\left\{\Pi_u^{t,\pi}\in(0,b(u))\cup(B(u),1)\right\}}du\right]
	\end{align*}
	after taking expectations. The result applied here, \cite[Theorem 3.1 and Remark 3.2]{peskir2005change}, extends It\^o's formula and the sufficient conditions are readily verified in our setting by using the smooth fit property from Proposition \ref{prop:PDE}, the boundedness of $\mathcal{L}g$ (implied by the boundedness of $\mathcal{A}g$ given in Remark \ref{rmk:consequences.of.assumptions}), and the fact that the boundaries $(b,B)$ are bounded variation on $[t,T-\epsilon]$ by being locally Lipschitz on $[0,T)$ (see Proposition \ref{prop:bdy.unif.local.lip}). 
	
	Now, as $\epsilon\downarrow 0$ we have $\tau_\epsilon\to T$ a.s. and by bounded convergence we obtain after rearranging:
	\begin{align*}V(t,\pi)&=\mathbb{E}\left[g( \Pi_{T}^{t,\pi})\right]+c\int_t^{T}\mathbb{P}\left(\Pi_u^{t,\pi}\in(b(u),B(u))\right)du\\
		&\quad -\int_t^{T}\mathbb{E}\left[\mathcal{L}g(u,\Pi_u^{t,\pi})\mathds{1}_{\left\{\Pi_u^{t,\pi}\in(0,b(u))\cup(B(u),1)\right\}}\right]du.
	\end{align*}
	Setting $\pi=b(t)$ and $\pi=B(t)$ we obtain that $b(t)$ and $B(t)$ solve \eqref{eq:integral.eq}.
	
	We now show that $b(\cdot), B(\cdot)$ are the maximal continuous solutions to \eqref{eq:integral.eq}. Suppose that $(r,R)$ are continuous solutions to the integral equations with $r\leq b$ and $B\leq R$. We will show that we must have $r=b$ and $B=R$. As in \cite[Theorem 5.1]{ekstrom2015bayesian}, define
	\begin{align*}
		H(t,\pi)&:=\mathbb{E}\left[g( \Pi_{T}^{t,\pi})\right]+c\int_t^{T}\mathbb{P}\left(\Pi_u^{t,\pi}\in(r(u),R(u))\right)du\\
		&\quad -\int_t^{T}\mathbb{E}\left[\mathcal{L}g(u,\Pi_u^{t,\pi})\mathds{1}_{\left\{\Pi_u^{t,\pi}\in(0,r(u))\cup(R(u),1)\right\}}\right]du.
	\end{align*}
	We have $H(t,r(t))= g(r(t))$ and $H(t,R(t))= g(R(t))$ since $r$ and $R$ are solutions to the equations \eqref{eq:integral.eq}, and further that $H(T,\pi)=g(\pi)$. For any $(t,\pi)$ define for $s\geq t$ the two processes $(M^{t,\pi}_s)_{s\geq t}$ and $(\tilde{M}^{t,\pi}_s)_{s\geq t}$ by
	\begin{align*}M^{t,\pi}_s&:=
	H(s,\Pi_{s}^{t,\pi})
	+c\int_t^{s} \mathds{1}_{\left(r(u),R(u)\right)}(\Pi^{t,\pi}_{u})du\\
	&\quad \quad -\int_t^{s}\mathcal{L}g(u,\Pi_u^{t,\pi})\mathds{1}_{\left\{\Pi_u^{t,\pi}\in(0,r(u))\cup(R(u),1)\right\}}du
	\end{align*}
	and
	\begin{align*}
		\tilde{M}^{t,\pi}_s&:=
		V(s,\Pi_{s}^{t,\pi})
		+c\int_t^{s} \mathds{1}_{\left(b(u),B(u)\right)}(\Pi^{t,\pi}_{u})du\\
		&\quad \quad -\int_t^{s}\mathcal{L}g(u,\Pi_u^{t,\pi})\mathds{1}_{\left\{\Pi_u^{t,\pi}\in(0,b(u))\cup(B(u),1)\right\}}du
	\end{align*}
	As in \cite[Theorem 5.1]{ekstrom2015bayesian}, both processes can be shown to be martingales for any $(t,\pi)$. Specifically, these processes are of the same form as the process $Y^{t,\pi}=(Y^{t,\pi}_\ell)_{\ell\geq t}$ defined below (for a placeholder integrable function $f$). Since we are conditioning on a random starting point in the equation for $H$ and $V$, by the Markov property:
	
		\begin{align*}
			Y^{t,\pi}_{\ell}&:=\mathbb{E}\left[g(\Pi_T^{\ell,\Pi^{t,\pi}_\ell})\right]+\int_\ell^T\mathbb{E}\left[f(u,\Pi_u^{\ell,\Pi^{t,\pi}_\ell})\right]du+\int_t^{\ell}f(u,\Pi_u^{t,\pi})du\\
			&=\mathbb{E}_{\ell,\Pi^{t,\pi}_\ell}\left[g(\Pi_T^{t,\pi})\right]+\int_\ell^T\mathbb{E}_{\ell,\Pi^{t,\pi}_\ell}\left[f(u,\Pi^{t,\pi}_u)\right]du+\int_t^{\ell}f(u,\Pi_u^{t,\pi})du\\
			&=\mathbb{E}\left[g(\Pi_T)\big|\mathcal{F}_\ell^{\Pi^{t,\pi}}\right]+\int_\ell^T\mathbb{E}\left[f(u,\Pi_u)\big|\mathcal{F}_\ell^{\Pi^{t,\pi}}\right]du+\int_t^{\ell}f(u,\Pi_u^{t,\pi})du\\
			&=\mathbb{E}\left[g(\Pi_T)+\int_t^Tf(u,\Pi_u)du\bigg|\mathcal{F}_\ell^{\Pi^{t,\pi}}\right].
		\end{align*}
	The martingale property follows immediately from this representation by taking conditional expectations and applying the tower property.
	
	The remainder of the proof proceeds in parts.
	
	\noindent \textit{\textbf{Part 1:}} \textit{We verify that $H(t,\pi)=g(\pi)$ for $\pi\not\in (r(t),R(t))$.}
	
	Assume that $\pi\leq r(t)\leq b(t)$ as the other case is similar. Define
	\[\tau_r:=\inf\{s\geq t: \Pi_s^{t,\pi}\geq r(s)\}\wedge T.\]
	We see that:
	\begin{equation}\label{eq:M.pt.1}
		M_{\tau_r}^{t,\pi}=H(\tau_r,\Pi_{\tau_r}^{t,\pi})-\int_t^{\tau_r}\mathcal{L}g(u,\Pi_u^{t,\pi})du
	\end{equation}
	and
	\begin{equation}\label{eq:tilde.M.pt.1}
		\tilde{M}_{\tau_r}^{t,\pi}=V(\tau_r,\Pi_{\tau_r}^{t,\pi})-\int_t^{\tau_r}\mathcal{L}g(u,\Pi_u^{t,\pi})du
	\end{equation}
	by the definition of the stopping time and our assumption on the boundaries. Then
	\begin{align*}
		H(t,\pi)&=\mathbb{E}\left[M_{\tau_r}^{t,\pi}\right]=\mathbb{E}\left[H(\tau_r,\Pi_{\tau_r}^{t,\pi})-\int_t^{\tau_r}\mathcal{L}g(u,\Pi_u^{t,\pi})du\right]\\
		&=\mathbb{E}\left[g(\Pi_{\tau_r}^{t,\pi})-\int_t^{\tau_r}\mathcal{L}g(u,\Pi_u^{t,\pi})du\right]\\
		&=\mathbb{E}\left[V(\tau_r,\Pi_{\tau_r}^{t,\pi})-\int_t^{\tau_r}\mathcal{L}g(u,\Pi_u^{t,\pi})du\right]\\
		&=\mathbb{E}\left[\tilde{M}_{\tau_r}^{t,\pi}\right]=V(t,\pi)=g(\pi).
	\end{align*}

	The first equality is justified by optional sampling and the martingale property of $M$, the second by \eqref{eq:M.pt.1}, the third by $H(T,\pi)=g(\pi)$ and $H(t,r(t))=g(r(t))$, the fourth by the fact that $r\leq b$ so $V=g$ at the stopping time, the fifth and sixth again by \eqref{eq:tilde.M.pt.1}, optional sampling and the martingale property, and the seventh by the fact that $\pi\leq r(t)\leq b(t)$. 
	
	\noindent \textit{\textbf{Part 2:}} \textit{We verify that $H(t,\pi)\geq V(t,\pi)$ for all $\pi\in(0,1)$.}
	
	In view of Part 1, it suffices to check that $H(t,\pi)\geq V(t,\pi)$ for $\pi\in (r(t),R(t))$. Let 
	\[\tau_{(r,R)}:=\inf\{s\geq t:\Pi_s^{t,\pi}\not\in (r(s),R(s))\}\wedge T\]
	Then, arguing similarly to Part 1 we obtain:
	
	\begin{align*}
		H(t,\pi)&=\mathbb{E}\left[M_{\tau_{(r,R)}}^{t,\pi}\right]\\
		&=\mathbb{E}\left[H(\tau_{(r,R)},\Pi_{\tau_{(r,R)}}^{t,\pi})
		+c\int_t^{\tau_{(r,R)}} \mathds{1}_{\left(r(u),R(u)\right)}(\Pi^{t,\pi}_{u})du\right]\\
		&=\mathbb{E}\left[g(\Pi_{\tau_{(r,R)}}^{t,\pi})
		+c(\tau_{(r,R)}-t)\right]\\
		&\geq V(t,\pi).
	\end{align*}
	Here the last inequality follows by the definition of $V$.
	
	\noindent \textit{\textbf{Part 3:}} \textit{We verify that $r(t)=b(t)$ and $B(t)=R(t)$ for all $t\in[0,T)$.}
	
	Suppose that $b(t)>r(t)$ for some $t$. Take $\pi=r(t)$ and let
	\[\tau_b:=\inf\{s\geq t:\Pi_s^{t,r(t)}\geq b(s)\}\wedge T.\]
	Now 
	\begin{align*}
		0&=g(r(t))-g(r(t))=H(t,r(t))-V(t,r(t))\\
		&=\mathbb{E}\left[M_{\tau_b}^{t,r(t)}\right]-\mathbb{E}\left[\tilde{M}_{\tau_b}^{t,r(t)}\right]\\
		&=\mathbb{E}\left[H(\tau_b,\Pi_{\tau_b}^{t,r(t)})-V(\tau_b,\Pi_{\tau_b}^{t,r(t)})\right] + \mathbb{E}\left[c\int_t^{\tau_b}\mathds{1}_{\left\{\Pi_u^{t,r(t)}\in(r(u),b(u))\right\}}du\right]\\
		&\quad \quad \quad -\mathbb{E}\left[\int_t^{\tau_b}\mathcal{L}g(u,\Pi_u^{t,r(t)})\mathds{1}_{\left\{\Pi_u^{t,r(t)}\in(0,r(u))\right\}}du\right]\\
		& \quad \quad \quad +\mathbb{E}\left[\int_t^{\tau_b}\mathcal{L}g(u,\Pi_u^{t,r(t)})\mathds{1}_{\left\{\Pi_u^{t,r(t)}\in(0,b(u))\right\}}du \right]\\
		&=\mathbb{E}\left[H(\tau_b,\Pi_{\tau_b}^{t,r(t)})-V(\tau_b,\Pi_{\tau_b}^{t,r(t)})\right] + \mathbb{E}\left[c\int_t^{\tau_b}\mathds{1}_{\left\{\Pi_u^{t,r(t)}\in(r(u),b(u))\right\}}du\right]\\
		&\quad \quad \quad  +\mathbb{E}\left[\int_t^{\tau_b}\mathcal{L}g(u,\Pi_u^{t,r(t)})\mathds{1}_{\left\{\Pi_u^{t,r(t)}\in(r(u),b(u))\right\}}du \right]\\
		&> \mathbb{E}\left[H(\tau_b,\Pi_{\tau_b}^{t,r(t)})-V(\tau_b,\Pi_{\tau_b}^{t,r(t)})\right] + \mathbb{E}\left[c\int_t^{\tau_b}\mathds{1}_{\left\{\Pi_u^{t,r(t)}\in(r(u),b(u))\right\}}du\right]\\
		&\quad \quad \quad -\mathbb{E}\left[c\int_t^{\tau_b}\mathds{1}_{\left\{\Pi_u^{t,r(t)}\in(r(u),b(u))\right\}}du \right]\\
		&=\mathbb{E}\left[H(\tau_b,\Pi_{\tau_b}^{t,r(t)})-V(\tau_b,\Pi_{\tau_b}^{t,r(t)})\right]\geq0.
	\end{align*}
	This is a contradiction and so we conclude $r=b$. Here the \textit{strict} inequality follows from  Proposition \ref{prop:cont.region.nonempty} given the form of $\mathscr{U}$, the fact that $\partial_\pi \mathcal{L}g<0$ on $(0,1/2)$ under Assumption (G2), and boundary continuity. The second inequality follows from Part 2 where we obtained $H\geq V$.  The argument that $B=R$ is analogous and so this completes the proof.
\end{proof}

\section{Existence of a Mean Field Equilibrium}\label{sec:main.proof.existence}

With the analysis of the single agent problem in Section \ref{sec:single.agent} complete, we are now in a position to tackle the proof of our main result (Theorem \ref{thm:existence}). Our approach centers on an application of the Schauder-Tychonoff Fixed Point Theorem (see \cite[Corollary 17.56]{guide2006infinite}) to the mapping \eqref{eqn:fixed.pt.mu} which we denote by $\Phi: \mathcal{P}([0,T])^2\rightarrow \mathcal{P}([0,T])^2$ and recall here as
%
\begin{equation}\label{eq:fixed.pt.mapping}
\Phi(\mu)= \left(\cL(\tau^{\iota, \ast}_\mu |\theta=0), \cL(\tau^{\iota, \ast}_\mu |\theta=1)\right).
\end{equation} 
This typically requires verifying the compactness and convexity of the underlying space (when viewed as a subset of a suitable locally convex Hausdorff topological vector space), as well as, the continuity of the mapping.

The continuity of $\Phi$ can be broken down into several steps, as shown in the next three auxiliary lemmas. In this section, we add back the index $i\in I$ on all agent-dependent elements of the problem. 

\begin{lemma}[Continuity of the value function in $\mu$] \label{lem:cont.val.funct}
	Let Assumption \ref{ass:g} hold. If $\mu_n\Rightarrow \mu$, then the corresponding signals $\mathbf{h}^i_{\mu_n}(t,\theta)$, volatilities $\|\Delta\bh^i_{\mu_n}(t)\|$, and value functions converge pointwise:
	\[\bh^i_{\mu_n}(t,\cdot)\to \bh^i_{\mu}(t,\cdot), \ \ \ \|\Delta\bh^i_{\mu_n}(t)\|\to\|\Delta\bh^i_\mu(t)\|,\]
	\[ \ \ \mathrm{and} \ \  V_i(t,\pi;\mu_n)\to V_i(t,\pi;\mu), \ \ \ \forall(t,\pi)\in[0,T)\times[0,1].\]
\end{lemma}
\begin{proof}
It is easy to check that as $\mu_n\Rightarrow \mu$ we have that the regularized distributions given by \eqref{eq:F.def} satisfy:
$F_{\mu_n}^j(t)\to F_\mu^j(t)$ for $j\in\{0,1\}$ and all $t\in[0,T]$. By the continuity of $\bh^i_\mu(t,j):=\bh^i(t,j,F^j_\mu(t))$ in the argument corresponding to $F^j_\mu(t)$, we readily obtain the pointwise convergence of the signals $\bh^i_{\mu_n}\to\bh^i_{\mu}$ and volatilities $\|\Delta\bh^i_{\mu_{n}}\|\to \|\Delta\bh^i_{\mu}\|$. It remains to show the continuity of the value function in $\mu$.

As noted in Remark \ref{rmk:dist.pi.mu}, the dependence of the value function on $\mu$ is entirely through the influence of $\|\Delta \bh^i_{\mu}\|$ on the distribution of $\Pi^i$. Hence it suffices to show that the value function is continuous in the volatility.
The arguments are adapted from those of Erik Ekstr\"om who graciously provided detailed notes on the omitted proof of \cite[Theorem 4.5]{ekstrom2004properties} which asserts continuity in a related (but not identical) setting. The analysis relies on the idea of volatility times which correspond to the quadratic variation of continuous local martingales. Since, in our problem, we may equivalently work with $\mathbb{F}^\Pi$ stopping times (see Section \ref{sec:value_fun}) we can allow the posterior probability process to be driven by an \textit{arbitrary} Brownian motion. Then, using the work of Janson and Tysk \cite{janson2003volatility} we may write the sequence of problems indexed by $\mu_n$ in terms of the \textit{same} Brownian motion but multiple volatility times. This perspective is incredibly helpful for establishing the desired continuity. For completeness, we include the proof in full detail in Appendix \ref{app:cont.val.funct}. 
\end{proof}

\begin{lemma}[Continuity of the boundaries in $\mu$]\label{lem:conv.bdys}
Let Assumption \ref{ass:g} hold. If $\mu_n\Rightarrow \mu$, then the corresponding optimal stopping boundaries $B^i_n, b^i_n$ converge to $B^i, b^i$ locally uniformly on $[0, T)$.
\end{lemma}

We defer the proof of boundary convergence to Appendix \ref{app:bdy.convergence}. The idea is to establish a unique limit point of the boundary sequence. For any subsequence of $(b^i_{n}), (B^i_{n})$, we argue that we may extract a further subsequence $(b^i_{n'}), (B^i_{n'})$ that
converges locally uniformly to some $\overline{b}^i$ and $\overline{B}^i$.  Then, by taking limits across the equations \[V_i(t,b_{n'}^i(t);\mu_{n'})=g_i(b_{n'}^i(t)), \ \ \  V_i(t,B_{n'}^i(t);\mu_{n'})=g_i(B_{n'}^i(t)),\] we can obtain $\overline{b}^i\leq b^i$, $B^i\leq \overline{B}^i$. The reverse direction (i.e. $b^i\leq \overline{b}^i$ and $\overline{B}^i\leq B^i$) uses that the pair $(b^i,B^i)$ is the maximal continuous solution to the integral equations \eqref{eq:integral.eq} in Proposition \ref{prop:boundary.integral.equation}.


\begin{lemma}[Continuity of the stopping times in $\mu$]\label{lem:conv.stopping.times}
Let Assumption \ref{ass:g} hold and recall that $\tau^{i, \ast}_{\mu_n}, \tau^{i, \ast}_{\mu}$ denote the smallest optimal stopping time for the single agent problem \eqref{eq:single-agent-problem} associated with the input measures $\mu_n$ and $\mu$, respectively. If $\mu_n\Rightarrow \mu$, then $\tau^{i, \ast}_{\mu_n}|_{\theta = j}$ converges to $\tau^{i, \ast}_{\mu}|_{\theta = j}$ in probability, for $j = 0, 1$.
\end{lemma}


\begin{proof}
For a fixed agent $i$ and an arbitrary input measure $\mu$, observe that their associated posterior probability process $\Pi^i$ hits their optimal stopping boundaries $b^i$ or $B^i$ if and only if $L^i=S^{-1}(\Pi^i)$ from \eqref{eq:defn-L} hits $m^i:=S^{-1}(b^i)$ or $M^i:=S^{-1}(B^i)$. Moreover, since $b^i$ and $B^i$ are bounded away from $0$ and $1$ (see Proposition \ref{prop:bounds.on.bdys}), and $S^{-1}(\cdot)$ is locally Lipschitz, the boundaries $m^i$ and $M^i$ inherit the uniform (in $\mu$) local Lipschitz property of $b^i$ and $B^i$ from Proposition \ref{prop:bdy.unif.local.lip}.
	
Let $b_n^i$ and $B_n^i$ be the boundaries associated with $\mu_n$ and define the corresponding transformed boundaries $m_n^i=S^{-1}(b^i_n)$ and $M^i_n=S^{-1}(B^i_n)$ analogously to the definition of $m^i$ and $M^i$ above. Similarly let $L^{n,i}$ and $L^i$ be the log-likelihood ratio processes whose dynamics in \eqref{eq:L.dynamics} are induced by $\mu_n$ and $\mu$ for the agent $i$. Define the hitting times to each of the boundaries up to time $T$:
\[\tau_{m^i_n}^i:=\inf\{t\geq0:m_n^i(t)-L^{n,i}_t\geq 0\}\wedge T,\]
\[\tau_{M^i_n}^i:=\inf\{t\geq0:L^{n,i}_t-M^i_n(t)\geq 0\}\wedge T.\]
By Lemma \ref{lem:conv.bdys}, $b_n^i\to b^i$ and $B_n^i\to B^i$ locally uniformly, and by our transformation we also have $m_n^i\to m^i$ and $M_n^i\to M^i$ locally uniformly on $[0,T)$. We can now write the conditional hitting times $\tau_{m^i_n,j}^i:=\tau_{m^i_n}^i|_{\theta=j}$, $\tau_{M^i_n,j}^i:=\tau_{M^i_n}^i|_{\theta=j}$, and express the (smallest) conditional optimal stopping time for the input measure $\mu_n$,  $\tau^{i,*}_{\mu_n,j}:=\tau^{i,*}_{\mu_n}|_{\theta=j}$, as:
\[\tau^{i,*}_{\mu_n,j}=\tau^i_{m^i_n,j}\wedge \tau^i_{M^i_n,j}.\]
In Lemma \ref{lem:properties.of.L}, we have already observed that $L^{n,i}$ and $L^i$ are conditionally Gaussian, and in Appendix \ref{app:conv.stopping.times} we establish a result on convergence of hitting times for Gaussian processes. In particular, since the coefficients of our Gaussian process converge (see \eqref{eq:cond.dyn.L.1}-\eqref{eq:cond.dyn.L.0} and Lemma \ref{lem:cont.val.funct}) and the initial condition $L_0^{n,i}=L_0^i=S^{-1}(\pi)$ is fixed, by Lemma \ref{lem:conv.stop.time} and the continuous mapping theorem we have that $\tau^{i,*}_{\mu_n,j}\to \tau^{i,*}_{\mu,j}$ in probability. 
\end{proof}

We now collect these results in order to prove our main theorem.

\begin{proof}[Proof of Theorem~\ref{thm:existence}] 
We formally verify the conditions of Schauder-Tychonoff Fixed Point Theorem for $\Phi$ beginning with the compactness and convexity of the underlying space. Let $\mathcal{M}([0,T])$ be the locally convex Hausdorff topological vector space of finite signed measures on the Borel $\sigma$-field of $[0,T]$. The space of probability measures $\mathcal{P}([0,T])$ is a compact and convex subset of $\mathcal{M}([0,T])$ when it is equipped with the weak*-topology\footnote{This coincides with the usual topology of weak convergence on $\cP([0,T])$.}. Since the product space $\mathcal{P}([0,T])^2$ inherits compactness and convexity in the parent product space (which is also a locally convex Hausdorff topological vector space), the remaining requirement of Schauder's Theorem is continuity.

We introduce the notation $\mathbb{P}^\theta$ to represent the conditional measure $\mathbb{P}(\cdot|\theta)$. It is clear from Definition \ref{def:MFE} that the measure associated with conditional law in \eqref{eq:fixed.pt.mapping} can be written by disintegrating in the agent space and representing the conditional law of the optimal stopping time for a given agent $i$ as the pushforward of $\mathbb{P}^\theta$ by $\tau^{i,*}_\mu$. We denote this measure by $\nu_\mu$ which is defined formally below for all Borel-measurable $A\subseteq[0,T]$:	\[\nu_\mu(A):=\int_{I}\int_{\Omega}\mathds{1}_{\tau^{i,*}_\mu(\omega)\in A}\mathbb{P}^\theta(d\omega)\lambda(di).\] 
	Suppose that $\mu_n\Rightarrow \mu$ and let $f\in C_b([0,T])$. Then,
	\begin{align*}
		\int f d\nu_{\mu_n}=\int_{I}\int_{\Omega}f(\tau^{i,*}_{\mu_n}(\omega))\mathbb{P}^\theta&(d\omega)\lambda(di)\\
		&\xrightarrow{n\to\infty} \int_{I}\int_{\Omega}f(\tau^{i,*}_\mu(\omega))\mathbb{P}^\theta(d\omega)\lambda(di)=\int f d\nu_{\mu},
	\end{align*}
	where the convergence follows by the bounded convergence theorem and the weak convergence of the conditional stopping times $\tau_{\mu_n}^{i,*}|_{\theta}$ for all $i$ (implied by Lemma \ref{lem:conv.stopping.times}). This gives the convergence we require and so, the mapping \eqref{eq:fixed.pt.mapping} is continuous.

Taken together, we have established that our mapping \eqref{eq:fixed.pt.mapping} has at least one fixed point. In view of Definition \ref{def:MFE} this gives the existence of a mean field equilibrium to the sequential testing game and completes the proof of Theorem \ref{thm:existence}.
\end{proof}

\section{Preemption Games with the Classic Loss}\label{sec:classic}

Our existence result (Theorem \ref{thm:existence}) does not extend directly to games involving the classic loss function. The penalty $g(\pi)=a_1\pi\wedge a_2(1-\pi)$ contains a kink which creates technical difficulties that are not addressed by the techniques we use to establish boundary regularity (Proposition \ref{prop:bdy.unif.local.lip}), a critical ingredient in our proof. However, we can get around this limitation in the special case where the volatility $\|\Delta\bh_\mu(t)\|$ is monotone decreasing for all input measures $\mu$. This can be used to prove monotonicity of the optimal stopping boundaries $b,B$, which in turn can be substituted for the aforementioned regularity requirements. 


\renewcommand{\theassumption}{C} 
\begin{assumption}[Classic Loss]\label{ass:classic}\mbox{}
	\begin{enumerate}
		\item[(C1)] The penalty function is $g_i(\pi)=a^i_1\pi\wedge a^i_2(1-\pi)$ where $a^i_1,a^i_2>0$.
		\item[(C2)] The volatility $\|\Delta\bh^i_\mu(t)\|$ is monotone decreasing for all $\mu\in\mathcal{P}([0,T])^2$.
	\end{enumerate}
\end{assumption}
\renewcommand{\theassumption}{\arabic{section}.\arabic{assumption}}

An example where this condition holds is the preemption game version (i.e. $\lambda_1<0$) of Example \ref{ex:h}. In general we can loosely think of (C2) as enforcing a preemption style game since it says that the signal gets weaker as time progresses so agents trade off the benefits of outlasting other agents and observing the process for a longer time with the cost of a weakening signal. If, as in the $\lambda_1<0$ case of Example \ref{ex:h}, the exit of other agents further weakens the signal (i.e. $\mu_1\leq \mu_2$ in first stochastic order implies $\|\Delta\bh^i_{\mu_1}(t)\|\leq \|\Delta\bh^i_{\mu_2}(t)\|$ for all $t$) then the benefit of ``preempting" other agents is even more pronounced and reinforced by the nature of the interaction.

As in Section \ref{sec:single.agent}, we will suppress the index $i$ when discussing the representative agent problem. We begin our analysis in this section by noting that most of the results in Section \ref{sec:single.agent} have analogues in the classic problem with monotone decreasing volatility. We report the main conclusions for the value function here and also obtain monotonicity in time, Lipschitz regularity in $\pi$, and a useful estimate on the second spatial derivative.

\begin{proposition}\label{prop:mon.and.lip.V}
	Under Assumption \ref{ass:classic},
	\begin{enumerate}
		\item[(i)]  $V(t,\pi)$ is jointly continuous and $\tau_\mathscr{D}$ from \eqref{eqn:smallest.stop.time} is the smallest optimal stopping time.
		\item[(ii)] 
		$V(t,\pi)$ is concave in $\pi$ and monotone in the volatility.
		\item[(iii)] $t\mapsto V(t,\pi)$ is increasing on $[0,T)$ for all $\pi\in(0,1)$.
		\item[(iv)] $\pi\mapsto V(t,\pi)$ is Lipschitz on $[0,1]$ for all $t\in[0,T)$.
		\item[(v)] $V$ is $C^{1,2}$ in $\mathscr{C}$ and satisfies:
		\[\partial_tV(t,\pi)+\mathcal{L}V(t,\pi)=-c \ \ \ \ \forall (t,\pi)\in \mathscr{C}.\]
		\item[(vi)] $\partial_{\pi\pi}V\leq -\frac{2c}{\|\Delta \bh_\mu (t)\|^2\pi^2(1-\pi)^2}\leq -\frac{8c}{\mathfrak{H}^2} <0$  for all $(t,\pi)\in\mathscr{C}$ and $\partial_{\pi\pi}V=0$ otherwise.
	\end{enumerate}
\end{proposition}

\begin{proof}
	 (i),(ii) and (v) follow from the same arguments as in Section \ref{sec:value_fun}.
	 
	(iii) This is a special case of \cite[Proposition 4.2]{ekstrom2015bayesian} using (in their notation) $\sigma(t,\pi)=\|\Delta \bh_\mu (t)\|\pi(1-\pi)$ for the diffusion coefficient\footnote{Note that the results of \cite[Section 4]{ekstrom2015bayesian} holds for both the finite and infinite horizon problems. By \cite[Remark 2.1]{ekstrom2015bayesian} the arguments extend to all constants $a_1,a_2>0$.}.

	(iv) Since $V(t,\cdot)$ is concave, it is locally Lipschitz on $(0,1)$. As $0\leq V\leq a_1\pi\wedge a_2(1-\pi)$ this can be strengthened to globally Lipschitz on $[0,1]$ with constant $\max\{a_1,a_2\}$.
	
	(vi) follows by rearranging (v) and noting from (iii) that we must have $\partial_t V\geq 0$. The bounds are then trivial.
\end{proof}

Additionally, the continuation region for the classic problem can still be characterized in terms of two time-dependent boundaries in a manner similar to Proposition \ref{prop:cont.region.nonempty}(iii) and Proposition \ref{prop:bounds.on.bdys}. We also report the monotonicity properties of the boundaries that will take the place of regularity in the analysis to follow.


\begin{proposition}\label{prop:bdy.char.classic}
	Under Assumption \ref{ass:classic} there exist boundaries $b(\cdot)<\frac{a_2}{a_1+a_2}<B(\cdot)$ such that:
	\[\mathscr{C}=\{(t,\pi)\in[0,T)\times(0,1):b(t)<\pi<B(t)\}.\]
	The lower boundary $b:[0,T)\to (0,a_2/(a_1+a_2)]$ is continuous and increasing. The upper boundary $B:[0,T)\to [a_2/(a_1+a_2),1)$ is continuous and decreasing.
	Moreover, there exist functions $\underline{b},\overline{b},\underline{B},\overline{B}$ and constants $b_*,B^*$ independent of $\mu$ such that
	\begin{enumerate}
		\item[(i)] For all $t\in[0,T)$: \[0< b_*\leq \underline{b}(t)\leq b(t)\leq \overline{b}(t)<\frac{a_2}{a_1+a_2}<\underline{B}(t)\leq B(t)\leq \overline{B}(t)\leq B^*<1,\] 
		\item[(ii)]$\underline{b}(T)=\frac{a_2}{a_1+a_2}=\overline{B}(T)$ and so $b(T)=\frac{a_2}{a_1+a_2}=B(T)$.
	\end{enumerate}
\end{proposition}

\begin{proof}
	We begin by establishing the existence of the boundaries. By the argument in \cite[Lemma 4.5]{ekstrom2015bayesian}, since $\|\Delta \bh_\mu (t)\|\geq \mathfrak{h}$ the value function satisfies
	\[V(t,a_2/(a_1+a_2))<g(a_2/(a_1+a_2))\]
	for all $t<T$. The claim is then a consequence of the concavity of $V(t,\pi)$ in $\pi$ since $V-g$ is piecewise concave and non-positive everywhere with $(V-g)(t,0)=(V-g)(t,1)=0$. Hence, $(V-g)(t,\cdot)$ must be decreasing on $[0,a_2/(a_1+a_2))$ and increasing on $(a_2/(a_1+a_2),1]$. It follows that for every $t$ there exists a $b(t), B(t)$ such that $V-g=0$ on $(b(t),B(t))^c$. Then, the claimed monotonicity and continuity follow from a special case of \cite[Proposition 4.6 and Theorem 4.10]{ekstrom2015bayesian} using (in their notation) $\sigma(t,\pi)=\|\Delta \bh_\mu (t)\|\pi(1-\pi)$ for the diffusion coefficient; also see \cite[Remark 4.11]{ekstrom2015bayesian}.
	
	For (i) it suffices to note that the upper and lower bounds in Proposition \ref{prop:value.func}(ii) on the value function and the definition of the continuation region in terms of $V$ (see \eqref{eqn:stop.region}) imply a corresponding upper and lower bound on the continuation region with respect to the usual set ordering. In view of the preceding characterization in terms of boundary functions $b(t),B(t)$ we can associate $\underline{b}(t),\overline{B}(t)$ with the formulation using $\|\Delta \bh_\mu (t)\|\equiv \mathfrak{H}$ and $\overline{b}(t),\underline{B}(t)$ with $\|\Delta \bh_\mu (t)\|\equiv \mathfrak{h}$. The upper and lower constant bounds, $b_*$ and $B^*$, then follow from the bounds on the classic problem in finite time (see \cite[Eq. (21.2.36)]{peskir2006optimal}). The terminal value in (ii) for the upper and lower bounds follows from \cite[Eq. (21.2.37)]{peskir2006optimal}.
\end{proof}

%
%

We remark here that the remaining key results of Section \ref{sec:single.agent} have analogues under Assumption \ref{ass:classic}. The representation of the value function as the solution to a free boundary problem and the associated smooth fit condition in Proposition \ref{prop:PDE} follows by appealing once more to \cite{ekstrom2015bayesian} (in particular, Proposition 4.9 and Remark 4.11 therein). Similarly, the integral equations in Proposition \ref{prop:boundary.integral.equation} for the boundaries follow from \cite[Theorem 5.1]{ekstrom2015bayesian} under Assumption \ref{ass:classic}. For the classic loss, $\mathcal{L}g$ is well defined away from $a_2/(a_1+a_2)$ and equal to $0$. Consequently, the last terms in the integral equations of Proposition \ref{prop:boundary.integral.equation} vanish in the classic setting.

With this, we have the preliminaries necessary for an adaptation of the proof in Section \ref{sec:main.proof.existence} of Theorem \ref{thm:existence}. The structure of the proof of existence (i.e. the application of Schauder's Theorem) remains the same for the classic loss, and the arguments and conclusion of Lemma \ref{lem:cont.val.funct} extend directly to the setting of Assumption \ref{ass:classic}. However, the existing boundary convergence result (Lemma \ref{lem:conv.bdys}) relies on regularity and so an alternative proof must be provided. We state this result formally here and defer the proof to Appendix~\ref{app:bdy.convergence.classic}.

\begin{lemma}[Convergence of the classic boundaries]\label{lem:conv.bdys.classic}
	Under Assumption \ref{ass:classic}, if $\mu_n\Rightarrow \mu$ then the associated boundaries for the classic penalty are such that $B_n(t)\to B(t)$ and $b_n(t)\to b(t)$ locally uniformly on $[0,T)$.
\end{lemma}


Finally, we arrive at the promised extension of our existence result to preemption games with the classic loss.

\begin{theorem}
	Under Assumption \ref{ass:classic}, there exists a mean field equilibrium for the sequential testing game.
\end{theorem}

\begin{proof}
	 To obtain the existence of an equilibrium (Theorem \ref{thm:existence}) under Assumption \ref{ass:classic} it suffices to ensure that the continuity of \eqref{eq:fixed.pt.mapping} still holds. The only missing ingredient from the proof employed in Section \ref{sec:main.proof.existence} is the conclusion of Lemma \ref{lem:conv.stopping.times}. However, this is straightforward to verify. First, the application of Lemma \ref{lem:conv.bdys} can be substituted with Lemma \ref{lem:conv.bdys.classic}. Then, by Proposition \ref{prop:bdy.char.classic} the boundaries are continuous and monotone. These properties are inherited by  $m=S^{-1}(b)$, $M=S^{-1}(B)$ which, in turn, allows us to apply the same consistency result as before (Lemma \ref{lem:conv.stop.time}) for the hitting times of Gaussian processes. It is then easy to check that the remaining arguments of Lemma \ref{lem:conv.stopping.times} carry over to the setting of Assumption \ref{ass:classic}.
\end{proof}

\section{Numerical Examples}\label{sec:numerical}

In this section we will illustrate the solution to several sequential testing games that arise from the parametric form of the signal in Example \ref{ex:h},
\begin{equation}\label{eqn:h.numeric.example}
	\mathbf{h}(t,j,F^j_\mu(t))=(j-1/2)\left(\lambda_0+\lambda_1F^j_\mu(t)\right),
\end{equation}
which leads to the volatility
 \begin{equation}\norm{\Delta \bh_\mu(t)}= \lambda_0+\lambda_1 (F^1_\mu(t) + F^0_\mu(t)).
 \end{equation}
Here we will restrict our attention to the homogeneous agent setting for simplicity and so the index $i$ is suppressed. 
We will primarily focus on the cross entropy loss (Example \ref{ex:g-cross-entropy}) due to its prevalence in machine learning applications but we also include a brief discussion on the classic loss.

As noted in Example \ref{ex:h}, this class of problems is rich enough to allow for the investigation of both \textit{preemption} and \textit{war of attrition} games. The parameter $\lambda_0$ plays the role of a baseline signal and $\lambda_1$ governs the strength of the interaction with the agent population. Since we are interested in investigating the effects of this interaction, this is the parameter we will modulate in \eqref{eqn:h.numeric.example}. Thus, we will fix the prior $\pi=1/2$ and the global parameters $\lambda_0=1$, $c=0.1$, and $T=5$. We noted in Example \ref{ex:h} that for our assumptions to be satisfied we require $\lambda_1>-\lambda_0$, so this will be the feasible range we investigate. Figure \ref{fig:pi.process.and.equil.bdys} (left panel) visualizes the trajectory of the posterior probability process when $\lambda_1=0$ for reference.

Since the state space for the representative agent problem is low dimensional, the solution remains amenable to conventional numerical techniques. For our purposes, we solve the sequential testing game by first solving a discrete time approximation of the agent problem using dynamic programming and then performing a fixed point interaction in the space of input measures until convergence.

\subsection{Cross entropy loss}

\begin{figure}[h]
	\begin{center}
		\includegraphics[width=0.49\linewidth]{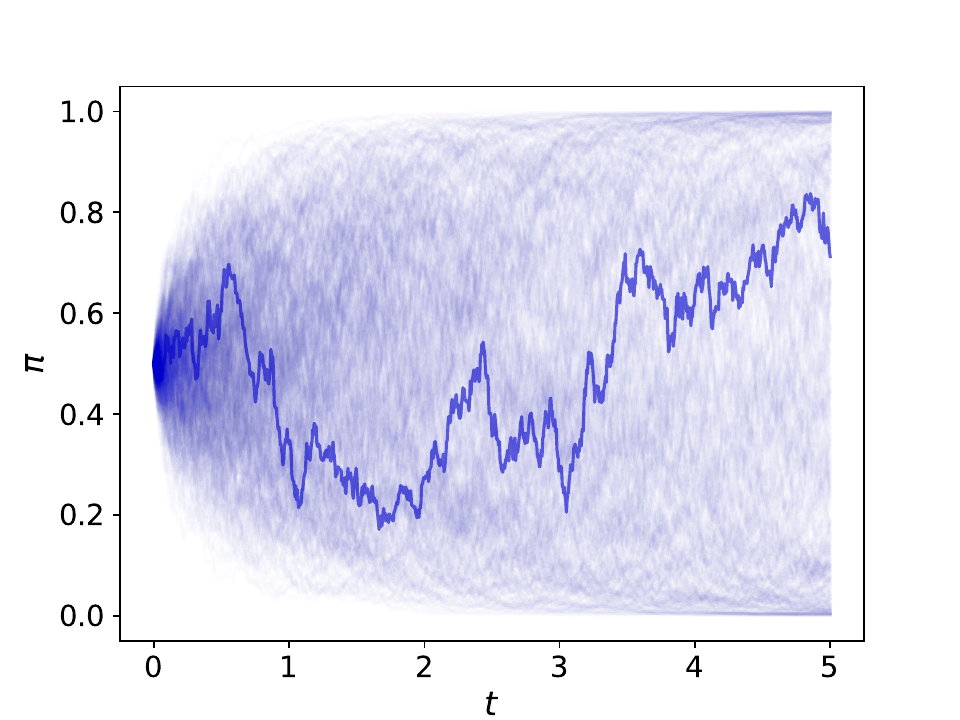}
		\includegraphics[width=0.49\linewidth]{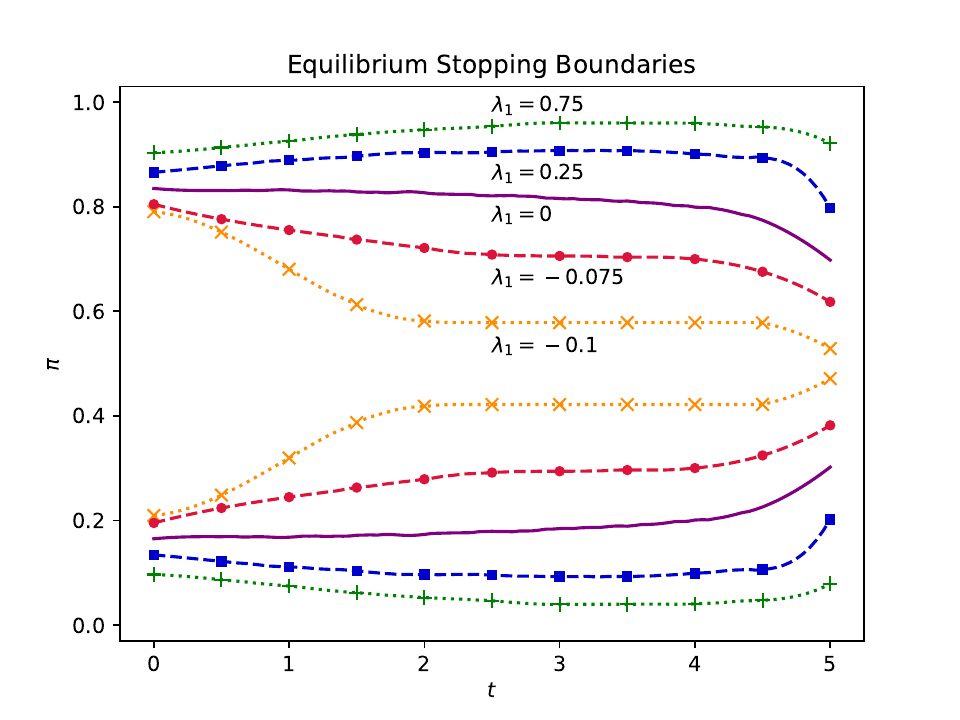}
	\end{center}
	\caption{$\Pi$ process trajectories in the non-interactive setting $(\lambda_1=0)$ (left panel) and equilibrium stopping boundaries (right panel).}
	\label{fig:pi.process.and.equil.bdys}
\end{figure}

Figure \ref{fig:pi.process.and.equil.bdys} (right panel) illustrates the optimal stopping boundaries for different values of $\lambda_1$ in equilibrium when the cross entropy loss is used. We observe that as $\lambda_1$ decreases the continuation region gets smaller. We also see that the optimal stopping boundaries are decreasing in time when $\lambda_1<0$. Indeed, this should be the case since this implies that the game is one of preemption and the volatility is decreasing in time. It is not hard to see that the proof of time monotonicity in Proposition \ref{prop:mon.and.lip.V}(iii) for the classic problem generalizes readily to the cross entropy loss (see the arguments of \cite[Proposition 4.2]{ekstrom2015bayesian}).

Figure \ref{fig:stop.time.and.values} (left panel) shows the corresponding stopping time distributions arising in equilibrium. Since we have enforced $\pi=1/2$ and the boundaries are symmetric, the conditional stopping time distributions are the same for each $\theta$ and coincide with the stopping time CDF. Interestingly, we observe an ordering on the stopping time distributions. In particular, as $\lambda_1$ decreases the stopping times appear to decrease in (first) stochastic order. Note that this is not immediate since a decrease in $\lambda_1$ decreases \textit{both} the continuation region size and the volatility (holding the input measure constant). These have conflicting effects as unilaterally decreasing the continuation region size makes the first exit time smaller, while unilaterally decreasing the volatility makes the first exit time larger. Evidently, here the former effect dominates in equilibrium.

\begin{figure}[h]
	\begin{center}
		\includegraphics[width=0.49\linewidth]{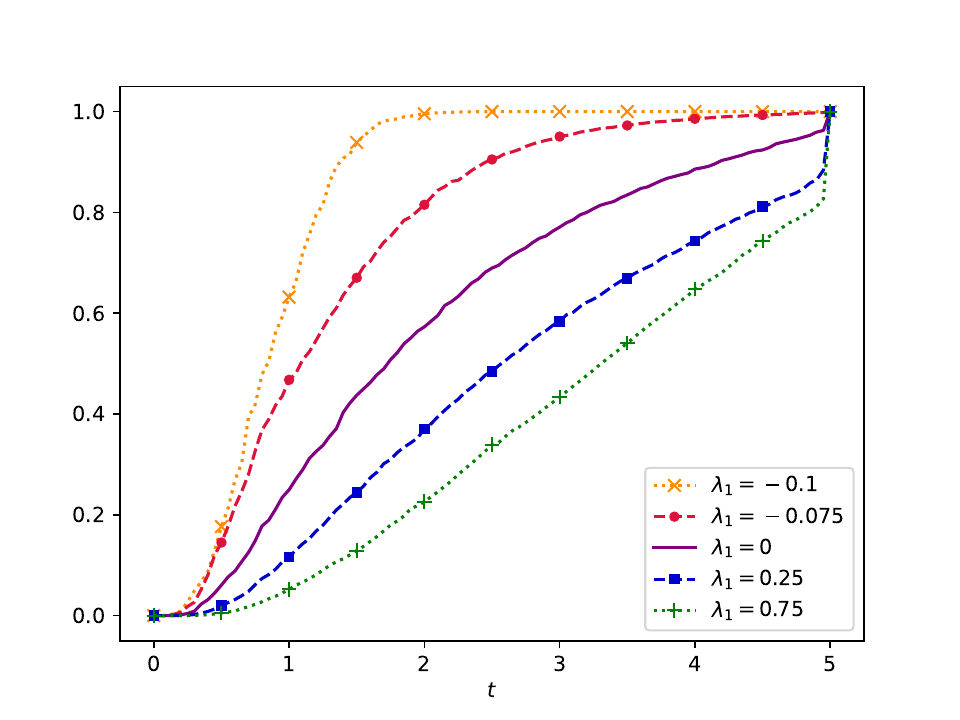}
		\includegraphics[width=0.49\linewidth]{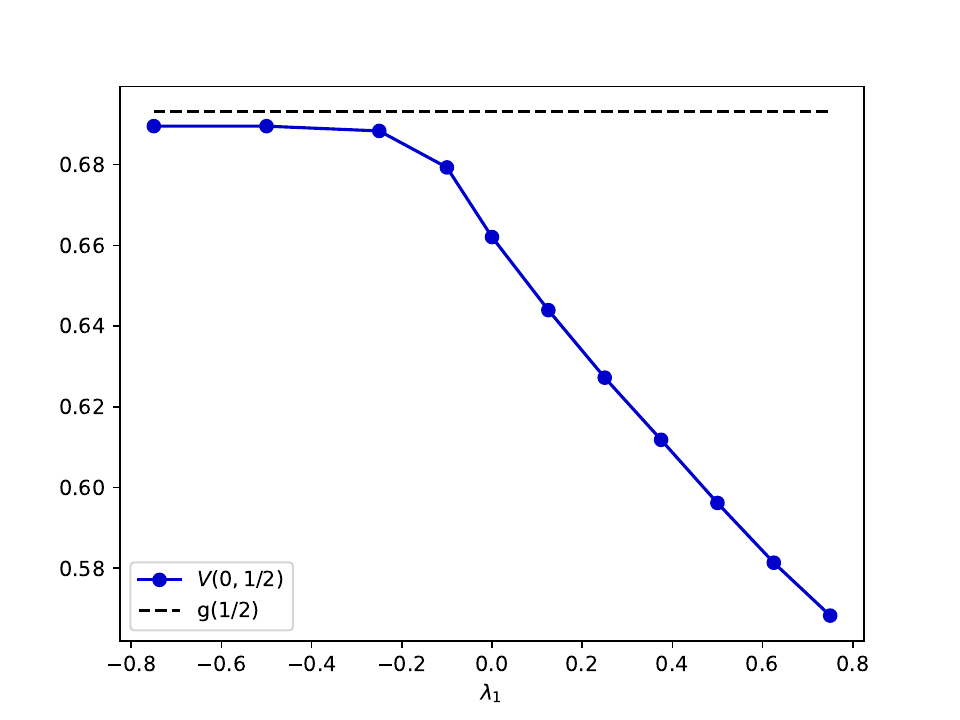}
	\end{center}
	\caption{Equilibrium stopping time CDFs (left panel) and problem values (right panel).}
	\label{fig:stop.time.and.values}
\end{figure}

As the continuation region and the value function are linked, the monotonicity of the continuation regions is also reflected in Figure \ref{fig:stop.time.and.values} (right panel). Here the value function is clearly seen to decrease as $\lambda_1$ increases. However, we see that there appears to be a difference in how the problem value behaves for preemption and war of attrition games. For positive values of $\lambda_1$, when $\lambda_1$ increases we see a relatively rapid decline in the problem value. On the other hand, for negative values of $\lambda_1$, we see a slow rate of increase in the problem value and a trend toward $g(1/2)$ as $\lambda_1\downarrow - \lambda_0$.

It is worth noting that we can a priori expect an ordering of the continuation regions across the two regimes $\lambda_1>0$ (war of attrition) and $\lambda_1<0$ (preemption). In equilibrium, the value function for $\lambda_1>0$ must be smaller (and hence continuation region is larger) than when $\lambda_1<0$. This follows as a consequence of Proposition \ref{prop:value.func} since we have a clear ordering of the volatility coefficients. On the other hand, if we are comparing two preemption or war of attrition games the ordering of the value function with respect to $\lambda_1$ is no longer immediate. This is since the equilibrium measure comes into the problem with the same sign and while the interaction parameter can be ordered there is no guarantee in advance that the induced equilibrium measures will have the same ordering. Nonetheless, we do observe, at least numerically, this ordering here.

\subsection{Classic loss}

We now briefly touch on the classic problem. In Figure \ref{fig:classic.prob} we illustrate the (numerical) solution in the preemption regime of Section \ref{sec:classic}. We use the same global parameters as in the previous figures, but choose the interaction parameters $\lambda_1\in\{-0.5,-0.25,0\}$. For the classic penalty function $g(\pi)=a_1\pi\wedge a_2(1-\pi)$, we choose the coefficients $a_1=3$ and $a_2=1.5$.

\begin{figure}[h]
	\begin{center}
		\includegraphics[width=0.49\linewidth]{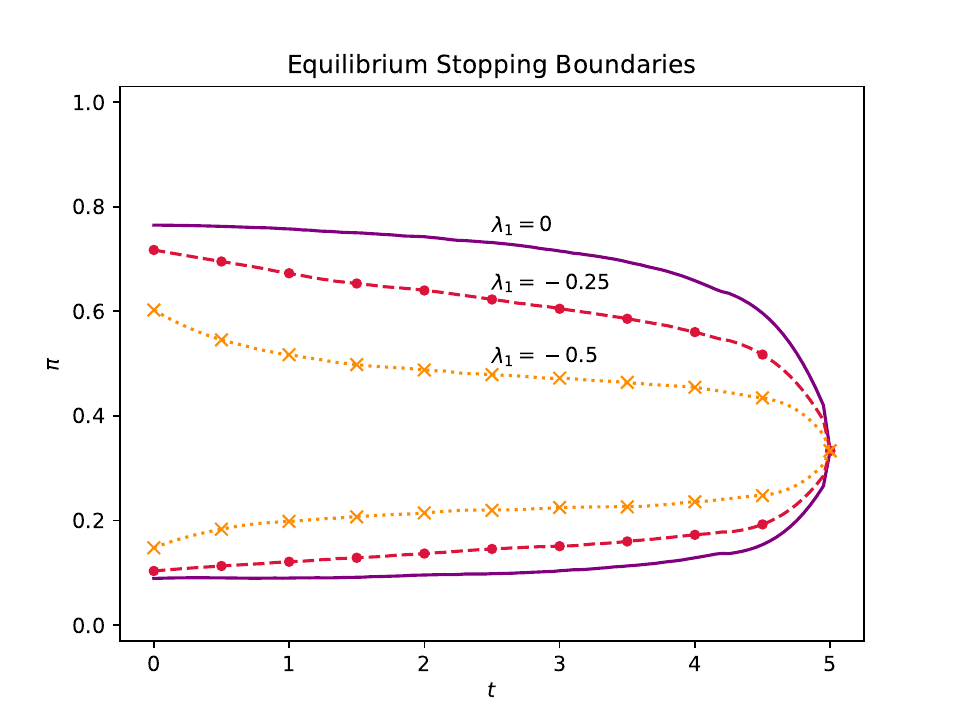}
		\includegraphics[width=0.49\linewidth]{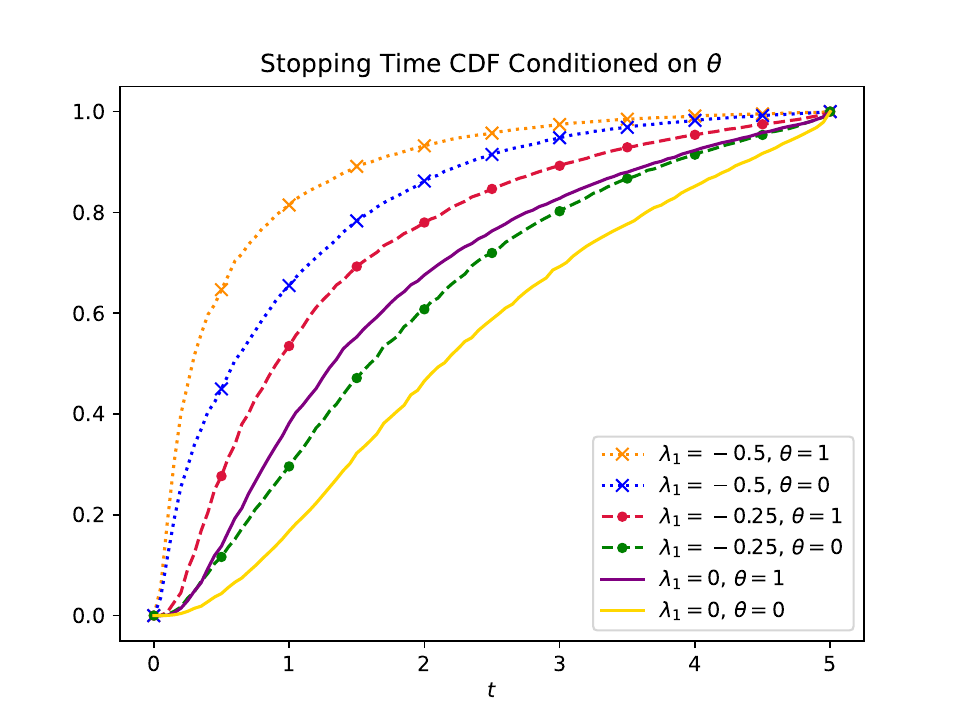}
	\end{center}
	\caption{Asymmetric classic problem continuation regions (left panel) and conditional stopping time distributions (right panel) in equilibrium. In the latter plot, the line styles distinguish between the choices of $\lambda_1$.}
	\label{fig:classic.prob}
\end{figure}


Figure \ref{fig:classic.prob} (left panel) illustrates the continuation region corresponding to these parameters and Figure \ref{fig:classic.prob} (right panel) visualizes the conditional stopping time distributions. A distinct feature here is the asymmetry of the penalty which we do not presently handle in the soft-classification setting. We see that the resulting continuation region is also asymmetric and that the conditional stopping time distribution (for prior $\pi=1/2$) is now different under each state of nature. We note again the apparent monotonicity of the stopping time distributions in $\lambda_1$ which also appeared in the cross entropy analysis.

\section{Conclusion}\label{sec:mfg.conclusion}

In this paper we have formulated a tractable mean field game of optimal stopping related to the classic Bayesian sequential testing of the drift of a Brownian motion. Our setting can be thought of as a simple example of interaction through information sharing. At a high level, this work contributes to the literature on optimal stopping mean field games by introducing both a common unobserved noise and agent learning into the game design. These additional elements require that we preserve the information structure available to each agent in the analysis. Our approach allows us to characterize the solution, establish the existence of a strong mean field equilibrium under suitable conditions, and illustrate the equilibria that arise. Based on our investigation, we believe that there are several potential avenues for future research which we describe below.

\begin{enumerate}
	\item[(i)] While we have solved for an equilibrium with a continuum of agents, it is desirable to study the associated $N$-player game (see \cite{nutz2020convergence} which conducts such a study for a different optimal stopping game). Establishing the existence of $\epsilon$-Nash equilibria and the consistency of the $N$-player formulation in our context appears to be a challenging open problem due to the presence of information sharing.
	\item[(ii)] An extension to multiple hypothesis testing problems (c.f. \cite{ekstrom2022multi,zhitlukhin2011bayesian}) is also natural, and will likely require new mathematical techniques to deal with the higher dimensional state processes. Moreover, potential generalizations of the existing setting to different structures for the signal process may also be of interest (see \cite{peskir2000sequential} which studies a sequential testing problem for the intensity of a Poisson process).
	\item[(iii)] There are also possible extensions to filtering problems with a continuous state space for the common noise. A famous related problem of this type is the quickest detection problem for Brownian motions (c.f. \cite{shiryaev1965some,shiryaev1963optimum}). In this formulation an agent tries to detect a common disorder time, $\theta$,
	typically
	representing the occurrence of some financial event, using the private signal $X_t=\int_0^t\mathds{1}_{\theta\leq s}ds+W_t$. 
	Embedding this in a game setting introduces some new challenges since the common noise is no longer binary and consequently, the filtering equations become much more complicated. Establishing the existence, and possible uniqueness, of Nash equilibria here would be an important result that lays the groundwork for more complicated unobserved common noise structures in filtering games with stopping.
	\item[(iv)] More broadly, applications of similar mean field game designs to concrete examples of interaction in financial markets with filtering, stopping, and unobserved common noises (e.g. indicators of economic health) is an important research direction that may be of interest to central banks and regulators.
\end{enumerate}

\pagebreak

\begin{appendix}

\section{Conditional Exact Law of Large Numbers}\label{sec:CELLN}




The notion of essential pairwise independence and the exact law of large numbers in the framework of a Fubini extension were introduced in \cite{sun2006exact}, generalized to the conditional setting in \cite{qiao2016conditional} and brought to the attention of the mean field game and mathematical finance community by \cite{nutz2018mean}. We include the results in the conditional setting
here for the readers' convenience.

\begin{definition}[{\cite[Definition 2]{qiao2016conditional}}\label{def:essPairInd}]
A family $X=(X_{i})_{i\in I}$ of random variables defined on a probability space $(\Omega,\cF,\bbP)$ is said to be
essentially pairwise conditionally independent given a $\sigma$-field $\mathcal{G}\subseteq\mathcal{F}$, if for $\lambda$-almost all $i\in I$, $X_{i}$ is conditionally independent of $X_{j}$ given $\mathcal{G}$ for $\lambda$-almost all $j\in I$.
\end{definition}

It is shown in \cite[Proposition 2.1]{sun2006exact} that the usual product $(I\times \Omega,\mathcal{I}\otimes\mathcal{F},\lambda\otimes\mathbb{P})$ is not large enough to support any nontrivial essentially pairwise independent family. One needs to define $X$ on a so-called \emph{rich Fubini extension} of the product space, denoted by $(I\times \Omega,\cI\boxtimes\cF,\lambda\boxtimes\bbP)$; we refer to \cite{sun2006exact} for its formal definition and existence. The name ``Fubini" comes form the fact that any integrable function in this extended space satisfies the Fubini theorem, which forms the basis for the exact law of large numbers. The following conditional version of it is taken from \cite[Corollary~2]{qiao2016conditional}.

\begin{proposition}[Conditional Exact Law of Large Numbers]\label{prop:CELLN}

Let $\cG\subseteq \cF$ be a countably generated $\sigma$-field and let $X$ be a real-valued integrable function on the rich Fubini extension $(I\times \Omega,\cI\boxtimes\cF,\lambda\boxtimes\bbP)$. If $X(i,\cdot)$, $i\in I$ 
are essentially pairwise conditionally independent given $\mathcal{G}$, then $\int X(\cdot,\omega)\,d\lambda=\int \bbE^{\lambda\boxtimes\bbP}[X|\mathcal{I}\otimes\mathcal{G}](\cdot,\omega)\,d\lambda$ for $\mathbb{P}$-almost all $\omega\in\Omega$. 
\end{proposition}

%
%
%
%
%

\section{Infinite Horizon Problem}\label{sec:inf.horizon}


In this appendix, we collect some results for the infinite horizon problem with constant volatility. That is, we suppose $T=\infty$ and $\Delta\bh_\mu(t)$ is constant.

\begin{proposition} \cite[Proposition 2.1]{irle2004solving}
	Under Assumption \ref{ass:g} the value function $V(x)$ is time-homogeneous, $|V(x)|<\infty$ for all $x\in(0,1)$, and the first exit time
	$\tau^*=\inf\{t\geq0:\Pi_t^{0,\pi}\not\in \mathscr{C}\}$
	from the continuation region $\mathscr{C}:=\{\pi\in(0,1):V(\pi)<g(\pi)\}$
	is an optimal stopping time.
\end{proposition}

The following gives a boundary characterization for the continuation region of the infinite horizon problem.

\begin{proposition}\label{prop:inf.horizon.bdy}
Under Assumption \ref{ass:g} there exist boundaries $0< b^*< 1/2<B^*< 1$ for $t\in[0,T)$ such that
	$\mathscr{C}=\{\pi\in(0,1): b^*<\pi<B^*\}$. 
\end{proposition}

\begin{proof}
	The fact that there are boundaries satisfying $0\leq b^*<1/2<B^*\leq 1$ follows from \cite[Proposition 3.1]{irle2004solving} and an identical argument to Proposition \ref{prop:cont.region.nonempty}. To see that the inequalities are all strict (i.e. $0<b^*$ and $B^*<1$) we can directly borrow the arguments of \cite[Proposition 4.8]{ekstrom2015bayesian}. 
\end{proof}

\section{Time Changed Problem}\label{sec:time-change}


We have seen in Section \ref{sec:post.prob.and.log.likelihood} that the space-dependence in the diffusion coefficient of $\Pi_t$ can be removed by the space change \eqref{eq:defn-L}. We can further remove the time-dependence by a time change, and formulate an equivalent time (and space) changed version of our optimal stopping problem. The main motivation for this formulation is that many auxiliary results that are critical to the proof of Proposition \ref{prop:bdy.unif.local.lip} are easier to obtain for the time changed problem than for the original problem. 

To simplify the presentation, we will borrow the notional convention of Proposition \ref{prop:value.func}(ii) by writing $\eta_\mu(t)=\|\Delta\bh_\mu(t)\|$ in this section.

\subsection{Time changed log likelihood process}

Consider the clock $\alpha_\mu(t):=\int_0^t \eta^2_\mu(s)ds$. It is strictly increasing since $\eta_\mu(t) \geq \mathfrak{h}>0$. Define the time change as $\zeta_\mu(t):=\alpha^{-1}_\mu(t)$. We have by the Inverse Function Theorem that
\begin{equation}\label{eq:time.change.der}\zeta'_\mu(u)=\frac{1}{\eta^2_\mu(\zeta_\mu(u))}.
\end{equation}
From our assumptions on $\mathbf{h}_\mu$, it is easy to check that $\zeta_\mu$ and $\zeta_\mu^{-1}=\alpha_\mu$ are both twice continuously differentiable, and have bounded first and second derivatives. Moreover, these bounds can be made independent of $\mu$.

Let $\widehat{L}_t:=L_{\zeta_\mu(t)}$ be the time changed log likelihood process. One can show that $\widehat{L}$ has the dynamics:
\begin{equation}\label{eq:L.hat.dynamics} 
	d\widehat{L}_{t}=a(\widehat{L}_t) dt+d\widehat{W}_t,
\end{equation}
where 
\begin{equation}\label{eq:a}
	a(l):=\frac{e^l-1}{2(e^l+1)}
\end{equation}
and $\widehat{W}_t:=\int_0^{\zeta(t)} \Delta \bh_\mu^\top (s) d\overline{\bW}_s$ is a Brownian motion with respect to the time changed filtration $\widehat \bbF = (\cF^{\overline{\bW}}_{\zeta_\mu(t)})_{t\in [0, \zeta^{-1}_\mu(T)]}$ by L\'evy's Characterization.
%
The infinitesimal generator of $\widehat{L}$, given by
\[\widehat{\mathcal{L}}:=a(\cdot)\partial_l+\frac{1}{2}\partial_{ll},\]
is related to $\cA$ (defined in \eqref{eq:operatorA}) and the sigmoid function $S$ (see \eqref{eq:defn-L}) via 
\begin{align}
	\widehat{\cL}(f\circ S)(l) &= a(l) f'(S(l)) S'(l) + \frac{1}{2}\left[f''(S(l))(S'(l))^2 + f'(S(l))S''(l)\right] \nonumber\\ 
	& = (\cA f)(S(l)), \quad f\in C^2(0,1). \label{eqn:relate.generator.A}
\end{align}
In the derivation of \eqref{eqn:relate.generator.A}, we have used the relations $S'(l) = S(l)(1-S(l))$ and $a(l)S'(l)+S''(l)/2 = 0$.

In addition to being time-homogenous, $\widehat{L}_t$ 
admits a simple \textit{differentiable flow}.

%

\begin{lemma}
	The flow $l\mapsto \widehat{L}_t^{u,l}(\omega)$ for $\omega\in \Omega$ is continuously differentiable. The derivative is given by
	\begin{equation}\label{eq:flow}
		\partial \widehat{L}_t^{u,l}=\exp(\int_u^ta'(\widehat{L}_s^{u,l})ds),
	\end{equation}
	and satisfies $1\leq \partial \widehat{L}_t^{u,l}\leq \exp((t-u)/4)$ for all $l\in\mathbb{R}$.
\end{lemma}

\begin{proof} 
	By \cite[Theorem V.49]{protter2005stochastic}, $\partial \widehat{L}_t^{u,l}$ is well-defined, continuous, and satisfies the equation
	\[\partial \widehat{L}^{u,l}_t=1+\int_u^ta'(\widehat{L}_s^{u,l})\partial \widehat{L}^{u,l}_sds.\]
	The solution is given by \eqref{eq:flow}. The bound on $\partial \widehat{L}_t^{u,l}$ follows from the fact that $0<a'(l)=e^l/(1+e^l)^2\leq 1/4$. 
\end{proof}

\subsection{Time and space changed value function}

Define the time changed value function by
\begin{equation}\label{eq:value.function.hat.L}
	\widehat{V}(u,l) = \widehat{V}(u,l;\mu)=\inf_{\widehat{\tau}\in\mathcal{T}_{u,\widehat{T}}^{\widehat{W}}}\bbE\left[c\int_u^{\widehat{\tau}}\frac{1}{\eta^2_\mu(\zeta_\mu(s))}ds+\widehat{g}(\widehat{L}_{\widehat{\tau}}^{u,l})\right],
\end{equation}
where $\widehat{g}:=g\circ S$, $\widehat T=\widehat{T}_\mu:=\zeta^{-1}_\mu(T)$,\footnote{We will suppress the $\mu$-dependence unless it is necessary to be explicit.}  and $\cT^{\widehat{W}}_{u,\widehat T}$ is the set of $\mathbb{F}^{\widehat{W}}$-stopping times taking values in $[u, \widehat T]$. 
It is not hard to see that this is related to $V$.

\begin{lemma} For all $(t,\pi)\in[0,T)\times(0,1)$
\begin{equation}\label{eq:relating.value.functions}
	V(t,\pi;\mu)=\widehat{V}(\zeta^{-1}_\mu(t),S^{-1}(\pi);\mu).
\end{equation}
\end{lemma}
\begin{proof}
	From \eqref{eqn:value.function.L} we have that $V(t,\pi;\mu)=\tilde{V}(t,S^{-1}(\pi);\mu)$. Letting $u=\zeta^{-1}_\mu(t)$, $l=S^{-1}(\pi)$, and observing that $\tilde{V}$ (resp. $\widehat{V}$) can be posed over $\mathbb{F}^{L}$ (resp. $\mathbb{F}^{\widehat{L}})$) stopping times (see e.g. Remark \ref{rmk:Vhat-filtration} below) we obtain:
	\begin{align*}
		\tilde V(t,l;\mu)&=\inf_{\tau\in \cT_{t,T}^{L}}\mathbb{E}\left[c(\tau-t)+\widehat{g}(L_\tau^{t,l})\right]\\
		&=\inf_{\tau\in \cT_{t,T}^{L}}\mathbb{E}\left[c(\zeta_\mu\circ\zeta^{-1}_\mu(\tau)-\zeta_\mu\circ\zeta^{-1}_\mu(t))+\widehat{g}(L_{\zeta_\mu\circ\zeta^{-1}_\mu(\tau)}^{\zeta_\mu\circ\zeta^{-1}_\mu(t),l})\right]\\
		&=\inf_{\tau\in \cT_{t,T}^{L}}\mathbb{E}\left[c(\zeta_\mu\circ\zeta^{-1}_\mu(\tau)-\zeta_\mu\circ\zeta^{-1}_\mu(t))+\widehat{g}(\widehat{L}_{\zeta^{-1}_\mu(\tau)}^{\zeta^{-1}_\mu(t),l})\right]\\
		&=\inf_{\widehat{\tau}\in \cT_{u,\widehat{T}}^{\widehat{L}}}\mathbb{E}\left[c(\zeta_\mu(\widehat{\tau})-\zeta_\mu(u))+\widehat{g}(\widehat{L}_{\widehat{\tau}}^{u,l})\right]\\
		&=\inf_{\widehat{\tau}\in \cT_{u,\widehat{T}}^{\widehat{L}}}\mathbb{E}\left[c\int_u^{\widehat{\tau}}\frac{1}{\eta^2_\mu(\zeta_\mu(s))}ds+\widehat{g}(\widehat{L}_{\widehat{\tau}}^{u,l})\right]\\
		&=\widehat{V}(u,l;\mu)
	\end{align*}
where we have used $\widehat{L}_u:=L_{\zeta_\mu(u)}$, $\mathcal{F}_u^{\widehat{L}}=\mathcal{F}_{\zeta_\mu(u)}^L$, and \eqref{eq:time.change.der}. The conclusion then follows.
\end{proof}

\begin{remark}\label{rmk:Vhat-filtration}
	We have seen in Section~\ref{sec:value_fun} that whether we pose the original problem over $\bbF^{\overline{\bW}}$-stopping times or $\bbF^\Pi$-stopping times does not change the optimal value. The set of $\bbF^{\overline{\bW}}$-stopping times (resp.\ $\bbF^\Pi$-stopping times) is in one-to-one correspondence with the set of $\widehat \bbF$-stopping times (resp.\ $\bbF^{\widehat L}$-stopping times) for the time changed problem. Since the filtration $\bbF^{\widehat W}$ is sandwiched between $\bbF^{\widehat L}$ and $\widehat \bbF$, we can define $\widehat V$ using $\bbF^{\widehat{W}}$-stopping times as in \eqref{eq:value.function.hat.L} without affecting the equivalence \eqref{eq:relating.value.functions}. Moreover, since $\widehat{V}$ only depends on the distributional property of $\widehat{L}$ and $\widehat{W}$, the Brownian motion $\widehat{W}$ in the problem definition can be taken to be a generic Brownian motion independent of $\mu$.
\end{remark}

The equivalence between $\widehat{V}$ and $V$ readily implies the following.
\begin{corollary}\label{cor:inherited.properties.hat.V.problem} The following hold under Assumption \ref{ass:g}.
	\begin{itemize}
		\item[(i)] $\widehat{V}$ is (jointly) continuous on $[0, \widehat{T}]\times \bbR$;
		\item[(ii)] The continuation region corresponding to $\widehat{V}$ is characterized by two stopping boundaries 
		\[\gamma:=S^{-1}\circ b\circ\zeta_\mu \quad \text{and} \quad \Gamma:=S^{-1}\circ B\circ\zeta_\mu;\] 
		\item[(iii)] The smallest optimal stopping time for $\widehat{V}(u,l)$ is given by
		\begin{equation}\label{eq:opt.stopping.time.L.hat}
			\widehat{\tau}^*(u,l)=\inf\{s\geq u:L^{u,l}_s\not\in(\gamma(s),\Gamma(s))\}\wedge \widehat{T};
		\end{equation}
		\item[(iv)] There are constants $\gamma_*$, $\gamma^*$, $\Gamma_*$ and $\Gamma^*$ (independent of $\mu$) such that
		\begin{equation}\label{eq:L.hat.bdy.bounds}
			-\infty<\gamma_*<\gamma(t)<\gamma^*<0<\Gamma_*<\Gamma(t)<\Gamma^*<\infty.
		\end{equation}
		\item[(v)] $\widehat{V}$ is $C^{1,2}$ in $\widehat\sC:=\{(u,l)\in [0, \widehat{T}) \times \bbR: \gamma(u)<l<\Gamma(u)\}$, and satisfies:
		\begin{equation}
			\partial_u\widehat{V}(u,l)+\widehat{\mathcal{L}}\widehat{V}(u,l)=\frac{-c}{\eta_\mu(\zeta_\mu(u))^2}, \quad (u,l)\in \widehat\sC.
		\end{equation}
	\end{itemize}
\end{corollary}
\begin{proof} (i) follows from \eqref{eq:relating.value.functions} and the continuity of $V$. (ii) follows since $V(\zeta_\mu(u),S(l))<g(S(l))$ if and only if $\widehat{V}(u,l)<\widehat{g}(l)$. By the characterization of the continuation region for $V$, the former inequality says $b(\zeta_\mu(u))<S(l)<B(\zeta_\mu(u))$. (iii) follows immediately from (ii). (iv) follows from (ii) and Proposition \ref{prop:bounds.on.bdys}. Finally, (v) is a direct consequence of Proposition \ref{prop:C1.interior.cont.reg.} and \eqref{eq:relating.value.functions}. It can also be shown directly by the same arguments leading to Proposition \ref{prop:C1.interior.cont.reg.}.
\end{proof}

\subsection{Probabilistic representation of derivatives}\label{sec:prob.representation.subsection}

In this section we establish probabilistic representations and estimates on the derivatives of the time changed value function. These results are drawn primarily from the recent work \cite{de2019lipschitz}. Henceforth, in this appendix we will define $\widehat{\tau}^*=\widehat{\tau}^*(u,l)\in \mathcal{T}_{u,\widehat{T}}^{\widehat{W}}$ as in Corollary \ref{cor:inherited.properties.hat.V.problem} to be the smallest optimal stopping time for our problem with initial data\footnote{To avoid ambiguity, when $\widehat{\tau}^*$ appears under an expectation (and in particular, when it appears alone) we will make explicit the conditioning on the initial data $(t,u)$ by writing $\mathbb{E}_{t,u}[\cdot]$. In these instances, the superscript on $\widehat{L}^{u,l}$ is dropped.} $(u,l)\in[0,\widehat{T})\times \mathbb{R}$.

\begin{lemma}\label{lem:prob.representation} Under Assumption \ref{ass:g}, the derivatives of $\widehat{V}$ satisfy the following, wherever they exist:
	\begin{enumerate}
		\item[(i)] $\partial_l\widehat{V}(u,l)=\bbE_{u,l}\left[\widehat{g}'(\widehat{L}_{\widehat{\tau}^*})\partial\widehat{L}_{\widehat{\tau}^*}\right]$;
		\item[(ii)] $\partial_u\widehat{V}(u,l)\geq\bbE_{u,l}\left[\int_u^{\widehat{\tau}^*} -\frac{2\eta_\mu'(\zeta_\mu(s))}{\eta_\mu(\zeta_\mu(s))^5}ds\right]$;
		\item[(iii)] There exists a constant $K>0$ (independent of $\mu$) such that
		\[\partial_u\widehat{V}(u,l)\leq\bbE_{u,l}\left[\int_u^{\widehat{\tau}^*} -\frac{2\eta_\mu'(\zeta_\mu(s))}{\eta_\mu(\zeta_\mu(s))^5}ds\right]+K\mathbb{P}(\widehat{\tau}^*=U).\]
		
	\end{enumerate}
\end{lemma}

\begin{proof}
	The Markov property and time-homogeneity of the state process $\widehat{L}$ allow us to rewrite our time changed value function as
	\begin{equation}\label{eqn:equiv.form.hat.V}
		\widehat{V}(u,l)= - \sup_{\widehat{\tau}\in\mathcal{T}_{0,\widehat{T}-u}^{\widehat{W}}}\bbE\left[-c\int_0^{\widehat{\tau}} \frac{1}{\eta^2_\mu(\zeta_\mu(u+s))}ds-\widehat{g}(\widehat{L}_{\widehat{\tau}}^{0,l})\right],
	\end{equation}
	which falls exactly into the setting of \cite{de2019lipschitz}.
	We then apply \cite[Theorem 3.1]{de2019lipschitz} to get (i) and (iii), and \cite[Corollary 3.3]{de2019lipschitz} to get (ii). (iii) additionally makes use of the bound $|\eta^{-2}_\mu|\leq \mathfrak{h}^{-2}$ and the boundedness of $\widehat{\mathcal{L}}\widehat{g}$ which holds by \eqref{eqn:relate.generator.A} and Remark \ref{rmk:consequences.of.assumptions}. 
\end{proof}

From our assumptions on $\mathbf{h}_\mu$ it is clear that $\eta_\mu(\cdot)$ is continuously differentiable with
first derivative bounded by some  $K>0$ (independent of $\mu$) on $[0,T]$. By the bound $\left|-2\eta_\mu'\eta_\mu^{-5}\right|\leq 2K\mathfrak{h}^{-5}$ and Markov's inequality we obtain the following useful corollary.

\begin{corollary}\label{cor:bd.time.der.hat.V}
	There exists a constant $C>0$ (independent of $\mu$) such that
	\[|\partial_u \widehat{V}(u,l)|\leq \frac{C}{\widehat{T}-u}\mathbb{E}_{u,l}\left[\widehat{\tau}^*-u\right].\]
\end{corollary}

With these results in hand, we turn to studying the defining function of the continuation region, $\widehat{V}-\widehat{g}$. By Assumption \ref{ass:g}, $V$ and $g$ are symmetric about $\pi=1/2$ which readily gives
\begin{equation}\label{eqn:der.vanishes}
	\partial_l\widehat{V}(u,0)-\widehat{g}'(0)=0 \ \ \forall u\in[0,\widehat{T}).
\end{equation}
Moreover, by \eqref{eqn:relate.generator.A} we obtain
\[\partial_\pi\mathcal{A}g(S(l))S'(l)=\partial_l\widehat{\mathcal{L}}\widehat{g}(l).\]
Since $S'(l)>0$ for all $l\in\mathbb{R}$ the (strict) signs of $\partial_\pi\mathcal{A}g$ and $\partial_l\widehat{L}\widehat{g}$ can be compared. It follows from Assumption \ref{ass:g} and $S(0)=1/2$ that

\begin{equation}\label{eqn:sign.der.generator} \partial_l\widehat{\mathcal{L}}\widehat{g}(l)<0, \ \   l\in(-\infty,0) \ \ \mathrm{and} \ \ \partial_l\widehat{\mathcal{L}}\widehat{g}(l)>0, \ \  l\in(0,\infty).
\end{equation}
The next result uses these observations to determine the sign of $\partial_l(\widehat{V}-\widehat{g})$.

\begin{lemma}\label{lem:sign.der.hat.V}
	For each $u\in[0,\widehat{T})$ it holds that $\partial_l(\widehat{V}-\widehat{g})(u,\cdot)<0$ on $(\gamma(u),0)$ and $\partial_l (\widehat{V}-\widehat{g})(u,\cdot)>0$ on $(0,\Gamma(u))$.
\end{lemma}

\begin{proof}
	Fix $l\in(\gamma(u),0)$. We use the representation of the derivative from Lemma \ref{lem:prob.representation}(i) and Dynkin's formula to get for $\widehat{\tau}^*=\widehat{\tau}^*(u,l)$:
	\begin{align}
		\partial_l(\widehat{V}-\widehat{g})(u,l)&=\mathbb{E}_{u,l}\left[\widehat{g}'(\widehat{L}_{\widehat{\tau}^*})\partial\widehat{L}_{\widehat{\tau}^*}-\widehat{g}'(l)\right] \nonumber \\
		&=\mathbb{E}_{u,l}\left[\int_u^{\widehat{\tau}^*} e^{\int_u^va'(\widehat{L}_r)dr}\left(a'\widehat{g}'+\widehat{\mathcal{L}}(\widehat{g}')\right)(\widehat{L}_v)dv\right] \nonumber \\
		&=\mathbb{E}_{u,l}\left[\int_u^{\widehat{\tau}^*} e^{\int_u^va'(\widehat{L}_r)dr}\partial_l\widehat{\mathcal{L}}\widehat{g}(\widehat{L}_v)dv\right]. \label{eqn:diff.hatVg.derivative.rep}
	\end{align}
	The last equality can be verified by differentiating. Letting \[\tau_{0}(u,l):=\inf\{v\geq u:\widehat{L}^{u,l}_v\geq 0\}\wedge \widehat{T}\] we argue as in \cite[Theorem 4.3]{de2019lipschitz} (see Equations (4.16)-(4.18) of the proof) by the Tower and Strong Markov properties that
	\begin{align*}\partial_l(\widehat{V}-\widehat{g})(u,l)&=\mathbb{E}_{u,l}\bigg[\int_u^{\tau_0\wedge\widehat{\tau}^*} e^{\int_u^va'(\widehat{L}_r)dr}\partial_l\widehat{\mathcal{L}}\widehat{g}(\widehat{L}_v)dv\\
		&  +\mathds{1}_{\tau_0<\widehat{\tau}^*}e^{\int_u^{\tau_0}a'(\widehat{L}_r)dr}\mathbb{E}_{\tau_0,\widehat{L}_{\tau_0}}\bigg[\int_{\tau_0}^{\widehat{\tau}^*} e^{\int_{\tau_{0}}^va'(\widehat{L}_r)dr}\partial_l\widehat{\mathcal{L}}\widehat{g}(\widehat{L}_v)dv\bigg]\bigg]\\
		&=\mathbb{E}_{u,l}\bigg[\int_u^{\tau_0\wedge \widehat{\tau}^*}e^{\int_u^va'(\widehat{L}_r)dr}\partial_l\widehat{\mathcal{L}}\widehat{g}(\widehat{L}_v)dv\\
		&\quad \quad \quad \quad \quad \quad \quad \quad  \quad +\mathds{1}_{\tau_0<\widehat{\tau}^*}e^{\int_u^{\tau_0}a'(\widehat{L}_r)dr}\partial_l(\widehat{V}-\widehat{g})(\tau_0,\widehat{L}_{\tau_0})\bigg].
	\end{align*}
	Note that at $\tau_0$ either $\widehat{L}_{\tau_0}^{u,l}=0$ so $\partial_l(\widehat{V}-\widehat{g})(\tau_0,0)=0$ by \eqref{eqn:der.vanishes} or $\tau_0=\widehat{T}$ in which case $\widehat{V}(\widehat{T},\cdot)=\widehat{g}(\cdot)$ and so we still have $\partial_l(\widehat{V}-\widehat{g})(\widehat{T},\widehat{L}_{\tau_0}^{u,l})=0$. We conclude that
	\[\partial_l(\widehat{V}-\widehat{g})(u,l)=\mathbb{E}_{u,l}\left[\int_u^{\tau_0\wedge \widehat{\tau}^*}e^{\int_u^va'(\widehat{L}_r)dr}\partial_l\widehat{\mathcal{L}}\widehat{g}(\widehat{L}_v)dv\right].\]
	As $a'\geq0$ we have by definition of $\tau_0$ and \eqref{eqn:sign.der.generator} that:
	\[\partial_l(\widehat{V}-\widehat{g})(u,l)\leq\mathbb{E}_{u,l}\left[\int_u^{\tau_0\wedge \widehat{\tau}^*}\partial_l\widehat{\mathcal{L}}\widehat{g}(\widehat{L}_v)dv\right]<0.\]
	Note in the above we use that the continuation region is open and contains $l$, so there is an $\epsilon>0$ such that the $\epsilon$-ball $B_\epsilon((u,l))$ is contained in the continuation region. As $l<0$ we can conclude $\tau_0\wedge \widehat{\tau}^*-u>0$ and so the inequality is strict. The same argument gives the analogous result for the derivative on $(0,\Gamma(u))$.
\end{proof}

For the last result of this subsection, we show that there are points in the continuation region for which Lemma \ref{lem:sign.der.hat.V} can be strengthened to obtain an explicit bound on the derivative. In particular, the bound depends only on the original times $t,T$ and not at all on $\mu$.

\begin{lemma}\label{lem:strict.bd.mu.indep.der.Vhat}
	For any $\underline{l},\overline{l}$ satisfying $\gamma^*<\underline{l}<0$ and $0<\overline{l}<\Gamma_*$ there exist constants $\kappa>0$ and $\Delta>0$ (independent of $\mu$) such that 
	for all $t\in[0,T)$ (equivalently, all $u=\zeta_\mu^{-1}(t)\in[0,\widehat{T})$):
	\[\partial_l(\widehat{V}-\widehat{g})(\zeta^{-1}_\mu(t),\underline{l})\leq -\kappa\mathbb{E}_{0,\underline{l}}\left[\tau_{\underline{l}\pm\Delta}\wedge(T-t)\right]<0\]
	and
	\[\partial_l(\widehat{V}-\widehat{g})(\zeta^{-1}_\mu(t),\overline{l})\geq \kappa\mathbb{E}_{0,\overline{l}}\left[\tau_{\overline{l}\pm\Delta}\wedge(T-t)\right]>0,\] 
	where $\tau_{x\pm\Delta}$ is the first exit time of $\widehat{L}$ from the interval $(x-\Delta,x+\Delta)$.
\end{lemma}


\begin{proof}
	Without loss of generality suppose that we are treating the case of $\underline{l}$. 
	Arguing exactly as in Lemma \ref{lem:sign.der.hat.V} we get for \[\tau_{0}(u,\underline{l}):=\inf\{v\geq u:\widehat{L}^{u,\underline{l}}_v\geq 0\}\wedge \widehat{T}\] that
	\[\partial_l(\widehat{V}-\widehat{g})(u,\underline{l})\leq\mathbb{E}_{u,\underline{l}}\left[\int_u^{\tau_0\wedge \widehat{\tau}^*}\partial_l\widehat{\mathcal{L}}\widehat{g}(\widehat{L}_v)dv\right].\]
	Choose $\Delta>0$ to satisfy $\gamma^*<\underline{l}-\Delta<\underline{l}<\underline{l}+\Delta<0$  
	so that the spatial values $\underline{l} \pm \Delta$ 
	are in the continuation region for all times $u$. We can then define the first exit time from the strip $(\underline{l}-\Delta,\underline{l}+\Delta)$ as
	\[\tau_{\underline{l}\pm \Delta}(u,\underline{l}):=\inf\{v\geq u:\widehat{L}_v^{u,\underline{l}}\not\in (\underline{l}-\Delta,\underline{l}+\Delta)\}.\]
	It is clear that $\tau_{\underline{l}\pm \Delta}\wedge \widehat{T}\leq \tau_0\wedge\widehat{\tau}^*$ and so
	\[\partial_l(\widehat{V}-\widehat{g})(u,\underline{l})\leq\mathbb{E}_{u,\underline{l}}\left[\int_u^{\tau_{\underline{l}\pm \Delta}\wedge \widehat{T}}\partial_l\widehat{\mathcal{L}}\widehat{g}(\widehat{L}_v)dv\right]\]
	as $\partial_l\widehat{\mathcal{L}}\widehat{g}<0$ on $(\gamma(u),0)$.
	Letting 
	\[\kappa:=\inf_{l\in [\underline{l}-\Delta,\underline{l}+\Delta]}-\partial_l\widehat{\mathcal{L}}\widehat{g}(l)>0\]
	(where the strict inequality holds by continuity of $\partial_l\widehat{\mathcal{L}}\widehat{g}$ and our assumption $\partial_l\widehat{\mathcal{L}}\widehat{g}<0$ on $(\gamma(u),0)\supset [\underline{l}-\Delta,\underline{l}+\Delta]$) we have
	\begin{align*}
		\partial_l(\widehat{V}-\widehat{g})(u,\underline{l})&\leq-\kappa\mathbb{E}_{u,\underline{l}}\left[\tau_{\underline{l}\pm \Delta}\wedge \widehat{T}-u\right]\\
		&=-\kappa\mathbb{E}_{u,\underline{l}}\left[(\tau_{\underline{l}\pm \Delta}(u,\underline{l})-u)\wedge(\widehat{T}-u)\right]<0.
	\end{align*}
	Then, we note that as $\widehat{L}$ is time homogenous \[\tau_{\underline{l}\pm \Delta}(u,\underline{l})-u\overset{d}{=}\inf\{v\geq0:\widehat{L}^{0,\underline{l}}_v\not\in(\underline{l}-\Delta,\underline{l}+\Delta)\}.\]
	and consequently,
	\begin{equation}\label{eqn:intermediate.bd.lem.der.hat.V}
		\partial_l(\widehat{V}-\widehat{g})(u,\underline{l})\leq -\kappa\mathbb{E}_{0,\underline{l}}\left[\tau_{\underline{l}\pm \Delta}\wedge(\widehat{T}-u)\right]<0.
	\end{equation}
	If we make the time change explicit so that $u=\zeta^{-1}_\mu(t)$ and $\widehat{T}=\zeta^{-1}_\mu(T)$ for $t<T$ we can obtain 
	\[\widehat{T}-u=\int_t^T(\zeta^{-1}_\mu)'(u)du\geq \mathfrak{h}^2(T-t)>0,\] 
	which in view of \eqref{eqn:intermediate.bd.lem.der.hat.V} (by modifying $\kappa$ if necessary) completes the proof.
\end{proof}

%

\subsection{An estimate for boundary regularity}

In this subsection we collect all of the preceding results in order to prove a key estimate on the derivatives of the value function arising in the proof of Proposition \ref{prop:bdy.unif.local.lip}. Recall that in the proof we defined the function $J=V-g$ where the spatial domain was restricted to $\pi\geq 1/2$. We used $J$ to define the (upper) $\delta$-level boundary $B_\delta(t)$ which solves $J(t,B_\delta(t))=\delta$ for $\delta<0$, and converges to $B(t)$ as $\delta\uparrow 0$.

Let us revisit the relationship between the derivatives (where they exist) of our original and time-and-space changed problems. We have	\[J(t,\pi)=V(t,\pi)-g(\pi)=\widehat{V}(\zeta_\mu^{-1}(t),S^{-1}(\pi))-\widehat{g}(S^{-1}(\pi)).\]
Consequently,
\begin{equation}\label{eqn:J.der.pi}
	\partial_\pi J(t,\pi)=\partial_\pi (V-g)(t,\pi)=\partial_l(\widehat{V}-\widehat{g})(\zeta_\mu^{-1}(t),S^{-1}(\pi))(S^{-1})'(\pi)
\end{equation}
and
\begin{equation}\label{eqn:J.der.t}
	\partial_tJ(t,\pi)=\partial_t (V-g)(t,\pi)=\partial_u(\widehat{V}-\widehat{g})(\zeta^{-1}_\mu(t),S^{-1}(\pi))(\zeta^{-1}_\mu)'(t).
\end{equation}
We will make use of this relationship in the proof of the following proposition. Note that these equalities also allow us to relate the signs of the derivatives through Lemma \ref{lem:sign.der.hat.V} since $(S^{-1})'(\pi)>0$ and $(\zeta^{-1}_\mu)'(t)>0$.


\begin{proposition}\label{prop:der.bd.for.bdy.reg}
	Consider the setting from the proof of Proposition \ref{prop:bdy.unif.local.lip}. Fix $\epsilon>0$, $0\leq t_0,t_1<T$, and non-empty $I_\epsilon:=[t_0+\epsilon/2,t_1-\epsilon/2]\subset[0,T)$. There exists a $K_\epsilon>0$ independent of $\mu$ and $\delta$ such that for any $t$ in $I_\epsilon$:
	\begin{equation}\label{eqn:bound.der.Bdelta}\left|\frac{\partial_t J(t,B_\delta(t))}{\partial_\pi J(t,B_\delta(t)}\right|\leq K_\epsilon.
	\end{equation}
\end{proposition}

\begin{proof}
	Our approach is to adapt the arguments of \cite[Theorem 4.3]{de2019lipschitz} to get an estimate that is independent of $\mu$. For brevity we do not reproduce the arguments surrounding \cite[Equations (4.16)-(4.24)]{de2019lipschitz} in their entirety as they apply directly to our setting with only minor modifications. The remainder of this proof prepares the initial estimates necessary for these arguments to carry through.
	
	Define $\widehat{J}:=\widehat{V}-\widehat{g}$ and restrict to the domain $[0,\widehat{T})\times[0,\infty)$ (to mirror the restriction of $J$ to $\pi\geq 1/2$ in Proposition \ref{prop:bdy.unif.local.lip}). Using our representation in terms of the derivatives of $\widehat{J}$ from \eqref{eqn:J.der.pi}-\eqref{eqn:J.der.t} we have
	\[\left|\frac{\partial_t J(t,B_\delta(t))}{\partial_\pi J(t,B_\delta(t)}\right|=\left|\frac{\partial_u \widehat{J}(\zeta^{-1}_\mu(t),S^{-1}(B_\delta(t)))(\zeta_\mu^{-1})'(t)}{\partial_l \widehat{J}(\zeta^{-1}_\mu(t),S^{-1}(B_\delta(t)))(S^{-1})'(B_\delta(t))}\right|=:E_1, \ \ \ t\in I_\epsilon.\]
	To be concise we will use $l_\delta:=S^{-1}(B_\delta(t))$ and $u:=\zeta_\mu^{-1}(t)$ where we do not need the functions to be explicit. Using that $\widehat{g}$ does not depend on time and Corollary \ref{cor:bd.time.der.hat.V} we have
	\begin{align*}
		E_1&\leq \frac{C}{\widehat{T}-u}\frac{\mathbb{E}_{u,l_\delta}\left[\widehat{\tau}^*-u\right](\zeta^{-1}_\mu)'(t)}{\left|\partial_l \widehat{J}(u,l_\delta)(S^{-1})'(B_\delta(t))\right|}\\
		&=\frac{C}{\zeta^{-1}_\mu(T)-\zeta^{-1}_\mu(t)}\frac{\mathbb{E}_{u,l_\delta}\left[\widehat{\tau}^*-u\right]}{\partial_l \widehat{J}(u,l_\delta)}\frac{(\zeta^{-1}_\mu)'(t)}{(S^{-1})'(B_\delta(t))}.
	\end{align*}
	The last equality follows since the denominator is strictly positive at these points. Now note that
	\[\zeta^{-1}_\mu(T)-\zeta^{-1}_\mu(t)=\int_t^T(\zeta^{-1}_\mu)'(s)ds\geq \mathfrak{h}^2(T-t)\geq \mathfrak{h}^2(T-t_1),\]
	$(S^{-1})'\geq 4$, and $(\zeta^{-1}_\mu)'\leq\mathfrak{H}^2$. Consequently,
	\begin{align*}
		E_1&\leq \frac{C\mathfrak{H}^2}{4\mathfrak{h}^2(T-t_1)}\frac{\mathbb{E}_{u,l_\delta}\left[\widehat{\tau}^*-u\right]}{\partial_l \widehat{J}(u,l_\delta)}.
	\end{align*}
	Using the representation for the derivative in the denominator from \eqref{eqn:diff.hatVg.derivative.rep} in the proof of Lemma \ref{lem:sign.der.hat.V} we find
	\begin{equation}\label{eqn:intermediate.bd.E1}
		E_1\leq\frac{C\mathfrak{H}^2}{4\mathfrak{h}^2(T-t_1)}\frac{\mathbb{E}_{u,l_\delta}\left[\widehat{\tau}^*-u\right]}{\mathbb{E}_{u,l_\delta}\left[\int_{u}^{\widehat{\tau}^*} e^{\int_{u}^va'(\widehat{L}_r)dr}\partial_l\widehat{\mathcal{L}}\widehat{g}(\widehat{L}_v)dv\right]}.
	\end{equation}
	
	We will now obtain a useful representation of the denominator in \eqref{eqn:intermediate.bd.E1} for which an additional estimate can be derived. First, we will define an analogue of $I_\epsilon$ for the time changed problem. Let $u_0:=\zeta_\mu^{-1}(t_0), \ u_1:=\zeta_\mu^{-1}(t_1)$ and $u_0^{\epsilon}:=\zeta_\mu^{-1}(t_0+\frac{\epsilon}{2}), \ u_1^{\epsilon}:=\zeta_\mu^{-1}(t_1-\frac{\epsilon}{2})$.
	Observe that the gaps between the interior points $u_0^{\epsilon},u_1^{\epsilon}$ and the endpoints $u_0,u_1$ can be bounded independently of $\mu$:
	\[u_0^\epsilon-u_0=\zeta_\mu^{-1}(t_0+\epsilon/2)-\zeta_\mu^{-1}(t_0)\geq \mathfrak{h}^2\epsilon/2,\]
	and similarly, $u_1-u_1^\epsilon\geq \mathfrak{h}^2\epsilon/2$. So, for $\epsilon':=\mathfrak{h}^2\epsilon$ we can define a set $\widehat{I}_{\epsilon'}:=\left[u_0+\frac{\epsilon'}{2},u_1-\frac{\epsilon'}{2}\right]\subset (u_0,u_1)$ such that any $t\in I_\epsilon$ maps to a $u\in \widehat{I}_{\epsilon'}$.
	
	In the proof of Proposition \ref{prop:bdy.unif.local.lip} we introduced a $1/2<\overline{\pi}<B_*$ (independent of $\mu$) which satisfies $\overline{\pi}< B_\delta(t)$ for all $t\in I_\epsilon$. Let us similarly choose here a $0<\overline{l}\leq S^{-1}(\overline{\pi})$ so that $\overline{l}<l_\delta$. With this, define
	\[\tau_{\overline{l}}:=\inf\{v\geq u:(v,L^{u,l_\delta}_v)\not\in (u_0,u_1)\times(\overline{l},\infty)\}.\]
	Arguing using iterated expectations and the strong Markov property as in \cite[Theorem 4.3]{de2019lipschitz} (see Equations (4.16)-(4.18) of the proof) we get for $\widehat{\tau}^*=\widehat{\tau}^*(u,l_\delta)$:
	\begin{align}
		\mathbb{E}_{u,l_\delta}&\left[\int_u^{\widehat{\tau}^*} e^{\int_u^va'(\widehat{L}_r)dr}\partial_l\widehat{\mathcal{L}}\widehat{g}(\widehat{L}_v)dv\right] \nonumber \\
		&\quad =\mathbb{E}_{u,l_\delta}\left[\int_u^{\tau_{\overline{l}}\wedge \widehat{\tau}^*}e^{\int_u^va'(\widehat{L}_r)dr}\partial_l\widehat{\mathcal{L}}\widehat{g}(\widehat{L}_v)dv\right] \nonumber \\
		&\quad\quad\quad+\mathbb{E}_{u,l_\delta}\left[\mathds{1}_{\tau_{\overline{l}}<\widehat{\tau}^*}e^{\int_u^{\tau_{\overline{l}}}a'(L_r)dr}\partial_l\widehat{J}(\tau_{\overline{l}},L_{\tau_{\overline{l}}})\right]=:E_2. \label{eqn:estimate.1}
	\end{align}
	
	Using this representation for the denominator in \eqref{eqn:intermediate.bd.E1}, take $\Gamma^*$ as in \eqref{eq:L.hat.bdy.bounds} to be the upper bound on the (time-changed) optimal stopping boundaries. Note that the optimal stopping time for the problem in terms of $\widehat{L}$ must arise before $\widehat{L}$ hits $\Gamma^*$. As $a'\geq0$ (and $\mathds{1}_{\tau_{\overline{l}}<\widehat{\tau}^*}\partial_l\widehat{J}(\tau_{\overline{l}},L^{u,l_\delta}_{\tau_{\overline{l}}})\geq0$) if we let $\alpha_{\overline{l}}:=\min_{l\in [\overline{l},\Gamma^*]}\partial_l\widehat{L}\widehat{g}(l)>0$ we have
	\begin{equation}\label{eqn:estimate.2}E_2\geq \mathbb{E}_{u,l_\delta}\left[\alpha_{\overline{l}} (\tau_{\overline{l}}\wedge \widehat{\tau}^*-u)+\mathds{1}_{\tau_{\overline{l}}<\widehat{\tau}^*}\partial_l\widehat{J}(\tau_{\overline{l}},L_{\tau_{\overline{l}}})\right]>0.
	\end{equation}
	This gives us an estimate for the denominator in our bound of $E_1$ from \eqref{eqn:intermediate.bd.E1}. Similarly, turning to the numerator in the same bound we can argue that we have:
	\begin{equation}\label{eqn:estimate.3}\mathbb{E}_{u,l_\delta}\left[\widehat{\tau}^*-u\right]=\mathbb{E}_{u,l_\delta}\left[ \tau_{\overline{l}}\wedge \widehat{\tau}^*-u+\mathds{1}_{\tau_{\overline{l}}<\widehat{\tau}^*}\mathbb{E}_{\tau_{\overline{l}},L_{\tau_{\overline{l}}}}\left[\widehat{\tau}^*-\tau_{\overline{l}}\right]\right]. 
	\end{equation}
	Lastly, using $\Delta>0$ and $\kappa>0$ given by Lemma \ref{lem:strict.bd.mu.indep.der.Vhat}
	\begin{align}
		\inf_{u\in [u_0,u_1]}\partial_l \widehat{J}(u,\overline{l})&\geq\inf_{t\in [t_0,t_1]}\kappa\mathbb{E}_{0,\overline{l}}\left[\tau_{\overline{l}\pm \Delta}\wedge(T-t)\right] \nonumber \\
		&\geq \kappa\mathbb{E}_{0,\overline{l}}\left[\tau_{\overline{l}\pm \Delta}\wedge(T-t_1)\right]>0. \label{eqn:estimate.4}
	\end{align} Now let
	$j_{\overline{l}}^*:=\kappa\mathbb{E}_{0,\overline{l}}\left[\tau_{\overline{l}\pm \Delta}\wedge(T-t_1)\right]$
	be this lower bound. It will play a role in the arguments that we borrow from \cite[Theorem 4.3]{de2019lipschitz} when manipulating the estimate in \eqref{eqn:estimate.2}.
	
	Using \eqref{eqn:estimate.1}-\eqref{eqn:estimate.4}, we are finally ready to invoke the proof of \cite[Theorem 4.3]{de2019lipschitz}. Observe that these equations give:
	\begin{equation}
		E_1\leq \frac{C\mathfrak{H}^2}{4\mathfrak{h}^2(T-t_1)}\frac{\mathbb{E}_{u,l_\delta}\left[ \tau_{\overline{l}}\wedge \widehat{\tau}^*-u+\mathds{1}_{\tau_{\overline{l}}<\widehat{\tau}^*}\mathbb{E}_{\tau_{\overline{l}},L_{\tau_{\overline{l}}}}\left[\widehat{\tau}^*-\tau_{\overline{l}}\right]\right]}{\mathbb{E}_{u,l_\delta}\left[\alpha_{\overline{l}} (\tau_{\overline{l}}\wedge \widehat{\tau}^*-u)+\mathds{1}_{\tau_{\overline{l}}<\widehat{\tau}^*}\partial_l\widehat{J}(\tau_{\overline{l}},L_{\tau_{\overline{l}}})\right]}
	\end{equation}
	which, in view of \cite[Equation (4.19)]{de2019lipschitz}, is (up to a constant) precisely the form of the second inequality in \cite[Equation (4.21)]{de2019lipschitz}. Hence, arguing as they do in \cite[Equation (4.21)-(4.23)]{de2019lipschitz} we obtain for any $u\in \widehat{I}_{\epsilon'}$ (and therefore, any $t\in I_\epsilon$):
	\begin{align*}&\frac{\mathbb{E}_{u,l_\delta}\left[ \tau_{\overline{l}}\wedge \widehat{\tau}^*-u+\mathds{1}_{\tau_{\overline{l}}<\widehat{\tau}^*}\mathbb{E}_{\tau_{\overline{l}},L_{\tau_{\overline{l}}}}\left[\widehat{\tau}^*-\tau_{\overline{l}}\right]\right]}{\mathbb{E}_{u,l_\delta}\left[\alpha_{\overline{l}} (\tau_{\overline{l}}\wedge \widehat{\tau}^*-u)+\mathds{1}_{\tau_{\overline{l}}<\widehat{\tau}^*}\partial_l\widehat{J}(\tau_{\overline{l}},L_{\tau_{\overline{l}}})\right]}\\
		&\quad \quad \quad \quad \quad \quad \quad \quad \leq\left(\alpha_{\overline{l}}^{-1}\vee \widehat{T}\right)\left(1+(j^*_{\overline{l}})^{-1}+2(\alpha_{\overline{l}}\epsilon')^{-1}\right)\\
		& \quad \quad \quad \quad \quad \quad \quad \quad \leq \left(\alpha_{\overline{l}}^{-1}\vee (\mathfrak{H}^2T)\right)\left(1+(j^*_{\overline{l}})^{-1}+2(\alpha_{\overline{l}}\mathfrak{h}^2\epsilon)^{-1}\right).
	\end{align*}
	Thus, we arrive at our desired bound:
	\[E_1\leq\frac{C\mathfrak{H}^2\left(\alpha_{\overline{l}}^{-1}\vee (\mathfrak{H}^2T)\right)}{4\mathfrak{h}^2(T-t_1)}\left(1+(j^*_{\overline{l}})^{-1}+2(\alpha_{\overline{l}}\mathfrak{h}^2\epsilon)^{-1}\right), \ \ \ t\in I_\epsilon.\]
	Notice that $\alpha_{\overline{l}}$ and $j^*_{\overline{l}}$ do not depend in any way on the choice of $\mu$ and so, the right hand side depends only on values that are common to the problem for all input measures. This completes the proof.
\end{proof}

\subsection{Smooth fit condition}
A consequence of the boundary regularity from Proposition \ref{prop:bdy.unif.local.lip} (which in view of Corollary \ref{cor:inherited.properties.hat.V.problem}(ii) is also inherited by the transformed boundaries $\gamma$, $\Gamma$) is the \textit{smooth fit} condition for the spatial derivative across the boundaries. That is, for any $u\in[0,\widehat{T})$ we have the mapping $l\mapsto\partial_l\widehat{V}(u,l)$ is $C^1$ (see \cite[Remark 4.5]{de2019lipschitz} for a related example and justification). We state this result formally here and also extend it to the original value function $V$.

\begin{proposition}\label{prop:smooth.fit}
	For any fixed times $u\in[0,\hat{T})$ and $t\in[0,T)$, the value functions $\widehat{V}(u,l)$ and $V(t,\pi)$ are continuously differentiable in their spatial variables $l$ and $\pi$, respectively.
\end{proposition}

\begin{proof}
	We treat $\widehat{V}(u,l)$ as the result for $V(t,\pi)$ then follows from \eqref{eq:relating.value.functions}.  To begin, for any $u\in [0,\widehat{T})$ we say that the boundary points $\Gamma(u),\gamma(u)$ are probabilistically regular for the interior of the stopping region:
	\[\mathrm{int}\widehat{\mathscr{D}}:=\{(u,l)\in(0,\widehat{T})\times \mathbb{R}: l\in (-\infty,\gamma(u))\cup(\Gamma(u),\infty)\}\] 
	if $\mathbb{P}(\sigma_{\mathrm{int}\widehat{\mathscr{D}}}(u,\Gamma(u))=0)=1$ and $\mathbb{P}(\sigma_{\mathrm{int}\widehat{\mathscr{D}}}(u,\gamma(u))=0)=1$
	for
	\[\sigma_{\mathrm{int}\widehat{\mathscr{D}}}(u,l):=\inf\left\{v >u:(v,L_v^{u,l})\in\mathrm{int}\widehat{\mathscr{D}}\right\}.\]
	By the local Lipschitz property of the boundaries $\Gamma, \gamma$ and standard properties of Brownian motion, we obtain
	\[\sigma_{\mathrm{int}\widehat{\mathscr{D}}}(u,\Gamma(u))=0, \ \ \ \sigma_{\mathrm{int}\widehat{\mathscr{D}}}(u,\gamma(u))=0 \ \ \ a.s.\]
	for all $u\in[0,\hat{T})$ and so the boundary points are probabilistically regular. With this, the claim in the proposition follows from \cite[Theorem 10]{de2020global}. 
\end{proof}
\section{Continuity of the Value Function in the Volatility}\label{app:cont.val.funct}

The proof techniques we employ rely crucially on the idea of volatility times which correspond to the quadratic variation of continuous local martingales. In particular, the work of Janson and Tysk \cite{janson2003volatility} shows that for a \textit{given} Brownian motion $W$, there exists a unique stopping time solution to
\[\xi(t)=\int_0^t\|\Delta\bh_\mu(s)\|^2W_{\xi(s)}^2(1-W_{\xi(s)})^2ds\]
such that $\widetilde{\Pi}_t=W_{\xi(t)}$ is a solution of 
\[d\widetilde{\Pi}_s=\|\Delta\bh_\mu(s)\|\widetilde{\Pi}_s(1-\widetilde{\Pi}_s)d\widetilde{W}_s\]
for \textit{some} Brownian motion $\widetilde{W}$. Moreover, this $\xi(t)$ is the quadratic variation of $\widetilde{\Pi}$. This statement is to be compared to the well known result that a continuous local martingale can be represented as the time change of \textit{some} Brownian motion. Working with the same Brownian motion and multiple volatility times is remarkably useful for arguments of the type undertaken in this appendix. 
For a formal statement of their result see \cite[Theorem 1]{janson2003volatility} or \cite[Theorem 3.2]{ekstrom2004properties}.

In view of this, we cast our problem \eqref{eqn:value.function} in a way that is reminiscent of the setup in \cite{ekstrom2004properties} and exploits the structure just introduced. As in Section \ref{sec:single.agent}, we suppress the (arbitrary) index $i\in I$ to simplify the exposition.  Recall from Section \ref{sec:value_fun} that in our problem we may equivalently work with $\mathbb{F}^\Pi$ stopping times, so by Levy's characterization our value function can be equivalently written as:

\begin{equation}
	V(t,\pi;\mu)=\inf_{\tau\in\mathcal{T}_t^\Pi}\mathbb{E}\left[c(\tau-t)+g(\Pi_\tau^{t,\pi})\right]
\end{equation}
where
\begin{equation}
	d\Pi_s^{t,\pi}=\|\Delta\bh_\mu(s)\|\Pi_s^{t,\pi}(1-\Pi_s^{t,\pi})d\widetilde{W}_s, \ \ \ \Pi_t^{t,\pi}=\pi,
\end{equation}
and $\widetilde{W}$ is an \textit{arbitrary} Brownian motion. In the remainder of this section we will maintain this simplification and treat the posterior probability process as driven by a \textit{one-dimensional} Brownian motion.

%
%

We are ready to prove the required continuity in the volatility.

\begin{proposition}\label{prop:val.func.cont.vol}
	If $\|\Delta\bh_{\mu_{n}}\|\to \|\Delta\bh_{\mu}\|$ pointwise as $n\to\infty$ then 
	\[\lim_{n\to\infty} V(t,\pi;\mu_n)=V(t,\pi;\mu), \ \ \ \forall (t,\pi)\in [0,T)\times (0,1).\]
\end{proposition}

\begin{proof}
	Without loss of generality, we will show the result for $t=0$. Let $W$ be a Brownian motion with $W_0=\pi$ and let $\xi$ and $\xi_n$ be the stopping time solutions to:
	\[\xi(t)=\int_0^t\|\Delta\bh_{\mu}(s)\|^2W_{\xi(s)}^2(1-W_{\xi(s)})^2ds,\]  \[\xi_n(t)=\int_0^t\|\Delta\bh_{\mu_n}(s)\|^2W_{\xi_n(s)}^2(1-W_{\xi_n(s)})^2ds.\]
	which by \cite[Theorem 1]{janson2003volatility} exist and are almost surely unique. Let $\mathbb{F}^W$, $\mathbb{F}^\Pi$, $\mathbb{F}^{\Pi^n}$ be the completed filtrations generated by $W$, $\Pi=(\Pi_t)_{t\geq0}:=(W_{\xi(t)})_{t\geq0}$, and $\Pi^n=(\Pi^n_t)_{t\geq0}:=(W_{\xi_n(t)})_{t\geq0}$, respectively. As we have noted, these processes can be interpreted as the posterior probability processes from our main discussions. 
	
	We now gather some properties for later use. By \cite[Lemma 9]{janson2003volatility} (which applies to our setting of a compact state space) we have $\xi_n(t)\to \xi(t)$ almost surely for every $t$, and so as they are monotone and continuous we get the stronger result that the convergence holds for all $t$ almost surely. Moreover, from the representation as the Quadratic Variation (see also \cite[Theorem 1]{janson2003volatility}), $\xi(t)$ and $\xi_n(t)$ are continuous and increasing. Since $\|\Delta\bh_{\mu}(t)\|,\|\Delta\bh_{\mu_n}(t)\|\geq \mathfrak{h}>0$ for all $t\in[0,T]$ and all $n$ we have $\|\Delta\bh_{\cdot}(t)\|\pi(1-\pi)>0$ for all $(t,\pi)$ such that $\pi\in(0,1)$. This, in turn, implies that $\xi(t)$ is strictly increasing so long as $W_{\xi(t)}\in (0,1)$. In particular, if $\xi$ is constant on some interval $[t,t+\epsilon)$ it must be constant on $[t,T]$ as both $0$ and $1$ are absorbing states for our process $\Pi$. On such an interval we must have $\Pi_t=W_{\xi(t)}\in\{0,1\}$ and $g(\Pi_t)=0$. 
	However, by Corollary \ref{cor:properties.of.Pi}, $\Pi$ and $\Pi^n$ do not hit 0 or 1 in finite time almost surely, so $\xi$ and $\xi_n$ are strictly increasing. 
	
	We begin by showing $\limsup_{n\to\infty} V(0,\pi;\mu_n)\leq V(0,\pi;\mu)$. Let $\tau^*$ be an optimal stopping time for the problem in terms of $\Pi$ over $\mathbb{F}^{\Pi}$-stopping times. Consider also the stopping times
	\[\tau^\epsilon:=\inf\{s\geq0: V(s,\Pi_s;\mu)\geq g(\Pi_s)-\epsilon\}.\]
	By the continuity of $V$ and $g$ we have that $\tau^\epsilon<\tau^*$ (unless $\tau^*=0)$ and by \cite{fakeev1970optimal} we have $\tau^\epsilon$ is $\epsilon$-optimal in the sense that
	\begin{equation}\label{eq:eps.opt.lower.bd}V(0,\pi;\mu)\geq \mathbb{E}_\pi\left[c\tau^\epsilon+g(\Pi_{\tau^\epsilon})\right]-\epsilon.
	\end{equation}
	Defining
	\begin{align*}
		\tau^\epsilon_n&:=\inf\{s:\xi_n(s)\geq \xi(\tau^\epsilon)\wedge \xi_n(T)\}\\
		&=\inf\{s:\xi_n(s)\geq \xi(\tau^\epsilon)\}\wedge \inf\{s:\xi_n(s)\geq \xi_n(T)\}
	\end{align*}
	we have by \cite[Lemma 3.5]{ekstrom2004properties} that $\tau_n^\epsilon\leq T$ is a $\mathbb{F}^{\Pi^n}$ stopping time. At the same time, it is straightforward to verify the following standalone fact: Suppose $f,f_n:[0,T]\to[0,\infty)$ are continuous functions with $f(0)=f_n(0)=0$, $f_n$ increasing, $f$ strictly increasing, and
	\[t_n:=\inf\{s\in[0,T]:f_n(s)\geq f(t)\}\]
	for any fixed $t\in[0,T]$. If $f_n\to f$ pointwise,then $t_n\to t$ as $n\to\infty$.
	
	By applying this result to our expression for $\tau_n^\epsilon$ we obtain that $\tau^\epsilon_n\to \tau^\epsilon$ almost surely as $n\to\infty$. It follows that
	\[0\leq \xi(\tau^\epsilon)-\xi_n(\tau^\epsilon_n)=(\xi(\tau^\epsilon)-\xi_n(T))\mathds{1}_{\xi(\tau^\epsilon)>\xi_n(T)}\leq(\xi(T)-\xi_n(T))\mathds{1}_{\xi(T)>\xi_n(T)}\]
	and so by taking limits across the above we have $\xi_n(\tau^\epsilon_n)\to\xi(\tau^\epsilon)$ a.s. As a result, $g(W_{\xi_n(\tau_n^\epsilon)})\to g(W_{\xi(\tau^\epsilon)})$ almost surely. 
	Note that
	\[0\leq \tau^\epsilon_n+g(W_{\xi_n(\tau_n^\epsilon)})\leq T+C\]
	for some $C>0$ playing the role of the upper bound on $g$ (as from Assumption \ref{ass:g} it is a concave function on $[0,1]$ with $g(0)=0=g(1)$).
	The dominated convergence theorem and \eqref{eq:eps.opt.lower.bd} then gives
	\[V(0,\pi;\mu)\geq\limsup_{n\to\infty} V(0,\pi;\mu_n)-\epsilon.\]
	Since $\epsilon>0$ was arbitrary our first claim follows.
	
	To complete the proof we will now show $\liminf_{n\to\infty}V(0,\pi;\mu_n)\geq V(0,\pi;\mu)$. First, let $\tau^*_n$ be a sequence of optimal $\mathbb{F}^{\Pi^n}$-stopping times associated with $V(0,\pi;\mu_n)$. Define 
	$\tilde{\tau}_n:=\inf\{s:\xi(s)\geq \xi_n(\tau^*_n)\wedge \xi(T)\}$. Again by \cite[Lemma 3.5]{ekstrom2004properties} we have that that $\tilde{\tau_n}\leq T$ is a $\mathbb{F}^{\Pi}$ stopping time for each $n$. Then
	\begin{align*}
		0\leq\xi_n(\tau^*_n)-\xi(\tilde{\tau}_n)&=\left(\xi_n(\tau^*_n)-\xi(T)\right)\mathds{1}_{\xi_n(\tau^*_n)>\xi(T)}\leq \left(\xi_n(T)-\xi(T)\right)\mathds{1}_{\xi_n(T)>\xi(T)}
	\end{align*}
	and so as $n\to\infty$, $\xi_n(\tau^*_n)-\xi(\tilde{\tau}_n)\to 0$ a.s. since $\xi_n(T)\to\xi(T)$ a.s. 
	Moreover, it can be seen that $\tau^*_n-\tilde{\tau}_n\to 0$ a.s. Writing $\xi^{-1}(y):=\inf\{x:\xi(x)\geq y\}$, for any $\epsilon>0$ we have
	\begin{align}\label{eq:conv.stop.times.pf.Lem.4.1}
		|\tilde{\tau}_n-\tau^*_n|&= |\xi^{-1}(\xi_n(\tau^*_n)\wedge \xi(T))-\tau^*_n|\nonumber \\
		&\leq |\xi^{-1}(\xi_n(\tau^*_n))-\tau^*_n)|\mathds{1}_{\tau_n^*\leq T-\epsilon}+(T-\tau^*_n)\mathds{1}_{\tau_n^*> T-\epsilon}\nonumber \\
		&\leq \sup_{0\leq t\leq T-\epsilon}|\xi^{-1}(\xi_n(t))-t|+\epsilon.
	\end{align}

	Let $t\in[0,T-\epsilon]$ be fixed. Observe that if there is a $\delta>0$ such that $\limsup_{n\to\infty}\xi^{-1}(\xi_n(t))>t+\delta$ then $\xi(t+\delta)\leq \xi_n(t)$ for sufficiently large $n$. By the strict increase of $\xi$ we must have $\xi(t)<\xi(t+\delta)\leq \xi_n(t)$ which contradicts the convergence of $\xi_n$ to $\xi$. It follows that $\limsup_{n\to\infty}\xi^{-1}(\xi_n(t))\leq t$. A similar argument gives $\liminf_{n\to\infty}\xi^{-1}(\xi_n(t))\geq t$ so $\lim_{n\to\infty}\xi^{-1}(\xi_n(t))=t$. Since $t$ was arbitrary in $[0,T-\epsilon]$ this holds for all such $t$. As the $\xi^{-1}(\xi_n(t))$ are increasing functions for each $n$ that have a continuous limit, the convergence is uniform on $[0,T-\epsilon]$. We conclude $\sup_{0\leq t\leq T-\epsilon}|\xi^{-1}(\xi_n(t))-t|\to 0$ and since $\epsilon$ was arbitrary in \eqref{eq:conv.stop.times.pf.Lem.4.1}, that $|\tilde{\tau}_n-\tau^*_n|\to 0$ a.s.
	
	Note now that $\xi(t),\xi_n(t)\leq \bar{\xi}(t)$ where $\bar{\xi}$ is defined as the stopping time solution to
	\begin{equation}\label{eqn:upper.bd.time.change}
		\bar{\xi}(t)=\int_0^t\mathfrak{H}^2W_{\bar{\xi}(s)}^2(1-W_{\bar{\xi}(s)})^2ds\end{equation}
	by \cite[Lemma 3.3.]{ekstrom2004properties}. Then, by monotonicity $\xi(t),\xi_n(t)\leq \bar{\xi}(T)$. By the H\"older continuity of the Brownian path on $[0,\bar{\xi}(T)]$ and the uniform continuity of $g$ on $[0,1]$, it is not hard to see that:
	\begin{align*}
		|\tilde{\tau}_n+g(W_{\xi(\tilde{\tau}_n)})-\tau^*_n-g(W_{\xi_n(\tau^*_n)})|\xrightarrow{n\to\infty} 0 \ \ a.s.
	\end{align*}
	Once again, since $g$ is bounded and the stopping times satisfy $\tilde{\tau}_n,\tau^*_n\leq T$ we have by the dominated convergence theorem:
	\[V(0,\pi;\mu)-V(0,\pi;\mu_n)\leq E\left[\tilde{\tau}_n+g(W_{\xi(\tilde{\tau}_n)})\right]-E\left[\tau^*_n+g(W_{\xi_n(\tau^*_n)})\right]\xrightarrow{n\to\infty}0.\]
	That is, $\liminf_{n\to\infty}V_n(0,\pi)\geq V(0,\pi)$ as required. Taken together with the previous claim, the proof is complete.
\end{proof}

\section{Continuity of the Boundaries in the Input Measure}

\subsection{Soft-classification problem}\label{app:bdy.convergence}


Before tackling the main proof we provide a preliminary technical result that allows us to pass to the limit across the integral equations \eqref{eq:integral.eq} of Proposition \ref{prop:boundary.integral.equation}. In what follows, we suppress the index $i\in I$. To make explicit the dependence on $\mu$ of the equations, we write $\mathcal{L} = \|\Delta\bh_{\mu}(\cdot)\|^2\mathcal{A}$ and the posterior probability process as $\Pi^{t,\pi,\mu}$.

\begin{lemma}\label{lem:passing.to.limit.int.eqns}
	Suppose $\mu_{n}\Rightarrow \mu$ and  $b_{n}\to \overline{b}$, $B_{n}\to \overline{B}$ pointwise. If $(\pi_{n}(t),\pi(t))=(b_{n}(t),\overline{b}(t))$ or $(\pi_{n}(t),\pi(t))=(B_{n}(t),\overline{B}(t))$ then:
	\begin{equation}\label{eq:conv.int.eq.1}\mathbb{E}\left[g( \Pi_{T}^{t,\pi_{n}(t),\mu_{n}})\right]\to \mathbb{E}\left[g( \Pi_{T}^{t,\pi(t),\mu})\right],
	\end{equation}
	\begin{align}\label{eq:conv.int.eq.2}\int_t^{T}\mathbb{P}&\left(\Pi_u^{t,\pi_{n}(t),\mu_{n}}\in(b_{n}(u),B_{n}(u))\right)du \nonumber\\ 
		&\quad \quad \quad \quad \quad \quad \quad \quad \quad \quad \to\int_t^{T}\mathbb{P}\left(\Pi_u^{t,\pi(t),\mu}\in(\overline{b}(u),\overline{B}(u))\right)du,
	\end{align}
	and
	\begin{align} \label{eq:conv.int.eq.3} \int_t^{T}&\mathbb{E}\left[\|\Delta\bh_{\mu_{n}}(u)\|^2\mathcal{A}g(\Pi_u^{t,\pi_{n}(t),\mu_{n}})\mathds{1}_{\left\{\Pi_u^{t,\pi_{n}(t),\mu_{n}}\in(0,b_{n}(u))\cup(B_{n}(u),1)\right\}}\right]du \nonumber \\
		&\rightarrow\int_t^{T}\mathbb{E}\left[\|\Delta\bh_{\mu}(u)\|^2\mathcal{A}g(\Pi_u^{t,\pi(t),\mu})\mathds{1}_{\left\{\Pi_u^{t,\pi(t),\mu}\in(0,\overline{b}(u))\cup(\overline{B}(u),1)\right\}}\right]du. 
	\end{align}
	
\end{lemma}

\begin{proof}
	By Lemma \ref{lem:cont.val.funct} we have $\|\Delta\bh_{\mu_{n}}\|\to \|\Delta\bh_{\mu}\|$ pointwise. Applying \cite[Theorem 3.1]{shevchenko2015convergence} on our compact state space of $[0,1]$ we have that $\Pi_u^{t,\pi_{n}(t),\mu_{n}}\to \Pi_u^{t,\pi(t),\mu}$ in probability. Moreover, by Corollary \ref{cor:properties.of.Pi} for any $u>t$, $\Pi_u^{t,\pi_{n}(t),\mu_{n}}$ and $\Pi_u^{t,\pi(t),\mu}$ admit a density on $(0,1)$. Since $g$ is bounded on $[0,1]$ the convergence in \eqref{eq:conv.int.eq.1} follows by the continuous mapping and bounded convergence theorems. Write
	\[G_0(\Pi,b,B):=\mathds{1}_{(-\infty,0)}(\Pi-B)-\mathds{1}_{(-\infty,0)}(\Pi-b)\]
	and
	\[G_1(\Pi,\eta,b,B):=\eta^2\mathcal{A}g(\Pi)\left(\mathds{1}_{(-\infty,0)}(\Pi-b)+\mathds{1}_{(-\infty,0)}(B-\Pi)\right)\]
	so that these functions correspond to the integrands of \eqref{eq:conv.int.eq.2} and \eqref{eq:conv.int.eq.3}. We have that $G_0$ and $G_1$ are continuous almost everywhere and are bounded over the range of our inputs. Thus, the continuous mapping and bounded convergence theorems again give for $u>t$: 
	\[\mathbb{E}\left[G_0(\Pi_u^{t,\pi_{n}(t),\mu_{n}},b_{n}(u),B_{n}(u))\right]\to \mathbb{E}\left[G_0(\Pi_u^{t,\pi(t),\mu},\overline{b}(u),\overline{B}(u))\right],\]
	\begin{align*}\mathbb{E}&\left[G_1(\Pi_u^{t,\pi_{n}(t),\mu_{n}},\|\Delta\bh_{\mu_{n}}(u)\|,b_{n}(u),B_{n}(u))\right]\\
		&\quad \quad \quad \quad \quad \quad \quad \quad \quad \quad \to \mathbb{E}\left[G_1(\Pi_u^{t,\pi(t),\mu},\|\Delta\bh_{\mu}(u)\|,\overline{b}(u),\overline{B}(u))\right].
	\end{align*}
	Another application of the bounded convergence theorem then gives the desired convergence in \eqref{eq:conv.int.eq.2} and \eqref{eq:conv.int.eq.3}.
\end{proof}


We are now ready to prove the Proposition.

\begin{proof}[Proof of Lemma \ref{lem:conv.bdys}]
	
	From Proposition \ref{prop:bdy.unif.local.lip} we have that the boundaries $(b_n)_{n\geq0}$, $b$, $(B_n)_{n\geq0}$, and $B$ are uniformly locally Lipschitz. Since the boundaries are also uniformly bounded away from $(0,1)$ (see Proposition \ref{prop:bounds.on.bdys}), a diagonalization argument gives that for any subsequence of $(b_{n})$ and $(B_{n})$ we can always extract further subsequences $(b_{n'})$, $(B_{n'})$ that converge locally uniformly to locally Lipschitz limits $\overline{b}$ and $\overline{B}$ on $[0,T)$.
	
	As $V(t,\cdot;\mu_{n'}),V(t,\cdot;\mu)$ are concave and uniformly bounded on $(0,1)$ we have that the pointwise convergence $V(t,\cdot;\mu_{n'})\to V(t,\cdot;\mu)$ from Lemma \ref{lem:cont.val.funct} can be extended to local uniform convergence on $(0,1)$ for each $t$ (see \cite[Theorem 10.8]{rockafellar2015convex}). Fix $t$ and take $b_*,B^*$ from Proposition \ref{prop:bounds.on.bdys} so that $b_*\leq b_{n'},b,B_{n'},B\leq B^*$. On $[b_*,B^*]\subset(0,1)$, we have that $V(t,\cdot;\mu_{n'})\to V(t,\cdot;\mu)$ uniformly. Moreover, by assumption
	\[V(t,b_{n'}(t);\mu_{n'})=g(b_{n'}(t)), \ \ \ V(t,B_{n'}(t);\mu_{n'})=g(B_{n'}(t)).\]
	By the convergence of the boundaries for $t\in[0,T)$, and the continuity and uniform convergence of $V(t,\cdot;\mu_n),V(t,\cdot;\mu)$ on $[b_*,B^*]$ we have by taking limits that
	\[V(t,\bar{b}(t);\mu)=g(\bar{b}(t)), \ \ \ V(t,\bar{B}(t);\mu)=g(\bar{B}(t)).\]
	Thus, since $t$ was arbitrary we conclude that $\bar{b}\leq b$, $B\leq \bar{B}$ by the definition of the continuation region for $V$.
	
	It remains to show the reverse inequality. Recall that the boundaries $b,b_{n'},B$ and $B_{n'}$ satisfy the integral equations from Proposition  \ref{prop:boundary.integral.equation}. We obtain for $b$:
	
	\begin{align*}g(b(t))&=\mathbb{E}\left[g( \Pi_{T}^{t,b(t),\mu})\right]+c\int_t^{T}\mathbb{P}\left(\Pi_u^{t,b(t),\mu}\in(b(u),B(u))\right)du\\
	&\quad \quad  -\int_t^{T}\mathbb{E}\left[\|\Delta\bh_{\mu}(u)\|^2\mathcal{A}g(\Pi_u^{t,b(t),\mu})\mathds{1}_{\left\{\Pi_u^{t,b(t),\mu}\in(0,b(u))\cup(B(u),1)\right\}}\right]du,
	\end{align*}
	and similarly for $b_{n'},B$, and $B_{n'}$. Taking limits across the integral equations we have:
	
	\begin{align*}
		&g(\overline{b}(t))=\\
		& \lim_{n'\to\infty}\bigg(\mathbb{E}\left[g( \Pi_{T}^{t,b_{n'}(t),\mu_{n'}})\right]+c\int_t^{T}\mathbb{P}\left(\Pi_u^{t,b_{n'}(t),\mu_{n'}}\in(b_{n'}(u),B_{n'}(u))\right)du\\
		& -\int_t^{T}\mathbb{E}\left[\|\Delta\bh_{\mu_{n'}}(u)\|^2\mathcal{A}g(\Pi_u^{t,b_{n'}(t),\mu_{n'}})\mathds{1}_{\left\{\Pi_u^{t,b_{n'}(t),\mu_{n'}}\in(0,b_{n'}(u))\cup(B_{n'}(u),1)\right\}}\right]du\bigg),\\
		&g(\overline{B}(t))=\\
		& \lim_{n'\to\infty}\bigg(\mathbb{E}\left[g( \Pi_{T}^{t,B_{n'}(t),\mu_{n'}})\right]+c\int_t^{T}\mathbb{P}\left(\Pi_u^{t,B_{n'}(t),\mu_{n'}}\in(b_{n'}(u),B_{n'}(u))\right)du\\
		& -\int_t^{T}\mathbb{E}\left[\|\Delta\bh_{\mu_{n'}}(u)\|^2\mathcal{A}g(\Pi_u^{t,B_{n'}(t),\mu_{n'}})\mathds{1}_{\left\{\Pi_u^{t,B_{n'}(t),\mu_{n'}}\in(0,b_{n'}(u))\cup(B_{n'}(u),1)\right\}}\right]du\bigg).
	\end{align*}
	
	By Lemma \ref{lem:passing.to.limit.int.eqns} the limit can be passed through on the right hand side. Hence, the integral equations of Proposition  \ref{prop:boundary.integral.equation} for input measure $\mu$ that are satisfied by $b$ and $B$ are also satisfied by the continuous functions $\overline{b}$ and $\overline{B}$. Moreover, we have already shown that $\overline{b}\leq b$  and $\overline{B}\geq B$. At the same time, Proposition  \ref{prop:boundary.integral.equation} gives that $b$ and $B$ are the unique maximal continuous solutions to these integral equations. We conclude $\overline{b}=b$ and $B=\overline{B}$. Since the sub-sequence was arbitrary the limit of $(b_n)_{n\geq0}$ and $(B_n)_{n\geq0}$ exists and coincides with $b$ and $B$. Moreover, by the uniform bounds and uniform local Lipschitz constants, the convergence is locally uniform on $[0,T)$.
\end{proof}

\subsection{Classic problem} \label{app:bdy.convergence.classic}

We organize this appendix by first proving a collection of short preliminary lemmas.

%

\begin{lemma}
	If $V(t,\pi;\mu_n)\to V(t,\pi;\mu)$ then $\partial_\pi V(t,\pi;\mu_n)\to \partial_\pi V(t,\pi;\mu)$ for all $(t,\pi)\in[0,T)\times(0,1)$. Moreover, the convergence is locally uniform in $\pi\in(0,1)$.
\end{lemma}

\begin{proof}
	Since $V(t,\cdot;\mu),V(t,\cdot;\mu_n)$ are concave, we have by standard results in convex analysis that the convergence of the derivatives holds at all points of differentiability of $V(t,\cdot;\mu)$ in $(0,1)$ for fixed $t\in[0,T)$. Since $V(t,\cdot;\mu)\in C^1(0,1)$ we have the convergence is pointwise on $(0,1)$. Since $\partial_\pi V(t,\cdot;\mu_n)$ and $\partial_\pi V(t,\cdot;\mu)$ are necessarily monotone (by concavity) and continuous, we have that the convergence is locally uniform in $\pi$.
\end{proof}

\begin{lemma}\label{lem:bound.in.terms.of.der} For fixed $\mu\in\mathcal{P}([0,T])^2$, the boundaries can be equivalently represented by:
	\[B(t)=\inf\left\{\pi\geq \frac{a_2}{a_1+a_2}: \partial_\pi(V-g)(t,\pi;\mu)=0\right\}\]
	\[b(t)=\sup\left\{\pi\leq \frac{a_2}{a_1+a_2}: \partial_\pi(V-g)(t,\pi;\mu)=0\right\}\]
\end{lemma}

\begin{proof}
	We treat the first case and the second is analogous. We have that $V-g$ is concave in $\pi$ and increasing on $\left[\frac{a_2}{a_1+a_2},1\right]$ (see the proof of Proposition \ref{prop:bdy.char.classic}). Moreover  if $\pi\geq B(t)$, $\partial_\pi(V-g)(t,\pi;\mu)=0$ by the smooth fit condition (see Proposition \ref{prop:PDE} which holds in the classical sense). Then, by Proposition \ref{prop:mon.and.lip.V}(vi) we have $\partial_{\pi\pi}(V-g)=\partial_{\pi\pi}V\leq \frac{-8c}{\mathfrak{H}^2}<0$ on $\left(\frac{a_2}{a_1+a_2},B(t)\right)$. Since for fixed $t$, $(V-g)(t,\pi;\mu)$ is $C^1$ everywhere and twice differentiable except at $\pi=B(t)$ we have that for any $\epsilon>0$ and $\frac{a_2}{a_1+a_2}\leq \pi_0\leq B(t)-\epsilon$ that
	\begin{align*}
		\partial_\pi(V-g)(t,\pi_0;\mu)&=\partial_\pi(V-g)(t,B(t)-\epsilon;\mu)-\int_{\pi_0}^{B(t)-\epsilon}\partial_{\pi\pi}V(t,u;\mu)du\\
		&\geq \partial_\pi(V-g)(t,B(t)-\epsilon;\mu)+ \frac{8c}{\mathfrak{H}^2}|B(t)-\epsilon-\pi_0|.
	\end{align*}
	Taking $\epsilon\downarrow 0$ and applying continuity gives:
	\[\partial_\pi(V-g)(t,\pi_0;\mu)\geq \frac{8c}{\mathfrak{H}^2}|B(t)-\pi_0|>0.\]
	Hence, if $\frac{a_2}{a_1+a_2}\leq \pi<B(t)$ we have $\partial_\pi(V-g)(t,\pi;\mu)>0$ and the result holds.
\end{proof}

\begin{lemma} For fixed $\mu\in\mathcal{P}([0,T])^2$
	\begin{enumerate}
		\item[(i)] $\partial_\pi(V-g)$ is strictly decreasing on $\left(\frac{a_2}{a_1+a_2},B(t)\right)$ and $\left(b(t),\frac{a_2}{a_1+a_2}\right)$. \item[(ii)] $|\partial_\pi(V-g)|\leq L$ for some constant $L>0$.
		\item[(iii)] On the domain $\left(b(t),\frac{a_2}{a_1+a_2}\right)$ or $\left(\frac{a_2}{a_1+a_2},B(t)\right)$ and fixed $t\in[0,T)$, the function $\partial_\pi(V-g)(t,\cdot;\mu)$ admits a strictly decreasing and Lipschitz inverse function $Q_t(\cdot)$ with Lipshitz constant independent of $\mu$.
	\end{enumerate}
	
\end{lemma}

\begin{proof}
	The first claim is an immediate consequence of Proposition \ref{prop:mon.and.lip.V}(vi). The second claim follows since $V$ and $g$ are both Lipschitz (see Proposition \ref{prop:mon.and.lip.V}). We now turn to the third claim. Since $\partial_\pi(V-g)$ is $C^1$ and strictly decreasing on these intervals, the inverse function theorem says:
	\[0\geq Q'_t(x)=\frac{1}{\partial_{\pi\pi}(V-g)(t,Q_t(x);\mu)}\geq-\frac{\mathfrak{H}^2}{8c} \]
	where we have used Proposition \ref{prop:mon.and.lip.V}(vi).
\end{proof}

\begin{lemma}
	For any sequence $V(t,\pi;\mu_n)$ restricted to $\pi\in\left(b_n(t),\frac{a_2}{a_1+a_2}\right)$ or $\pi\in \left(\frac{a_2}{a_1+a_2},B_n(t)\right)$ (where $b_n,B_n$ are the corresponding boundaries) the associated inverse functions $Q_t^n(\cdot)$ have a common domain $\left(-\alpha_t,0\right)$ or $\left(0,\alpha_t\right)$, respectively, for some $\alpha_t>0$ depending on $t$. Moreover, if the pointwise convergence $\partial_\pi V(t,\cdot;\mu_n)\to \partial_\pi V(t,\cdot;\mu)$ holds, the associated inverses satisfy $Q^n_t(\cdot)\to Q(\cdot)$ pointwise on $(0,\alpha_t)$ (resp. $(-\alpha_t,0)$).
\end{lemma}

\begin{proof}
	We will treat the upper interval and the lower interval is analogous. From the proof of Lemma \ref{lem:bound.in.terms.of.der} we have that 
	\begin{align*}
		\partial_\pi(V-g)\left(t,\frac{a_2}{a_1+a_2};\mu_n\right)&\geq\frac{8c}{\mathfrak{H}^2}\left(B_n(t)-\frac{a_2}{a_1+a_2}\right)>0.
	\end{align*} By Proposition \ref{prop:bdy.char.classic} we have that there is a $\underline{B}(t)$ such that $\frac{a_2}{a_1+a_2}<\underline{B}\leq B_n$. Hence, we have
	\[\partial_\pi(V-g)\left(t,\frac{a_2}{a_1+a_2};\mu_n\right)\geq\frac{8c}{\mathfrak{H}^2}\left(\underline{B}(t)-\frac{a_2}{a_1+a_2}\right)=:\alpha_t>0.\]
	Since $\partial_\pi(V-g)(t,\cdot;\mu_n)$ is decreasing to $0$, by continuity it must take on every value in $\left(0,\alpha_t\right)$ on $\left(\frac{a_2}{a_1+a_2},B_n(t)\right)$. It follows that for any $n$ there is a common domain for the inverse functions $Q^n_t$.
	
	It remains to show the convergence. Without loss of generality we treat the first case of $(0,\alpha_t)$. This proof is a modification of standard arguments. Let $\epsilon>0$ and fix $y\in(0,\alpha_t)$. Since $Q_t$ is continuous by our above analysis we can choose $\delta>0$ sufficiently small so that $0<y-\delta<y+\delta<\alpha_t$ and:
	\[\left|Q_t(y+\delta)-Q_t(y)\right|<\epsilon, \ \ \ \left|Q_t(y-\delta)-Q_t(y)\right|<\epsilon\]
	Since $\partial_\pi V(t,\cdot;\mu_n)\to \partial_\pi V(t,\cdot;\mu)$ on $\left(\frac{a_2}{a_1+a_2},1\right)$ and $Q_t(y\pm\delta)\in\left(\frac{a_2}{a_1+a_2},1\right)$ we have that there exists a sufficiently large $N$ such that for all $n>N$:
	\[\left|\partial_\pi V(t,Q_t(y+\delta);\mu_n)-(y+\delta)\right|<\delta,\] \[\left|\partial_\pi V(t,Q_t(y-\delta);\mu_n)-(y-\delta)\right|<\delta.\]
	As a consequence
	\[\partial_\pi V(t,Q_t(y+\delta);\mu_n)>y \ \ \ \mathrm{and} \ \ \ \partial_\pi V(t,Q_t(y-\delta);\mu_n)<y\]
	Since $Q_t^n$ is monotonic by our above analysis, the above inequalities tell us that $Q_t^n(y)$ must be contained in the interval defined by $Q_t(y+\delta)$ and $Q_t(y-\delta)$. However, the first set of inequalities derived above implies that these values are in turn contained in  $[Q_t(y)-\epsilon,Q_t(y)+\epsilon]$
	which then implies $|Q_t^n(y)-Q_t(y)|<\epsilon$, $\forall n>N$
	as required.
\end{proof}

With this we are prepared to tackle the proof.

\begin{proof}[Proof of Lemma \ref{lem:conv.bdys.classic}]
	We treat the case for the upper boundary $B$ and the lower boundary is similar. Let $t\in[0,T)$ and $Q_t^n, Q_t$ be the inverses of $\partial_\pi V(t,\cdot;\mu_n)$ and $\partial_\pi V(t,\cdot;\mu)$ as per the above analysis. We define $Q_t^n(0):=\lim_{x\downarrow0}Q_t^n(x)$, $Q_t^n(\alpha_t):=\lim_{x\uparrow\alpha_t}Q_t^n(x)$ and $Q_t(0):=\lim_{x\downarrow0}Q_t(x)$, $Q_t(\alpha_t):=\lim_{x\uparrow\alpha_t}Q_t(x)$. We also note that $B(t)=Q_t(0), B_n(t)=Q_t^n(0)$. By continuity in the volatility (Lemma \ref{lem:cont.val.funct}), the above Lemmas give us that $Q_t^n(x)\to Q_t(x)$ for $x\in(0,\alpha_t)$. Since we have shown that $Q_t,Q^n_t$ are uniformly Lipschitz on $(0,\alpha_t)$, their extensions are also uniformly Lipschitz on $[0,\alpha_t]$. To see this simply take any sufficiently small $\epsilon_0,\epsilon_1>0$ and note that by the Lipschitz property on $(0,\alpha_t)$ we have for any $y\in(0,\alpha_t)$:
	\[|Q_t(\epsilon_0)-Q_t(y)|\leq \frac{\mathfrak{H}^2}{8c}|\epsilon_0-y|,\]
	\[|Q_t(\alpha_t-\epsilon_1)-Q_t(y)|\leq \frac{\mathfrak{H}^2}{8c}|\alpha_t-\epsilon_1-y|,\]
	\[|Q_t(\alpha_t-\epsilon_1)-Q_t(\epsilon_0)|\leq \frac{\mathfrak{H}^2}{8c}|\alpha_t-\epsilon_1-\epsilon_0|.\]
	Taking $\epsilon_0,\epsilon_1\downarrow0$ shows the claim. They are also uniformly bounded by definition. Then, by the Arz\'ela Ascoli Theorem, for any subsequence $(n')$ we can extract a further subsequence $(n'')$ for which $(Q_t^{n''})$ converges uniformly to some Lipschitz continuous limit function $Q^*_t$ on $[0,\alpha_t]$. But then, since $Q_t^{n''}\to Q_t$ on $(0,\alpha_t)$, by continuity $Q_t\equiv Q_t^*$. Since the subsequence was arbitrary we have the convergence of the functions on $[0,\alpha_t]$. So, $B_n(t)=Q_t^n(0)\to Q_t(0)=B(t)$ as $n\to\infty$. Since $t\in[0,T)$ was arbitrary we have pointwise convergence. Additionally, since all the boundaries are continuous and monotone we have that the convergence is, in fact, locally uniform.
\end{proof}

\section{Continuity of Hitting Times for Gaussian Processes}\label{app:conv.stopping.times}

This appendix is a self-contained analysis of the convergence of stopping times for Gaussian processes. We make use of this result in our proof of optimal stopping time convergence for our sequential testing problem (see Lemma \ref{lem:conv.stopping.times}). 

Consider a sequence of time dependent boundaries $(\Gamma_n(t))_{n\geq1}, (\gamma_n(t))_{n\geq1}$ on $[0,T]$ converging on $[0,T)$ to limiting boundaries $\Gamma_\infty(t)$ and $\gamma_\infty(t)$, respectively. Additionally, consider a sequence of Gaussian processes $(X^n)_{n\geq1}$ whose drift and diffusion coefficients converge pointwise to those of the Gaussian process $X^\infty$. Define the stopping times given by the first hitting time to each of the boundaries by:
\[\tau_\gamma^n:=\inf\{t\geq0:\gamma_n(t)-X^n_t\geq0\}\wedge T\]
\[\tau_\Gamma^n:=\inf\{t\geq0:X^n_t-\Gamma_n(t)\geq0\}\wedge T\]
for all $n$ (including $n=\infty$). With this we are ready to state the result.

\begin{lemma}\label{lem:conv.stop.time}
	Let $(X^n)_{n\geq1}$, and $X^\infty$ be Gaussian processes of the form:
	\[X^n_t=X^n_0+\int_0^t\mu_n(s)ds+\int_0^t\sigma_n(s)^\top d\mathbf{W}_s\]
	for a $d$-dimensional Brownian Motion $\mathbf{W}$ and continuous functions $\mu_n:\mathbb{R}_+\to \mathbb{R}$ , $\sigma_n:\mathbb{R}_+\to\mathbb{R}^d$ satisfying the bounds $\kappa\leq \|\sigma_n(\cdot)\|\leq K$, $|\mu_n(\cdot)|\leq K$, for some $\kappa, K>0$ and all $n$ (including $n=\infty$).
	Suppose
	\begin{enumerate}
		\item[(i)] $\gamma_n(t)\to \gamma_\infty(t)$ and $\Gamma_n(t)\to \Gamma_\infty(t)$ locally uniformly on $[0,T)$.
		\item[(ii)] The initial conditions and coefficients satisfy $X_0^n\to X_0^\infty$, $\mu_n(t)\to \mu_\infty(t)$, and  $\sigma_n(t)\to\sigma_\infty(t)$ for all $t\in[0,T)$.
		\item[(iii)] The boundaries $(\gamma_n(t))_{n\geq1}, \gamma_\infty(t)$ are increasing in time and continuous \textit{OR} uniformly locally Lipschitz on $[0,T)$.
		\item[(iv)] The boundaries $(\Gamma_n(t))_{n\geq1}, \Gamma_\infty(t)$ are decreasing in time and continuous \textit{OR} uniformly locally Lipschitz on $[0,T)$.
	\end{enumerate}
	Then we have the convergence in probability of the stopping times $\tau^n_\gamma\rightarrow \tau^\infty_\gamma$  and $\tau^n_\Gamma\rightarrow \tau^\infty_\Gamma$ as $n\to\infty$.
\end{lemma}

\begin{proof}
	The proof for both types of stopping times proceeds in a symmetric fashion. So, without loss of generality, we restrict our attention to $\tau^n_\gamma$, $\tau^\infty_\gamma$.
	We note that many of the initial arguments in the proof are adaptations of the arguments employed in \cite{shevchenko2015convergence}.
	
	We aim to show that for any small $\epsilon,\delta>0$, we can find an $N\geq0$ such that if $n\geq N$ then
	\[\mathbb{P}(|\tau_\gamma^n-\tau_\gamma^\infty|>\epsilon)<\delta\]
	from which the convergence in probability follows.
	
	Begin by fixing $\epsilon,\delta>0$. Let $\eta>0$ be a constant to be chosen later
	and define the stopping time
	\[\theta^n:=\inf\left\{t\geq0:|X^n_t-X^\infty_t|\geq \frac{\eta}{2}\right\}\wedge T.\]
	We have that:
	\begin{align*}
		\mathbb{P}(|\tau_\gamma^n-&\tau_\gamma^\infty|>\epsilon)\\
		&\leq \mathbb{P}\left(|\tau_\gamma^n-\tau_\gamma^\infty|>\epsilon,  \theta^n>T-\frac{\epsilon}{2}\right)+ \mathbb{P}\left(\theta^n\leq T-\frac{\epsilon}{2}\right)\\
		&\leq \mathbb{P}\left(|\tau_\gamma^n-\tau_\gamma^\infty|>\epsilon,  \theta^n>T-\frac{\epsilon}{2}\right)+ \mathbb{P}\left( \sup_{t\in[0,T]}|X^n_t-X^\infty_t|\geq \frac{\eta}{2}\right).
	\end{align*}
	Assumption (ii) in the theorem statement and standard stochastic differential equation estimates show that $\mathbb{E}\left[\sup_{t\in[0,T]}|X_t^n-X_t^\infty|^2\right]$ converges to 0. 
	Hence, we can find an $N_1\geq0$ such that for all $n\geq N_1$ we have that
	\[\mathbb{P}\left( \sup_{t\in[0,T]}|X^n_t-X^\infty_t|\geq \eta\right)<\frac{\delta}{2}.\]
	As a consequence for $n\geq N_1$ we have by the above that:
	\begin{equation}\label{eqn:first.main.est.Gaussian}
		\mathbb{P}(|\tau_\gamma^n-\tau_\gamma^\infty|>\epsilon)<\mathbb{P}\left(|\tau_\gamma^n-\tau_\gamma^\infty|>\epsilon,  \theta^n>T-\frac{\epsilon}{2}\right)+ \frac{\delta}{2}.
	\end{equation}
	We now proceed to further bound the first term on the right hand side above. We have trivially that
	\begin{align}\label{eqn:inter.bds.Gaussian}
		\mathbb{P}&\left(|\tau_\gamma^n-\tau_\gamma^\infty|>\epsilon,  \theta^n>T-\frac{\epsilon}{2}\right)\\\nonumber
		&\quad \quad \leq \mathbb{P}\left(\tau_\gamma^\infty+\epsilon<\tau_\gamma^n,  \theta^n>T-\frac{\epsilon}{2}\right) + \mathbb{P}\left(\tau_\gamma^n+\epsilon<\tau_\gamma^\infty,  \theta^n>T-\frac{\epsilon}{2}\right).
	\end{align}
	We will now treat each of these terms on the right hand side starting with \[\mathbb{P}\left(\tau_\gamma^\infty+\epsilon<\tau_\gamma^n,  \theta^n>T-\frac{\epsilon}{2}\right).\]
	First note that
	\[\mathbb{P}\left(\tau_\gamma^\infty+\epsilon<\tau_\gamma^n,  \theta^n>T-\frac{\epsilon}{2}\right)\leq \mathbb{P}\left(\tau_\gamma^\infty<T-\epsilon,  \tau_\gamma^\infty+\frac{\epsilon}{2}<\tau_\gamma^n,\theta^n>T-\frac{\epsilon}{2}\right).\]
	Now if the three conditions: 
	\begin{equation}\label{eqn:three.cond}
		\tau_\gamma^\infty<T-\epsilon, \ \ \  \tau_\gamma^\infty+\frac{\epsilon}{2}<\tau_\gamma^n, \ \ \  \mathrm{and} \ \ \  \theta^n>T-\frac{\epsilon}{2}
	\end{equation} hold we have on $\left[\tau^\infty_\gamma,\tau^\infty_\gamma+\frac{\epsilon}{2}\right]$ that
	$|X^n_t-X^\infty_t|\leq \frac{\eta}{2}.$
	Moreover, for both $(\Gamma_n)_{n\geq 1}$ and $(\gamma_n)_{n\geq 1}$ by our assumptions we can find an $N_2\geq 0$ such that if $n\geq N_2$ then
	\[\sup_{t\in\left[0,T-\frac{\epsilon}{2}\right]}|\gamma_n(t)-\gamma_\infty(t)|<\frac{\eta}{2}.\] 
	As a consequence for $n\geq N_1\vee N_2$ we have
	\[\sup_{t\in\left[\tau_\gamma^\infty,\tau_\gamma^\infty+\frac{\epsilon}{2}\right]}|\gamma_n(t)-\gamma_\infty(t)|<\frac{\eta}{2}.\] 
	For the case where the boundaries are uniformly locally Lipschitz we also have that there is a local Lipschitz constant $L_{\epsilon/2}>0$ (holding on $[0,T-\frac{\epsilon}{2}]$) such that for all $n$ (including $n=\infty$):
	\[|\gamma_n(t)-\gamma_n(s)|\leq L_{\epsilon/2}|t-s|,  \ \ \ \forall (t,s)\in\left[0,T-\frac{\epsilon}{2}\right],\] 
	and in particular, (as $\tau^\infty_\gamma+\frac{\epsilon}{2}<T-\frac{\epsilon}{2}$) this holds for all $ (t,s)\in\left[\tau_{\gamma}^\infty,\tau_{\gamma}^\infty+\frac{\epsilon}{2}\right]\subset \left[0,T-\frac{\epsilon}{2}\right]$.
	
	\medskip
	
	\noindent\textit{\textbf{Case 1 (Monotone Boundaries):}} Let us start with the case where the $\gamma_n$ are increasing in $t$. Since under our three conditions in \eqref{eqn:three.cond}, $\tau^n_\gamma\in (\tau_\gamma^\infty+\frac{\epsilon}{2},T]$ we have (by the path continuity of the Gaussian process with continuous coefficients and the definition of $\tau^n_\gamma$):
	\[\max_{t\in \left[\tau^\infty_\gamma,\tau^\infty_\gamma+\frac{\epsilon}{2}\right]}(\gamma_n(t)-X^n_t)<0.\]
	Hence,
	\begin{align*}
		0&>\max_{t\in \left[\tau^\infty_\gamma,\tau^\infty_\gamma+\frac{\epsilon}{2}\right]}(\gamma_n(t)-X^n_t)\\
		&=\max_{t\in \left[\tau^\infty_\gamma,\tau^\infty_\gamma+\frac{\epsilon}{2}\right]}(\gamma_\infty(t)-X^\infty_t+\gamma_n(t)-\gamma_\infty(t)+X^\infty_t-X^n_t)\\
		&\geq \max_{t\in \left[\tau^\infty_\gamma,\tau^\infty_\gamma+\frac{\epsilon}{2}\right]}(\gamma_\infty(t)-X^\infty_t)-\eta\\
		&= \max_{t\in \left[\tau^\infty_\gamma,\tau^\infty_\gamma+\frac{\epsilon}{2}\right]}\left(\gamma_\infty(t)-X^\infty_{\tau^\infty_\gamma}-\int_{\tau_\gamma^\infty}^t\mu_\infty(s)ds-\int_{\tau_\gamma^\infty}^t\sigma_\infty(s)^\top d\mathbf{W}_s\right)-\eta\\
		&\geq \max_{t\in \left[\tau^\infty_\gamma,\tau^\infty_\gamma+\frac{\epsilon}{2}\right]}\left(\gamma_\infty(t)-X^\infty_{\tau^\infty_\gamma}-\int_{\tau_\gamma^\infty}^t\frac{K}{\kappa^2}\|\sigma_\infty(s)\|^2ds-\int_{\tau_\gamma^\infty}^t\sigma_\infty(s)^\top d\mathbf{W}_s\right)-\eta\\
		&=\max_{t\in \left[\tau^\infty_\gamma,\tau^\infty_\gamma+\frac{\epsilon}{2}\right]}\left(\gamma_\infty(t)-\gamma_\infty(\tau^\infty_\gamma)-\int_{\tau_\gamma^\infty}^t\frac{K}{\kappa^2}\|\sigma_\infty(s)\|^2ds-\int_{\tau_\gamma^\infty}^t\sigma_\infty(s)^\top d\mathbf{W}_s\right)-\eta\\
		&\geq \max_{t\in \left[\tau^\infty_\gamma,\tau^\infty_\gamma+\frac{\epsilon}{2}\right]}\left(-\int_{\tau_\gamma^\infty}^t\frac{K}{\kappa^2}\|\sigma_\infty(s)\|^2ds-\int_{\tau_\gamma^\infty}^t\sigma_\infty(s)^\top d\mathbf{W}_s\right)-\eta
	\end{align*}
	where in the second last line we have used path continuity to get $X^\infty_{\tau_\gamma^\infty}=\gamma_\infty(\tau_\gamma^\infty)$. The last line then follows by our assumption that says $\gamma_\infty(t)$ is increasing.
	Taken together, by rearranging we get that if $\tau_\gamma^\infty<T-\epsilon,  \tau_\gamma^\infty+\frac{\epsilon}{2}<\tau_\gamma^n$ and $\theta^n>T-\frac{\epsilon}{2}$ then for $n\geq N_1\vee N_2$:
	\[\max_{t\in \left[\tau^\infty_\gamma,\tau^\infty_\gamma+\frac{\epsilon}{2}\right]}\left(-\frac{K}{\kappa^2}\int_{\tau_\gamma^\infty}^t\|\sigma_\infty(s)\|^2ds-\int_{\tau_\gamma^\infty}^t\sigma_\infty(s)^\top d\mathbf{W}_s\right)<\eta. \] 
	Hence we arrive at the bound:
	\begin{align*}
		\mathbb{P}&\left(\tau_\gamma^\infty<T-\epsilon,  \tau_\gamma^\infty+\frac{\epsilon}{2}<\tau_\gamma^n,\theta^n>T-\frac{\epsilon}{2}\right) \\
		&\leq  \mathbb{P}\left(\max_{t\in \left[\tau^\infty_\gamma,\tau^\infty_\gamma+\frac{\epsilon}{2}\right]}\left(-\frac{K}{\kappa^2}\int_{\tau_\gamma^\infty}^t\|\sigma_\infty(s)\|^2ds-\int_{\tau_\gamma^\infty}^t\sigma_\infty(s)^\top d\mathbf{W}_s\right)<\eta, \tau_\gamma^\infty<T-\epsilon\right).
	\end{align*}
	
	\medskip
	
	\noindent\textit{\textbf{Case 2 (Uniformly Locally Lipschitz Boundaries):}} When instead the uniform local Lipschitz condition holds for the $\gamma_n$ we repeat the arguments to get
	
	\begin{align*}
		0&>\max_{t\in \left[\tau^\infty_\gamma,\tau^\infty_\gamma+\frac{\epsilon}{2}\right]}\left(\gamma_\infty(t)-\gamma_\infty(\tau^\infty_\gamma)-\int_{\tau_\gamma^\infty}^t\mu_\infty(s)ds-\int_{\tau_\gamma^\infty}^t\sigma_\infty(s)^\top d\mathbf{W}_s\right)-\eta\\
		&\geq \max_{t\in \left[\tau^\infty_\gamma,\tau^\infty_\gamma+\frac{\epsilon}{2}\right]}\left(-L_{\epsilon/2}(t-\tau_\gamma^\infty)-\int_{\tau_\gamma^\infty}^t\mu_\infty(s)ds-\int_{\tau_\gamma^\infty}^t\sigma_\infty(s)^\top d\mathbf{W}_s\right)-\eta\\
		&= \max_{t\in \left[\tau^\infty_\gamma,\tau^\infty_\gamma+\frac{\epsilon}{2}\right]}\left(-\int_{\tau_\gamma^\infty}^t(L_{\epsilon/2}+\mu_\infty(s))ds-\int_{\tau_\gamma^\infty}^t\sigma_\infty(s)^\top d\mathbf{W}_s\right)-\eta\\
		&= \max_{t\in \left[\tau^\infty_\gamma,\tau^\infty_\gamma+\frac{\epsilon}{2}\right]}\left(-\int_{\tau_\gamma^\infty}^t\frac{(L_{\epsilon/2}+\mu_\infty(s))}{\|\sigma_\infty(s)\|^2}\|\sigma_\infty(s)\|^2ds-\int_{\tau_\gamma^\infty}^t\sigma_\infty(s)^\top d\mathbf{W}_s\right)-\eta\\
		&\geq \max_{t\in \left[\tau^\infty_\gamma,\tau^\infty_\gamma+\frac{\epsilon}{2}\right]}\left(-\int_{\tau_\gamma^\infty}^t\frac{(L_{\epsilon/2}+K)}{\kappa^2}\|\sigma_\infty(s)\|^2ds-\int_{\tau_\gamma^\infty}^t\sigma_\infty(s)^\top d\mathbf{W}_s\right)-\eta
	\end{align*}
	Once again by rearranging we get that if $\tau_\gamma^\infty<T-\epsilon,  \tau_\gamma^\infty+\frac{\epsilon}{2}<\tau_\gamma^n$ and $\theta^n>T-\frac{\epsilon}{2}$ then for $n\geq N_1\vee N_2$:
	\[\max_{t\in \left[\tau^\infty_\gamma,\tau^\infty_\gamma+\frac{\epsilon}{2}\right]}\left(-K_\epsilon\int_{\tau_\gamma^\infty}^t\|\sigma_\infty(s)\|^2ds-\int_{\tau_\gamma^\infty}^t\sigma_\infty(s)^\top d\mathbf{W}_s\right)<\eta,\] 
	for some $K_\epsilon>0$. Hence we arrive at the bound:
	\begin{align*}
		&\mathbb{P}\left(\tau_\gamma^\infty<T-\epsilon,  \tau_\gamma^\infty+\frac{\epsilon}{2}<\tau_\gamma^n,\theta^n>T-\frac{\epsilon}{2}\right) \\
		&\leq\mathbb{P}\left(\max_{t\in \left[\tau^\infty_\gamma,\tau^\infty_\gamma+\frac{\epsilon}{2}\right]}\left(-K_\epsilon\int_{\tau_\gamma^\infty}^t\|\sigma_\infty(s)\|^2ds-\int_{\tau_\gamma^\infty}^t\sigma_\infty(s)^\top d\mathbf{W}_s\right)<\eta, \tau_\gamma^\infty<T-\epsilon\right).
	\end{align*}
	Note that the monotone case leads us to the choice $K_\epsilon=\frac{K}{\kappa^2}$ so rather than treat these cases separately, we will proceed using this full generality without further comment. 
	
	For the remaining term in \eqref{eqn:inter.bds.Gaussian}, a nearly identical argument gives us:
	\[\mathbb{P}\left(\tau_\gamma^n+\epsilon<\tau_\gamma^\infty,  \theta^n>T-\frac{\epsilon}{2}\right)\leq \mathbb{P}\left(\tau_\gamma^n<T-\epsilon,  \tau_\gamma^n+\frac{\epsilon}{2}<\tau_\gamma^\infty,\theta^n>T-\frac{\epsilon}{2}\right)\]
	and by exploiting the uniform local Lipschitz property (or monotonicity) that:
	\begin{align*}
		\mathbb{P}&\left(\tau_\gamma^n<T-\epsilon,  \tau_\gamma^n+\frac{\epsilon}{2}<\tau_\gamma^\infty,\theta^n>T-\frac{\epsilon}{2}\right)\\
		&\leq \mathbb{P}\left(\max_{t\in \left[\tau^n_\gamma,\tau^n_\gamma+\frac{\epsilon}{2}\right]}\left(-K_\epsilon\int_{\tau_\gamma^n}^t\|\sigma_n(s)\|^2ds-\int_{\tau_\gamma^n}^t\sigma_n(s)^\top d\mathbf{W}_s\right)<\eta, \tau_\gamma^n<T-\epsilon\right)
	\end{align*}
	Now, if $\eta$ can be chosen sufficiently small so that for all $n$ the above quantities are each bounded by $\frac{\delta}{4}$ then we get from \eqref{eqn:inter.bds.Gaussian} that
	\[\mathbb{P}\left(|\tau_\gamma^n-\tau^\infty_\gamma|>\epsilon, \theta^n>T-\frac{\epsilon}{2}\right)< \frac{\delta}{2}.\]
	Taking this bound together with \eqref{eqn:first.main.est.Gaussian} we will then arrive at
	\[\mathbb{P}(|\tau_\gamma^n-\tau_\gamma^\infty|>\epsilon)<\delta\]
	as required to complete the proof.
	
	This final piece of the proof will follow from the following claim which gives the result for all $n$ by our assumptions on $\sigma_n(s)$.
	
	\textit{\textbf{Claim:} For any $\delta>0$ we can find an $\eta>0$ sufficiently small so that for all $\mathbb{F}^{\mathbf{W}}$-stopping times $\tau$ and $\|\sigma(s)\|\geq\kappa>0$:}
	\[\mathbb{P}\left(\max_{t\in\left[\tau,\tau+\frac{\epsilon}{2}\right]}\left(-K_\epsilon\int_\tau^t\|\sigma(s)\|^2ds+\int_\tau^t\sigma(s)^\top d\mathbf{W}_s\right)<\eta, \tau< T-\epsilon\right)<\frac{\delta}{4}.\]
	
	Note first by the strong Markov property that the left hand side is trivially bounded by:
	\[\max_{t_0\in[0,T-\epsilon]}\mathbb{P}\left(\max_{t\in\left[t_0,t_0+\frac{\epsilon}{2}\right]}\left(-K_\epsilon\int_{t_0}^t\|\sigma(s)\|^2ds+\int_{t_0}^t\sigma(s)^\top d\mathbf{W}_s\right)<\eta\right)\]
	since the distribution of the Gaussian process under the maximum depends on $\tau$ only through the random starting point for the interval and the outer maximum will choose the worst case starting time. Now, for a given $t_0$ we have the equivalent representation:
	\[\mathbb{P}\left(\max_{u\in\left[0,\frac{\epsilon}{2}\right]}\left(-K_\epsilon\int_0^{u}\|\sigma(s+t_0)\|^2ds+\int_0^{u}\sigma(s+t_0)^\top d\mathbf{W}_s\right)<\eta\right)\]
	Now by traditional arguments we have that there exists a Brownian motion $\widehat{W}$ such that for the ``clock"
	\[\alpha_0(u)=\int_0^u\|\sigma(s+t_0)\|^2ds\]
	we have 
	\[\int_0^{u}\sigma(s+t_0)^\top d\mathbf{W}_s\overset{d}{=}\widehat{W}_{\alpha_0(u)}.\]
	As a result we can treat
	\[\mathbb{P}\left(\max_{u\in\left[0,\frac{\epsilon}{2}\right]}\left(-K_\epsilon\alpha_0(u)+\widehat{W}_{\alpha_0(u)}\right)<\eta\right).\]
	Equivalently:
	\[\mathbb{P}\left(\max_{s\in\left[0,\alpha_0(\frac{\epsilon}{2})\right]}\left(-K_\epsilon s+\widehat{W}_{s}\right)<\eta\right)\]
	since $\alpha_0(0)=0$ and $\alpha(u)$ is increasing and continuous on $[0,\epsilon/2]$. Now since $\|\sigma(s+t_0)\|>\kappa$ we have $\alpha(\epsilon/2)>\kappa\frac{\epsilon}{2}$ and
	\[\mathbb{P}\left(\max_{s\in\left[0,\alpha_0(\frac{\epsilon}{2})\right]}\left(-K_\epsilon s+\widehat{W}_{s}\right)<\eta\right)\leq \mathbb{P}\left(\max_{s\in\left[0,\kappa\frac{\epsilon}{2}\right]}\left(-K_\epsilon s+\widehat{W}_{s}\right)<\eta\right)\]
	since if the maximum over the larger interval is less than $\eta$ then so is the maximum over the smaller interval. But now this is just in terms of the probability of the maximum of a Brownian motion with drift. In particular, we have by standard results (see for example \cite{shreve2004stochastic}) that:
	
	\[\mathbb{P}\left(\max_{s\in\left[0,\kappa\frac{\epsilon}{2}\right]}\left(-K_\epsilon s+\widehat{W}_{s}\right)<\eta\right)=\Phi\left(\frac{\eta+\kappa K_{\epsilon}\epsilon/2}{\sqrt{\kappa\epsilon/2}}\right)-e^{-2K_\epsilon\eta}\Phi\left(\frac{-\eta+\kappa K_\epsilon\epsilon/2}{\sqrt{\kappa\epsilon/2}}\right)\]
	where $\Phi(\cdot)$ is the standard normal CDF. Clearly as $\eta\downarrow0$ we have that the right hand side converges to $0$. Thus, for any $\delta>0$ we can find a sufficiently small $\eta>0$ such that 
	\[\Phi\left(\frac{\eta+\kappa K_{\epsilon}\epsilon/2}{\sqrt{\kappa\epsilon/2}}\right)-e^{-2K_\epsilon\eta}\Phi\left(\frac{-\eta+\kappa K_\epsilon\epsilon/2}{\sqrt{\kappa\epsilon/2}}\right)<\frac{\delta}{4}\]
	Since this argument was independent of $t_0\in[0,T-\epsilon]$ and the particular choice of $\tau$, this proves the claim. 
	
	We now close by remarking that choosing $\eta$ at the beginning of our proof to depend on $\epsilon,\delta$ as per the above and taking $n\geq N_1\vee N_2$ gives us the proof of the lemma.
\end{proof}


\end{appendix}


\bibliographystyle{abbrv}
\bibliography{seqTest.bib}

\end{document}